\documentclass[11pt]{article}

\usepackage[dvips]{graphicx}
\usepackage{calc}
\usepackage{color}
\usepackage{amsmath}
\usepackage{amssymb}
\usepackage{amscd}
\usepackage{amsthm}
\usepackage{amsbsy}
\usepackage{delarray}
\usepackage{enumerate}
\usepackage[T1]{fontenc}
\usepackage{inputenc}
\usepackage{enumerate}
\usepackage[all]{xy}
\usepackage{hyperref}

\begin{document}



\setlength{\parindent}{5mm}
\renewcommand{\leq}{\leqslant}
\renewcommand{\geq}{\geqslant}

\newcommand{\N}{\mathbb{N}}
\newcommand{\sph}{\mathbb{S}}
\newcommand{\Z}{\mathbb{Z}}
\newcommand{\R}{\mathbb{R}}
\newcommand{\D}{\mathbb{D}}
\newcommand{\C}{\mathbb{C}}
\newcommand{\F}{\mathbb{F}}
\newcommand{\g}{\mathfrak{g}}
\newcommand{\h}{\mathfrak{h}}
\newcommand{\K}{\mathbb{K}}
\renewcommand{\S}{\mathbb{S}}
\renewcommand{\H}{\mathbb{H}}

\newcommand{\cA}{\mathcal{A}}
\newcommand{\cN}{\mathcal{N}}
\newcommand{\cD}{\mathcal{D}}
\newcommand{\cO}{\mathcal{O}}
\newcommand{\cF}{\mathcal{F}}

\newcommand{\RN}{\mathbb{R}^{2n}}

\newcommand{\eps}{\varepsilon}
\newcommand{\chzrel}{c_{\mathrm{LR}}}
\newcommand{\ci}{c^{\infty}}
\newcommand{\derive}[2]{\frac{\partial{#1}}{\partial{#2}}}
\newcommand{\CZ}{\mu_{CZ}}

\theoremstyle{plain}
\newtheorem{theo}{Theorem}
\newtheorem{prop}[theo]{Proposition}
\newtheorem{lemma}[theo]{Lemma}
\newtheorem{definition}[theo]{Definition}
\newtheorem*{notation*}{Notation}
\newtheorem*{notations*}{Notations}
\newtheorem{corol}[theo]{Corollary}
\newtheorem{conj}[theo]{Conjecture}
\newtheorem{claim}[theo]{Claim}

\newenvironment{demo}[1][]{\addvspace{8mm} \emph{Proof #1.
    ~~}}{~~~$\Box$\bigskip}

\newlength{\espaceavantspecialthm}
\newlength{\espaceapresspecialthm}
\setlength{\espaceavantspecialthm}{\topsep} \setlength{\espaceapresspecialthm}{\topsep}

\newenvironment{example}[1][]{ \refstepcounter{theo} 
\vskip \espaceavantspecialthm \noindent \textsc{Example~\thetheo
#1.} }
{\vskip \espaceapresspecialthm}

\newenvironment{question}[1][]{
\vskip \espaceavantspecialthm \noindent \textsc{Question
#1.} }%
{\vskip \espaceapresspecialthm}

\newenvironment{remark}[1][]{\refstepcounter{theo} 
\vskip \espaceavantspecialthm \noindent \textsc{Remark~\thetheo
#1.} }%
{\vskip \espaceapresspecialthm}

\newenvironment{defi}[1][]{\refstepcounter{theo} 
\vskip \espaceavantspecialthm \noindent \textsc{Definition~\thetheo
#1.} }%
{\vskip \espaceapresspecialthm}

\newenvironment{notation}[1][]{ \refstepcounter{theo} 
\vskip \espaceavantspecialthm \noindent \textsc{Notation~\thetheo
#1.} }
{\vskip \espaceapresspecialthm}

\def\bb#1{\mathbb{#1}} \def\m#1{\mathcal{#1}}

\def\del{\partial}
\def\co{\colon\thinspace}
\def\Homeo{\mathrm{Homeo}}
\def\Hameo{\mathrm{Hameo}}
\def\Diffeo{\mathrm{Diffeo}}
\def\Symp{\mathrm{Symp}}
\def\Sympeo{\mathrm{Sympeo}}
\def\id{\mathrm{Id}}
\newcommand{\norm}[1]{||#1||}
\def\Ham{\mathrm{Ham}}
\def\lagham#1{\mathcal{L}^\mathrm{Ham}({#1})}
\def\Hamtilde{\widetilde{\mathrm{Ham}}}
\def\cOlag#1{\mathrm{Sympeo}({#1})}
\def\Crit{\mathrm{Crit}}
\def\Spec{\mathrm{Spec}}
\def\osc{\mathrm{osc}}
\def\Cal{\mathrm{Cal}}

\def\Ham{\mathrm{Ham}}
\def\Fix{\mathrm{Fix}}
\def\Crit{\mathrm{Crit}}
\def\Per{\mathrm{Per}}
\def\spec{\mathrm{spec}}
\def\supp{\mathrm{supp}}

\definecolor{sobhan}{rgb}{0,.6,0}

\title{
Towards a dynamical interpretation of Hamiltonian spectral invariants on surfaces}
\author{Vincent Humili\`ere, Fr\'ed\'eric Le Roux, Sobhan Seyfaddini}
\date{\today}





\maketitle

\begin{abstract}
Inspired by Le Calvez' theory of transverse foliations for dynamical systems of surfaces \cite{PLC1, lecalvez2005}, we introduce a dynamical invariant, denoted by $\m N$, for Hamiltonians of any surface other than the sphere.  When the surface is the plane or is closed and aspherical, we prove that on the set of autonomous Hamiltonians  this invariant coincides with the spectral invariants constructed by Viterbo on the plane and Schwarz on closed and aspherical surfaces.
  
  Along the way, we obtain several results of independent interest: We show that a \emph{formal} spectral invariant, satisfying a minimal set of axioms, must coincide with $\m N$ on autonomous Hamiltonians thus establishing a certain uniqueness result for spectral invariants, we obtain a ``Max Formula'' for spectral invariants on aspherical manifolds, give a very simple description of the Entov-Polterovich quasi-state on aspherical surfaces  and characterize the heavy and super-heavy subsets of such surfaces.
\end{abstract}


\tableofcontents

\section{Introduction}
Let $(M, \omega)$ denote an aspherical symplectic manifold. Recall that being aspherical means $\omega|_{\pi_2} = c_1|_{\pi_2} =0$, where $c_1$ is the first Chern class of $M$.  We allow $M$ to be either the Euclidean space $\R^{2n}$ with its standard symplectic structure or a closed and connected symplectic manifold.  As a consequence of the theory of spectral invariants, one can associate to every smooth Hamiltonian $H$ a real number $c(H)$ referred to as the spectral invariant of $H$.  This number is, roughly speaking, the action level at which the fundamental class $[M]$ appears in the Floer homology of the Hamiltonian $H$.\footnote{Similarly, one can associate spectral invariants to other homology classes of $M$ as well.  The focus of this article is on the invariant associated to the fundamental class.} These invariants were introduced by Viterbo \cite{viterbo} for $M= \R^{2n}$ using generating function theory and by Schwarz \cite{schwarz} for closed aspherical symplectic manifolds using Hamiltonian Floer theory.\footnote{In \cite{Oh05b}, Oh extended Schwarz's work to arbitrary closed symplectic manifolds.  See the papers \cite{FrSc} and \cite{lanzat} for extensions to other types of symplectic manifolds.}  Spectral invariants have had many important and interesting applications in symplectic topology and dynamical systems; see for example \cite{entov-polterovich03, entov-polterovich06, ginzburg}.  A recently discovered application which has largely motivated this article is a simple solution to the displaced disks problem of B\'eguin, Crovisier and Le Roux: using the spectral invariant $c$ one can show that arbitrarily $C^0$-small area preserving homeomorphisms of a closed surface can not displace disks of a given area; see \cite{sey, dore-hanlon}.

One drawback of the spectral invariant $c$ is the complexity of its construction which relies on the difficult machinery of Floer theory.  As a consequence, despite its widespread use, $c$ can only be computed in a handful of scenarios where the Floer theoretic picture is simple enough.

Motivated by the resolution of the displaced disks problem, we introduce a new invariant $\m N$ on aspherical surfaces which, like $c,$ associates a real number to every Hamiltonian.  The construction of $\m N$ is purely dynamical and is far more elementary than that of $c$.  We then prove that $\m N$ and $c$ coincide on autonomous Hamiltonians. 

An intriguing aspect of this work is that, beyond spectrality, the obvious properties of $\m N$ are quite different from the known properties of $c$. Indeed, $\m N$ is computable in practice for autonomous Hamiltonians.  Furthermore, one can easily see that it satisfies a certain maximum formula which was not known for $c$. On the other hand, $\m N$ does not \emph{a priori} seem to share the continuity properties of $c$ (see Definition~\ref{def:formal_spec} below). 
Proving that $c$ and $\m N$ coincide consists of two main components which are perhaps of their own independent interest: First, we prove that $c$ satisfies the same max formula as $\m N$.  Second, we show that a ``formal'' spectral invariant satisfying a minimal set of axioms must coincide with $\m N$ on autonomous Hamiltonians. This establishes a certain uniqueness result for spectral invariants which would be interesting to pursue in more general settings. See Theorem \ref{theo.axiomatic-c=n}.

As a by product of our work, we obtain a very simple description of the Entov-Polterovich (partial) quasi-state on closed aspherical surfaces using which we characterize heavy and super-heavy subsets of these surfaces.

An inspirational factor in writing this article has been our hopes of better understanding the link between Hamiltonian Floer theory and Le Calvez's theory of transverse foliations for dynamical systems on surfaces \cite{PLC1, lecalvez2005}.  In a sense, as far as surfaces are concerned, the two theories appear to be equivalent: much of what can be done via one theory can also be achieved via the other.  As examples of this phenomenon, one could point to proofs of the Arnol'd conjecture and recent articles by Bramham \cite{bramham2, bramham1} and Le Calvez \cite{PLC3}.  A prominent missing link from this hypothetical equivalence is the spectral invariant $c$ which to this date has had no analogue in Le Calvez's theory.  The introduction of $\m N$ in this article is an attempt to recover spectral invariants, and the solution to the displaced disks problem, via the techniques of transverse foliations.   Of course, whether $\m N$ coincides with $c$ on all Hamiltonians, and not just the autonomous ones, is a glaring open question which we hope to answer in the future.

 \paragraph{Aknowledgments} This work began after the crucial insight by Patrice Le Calvez, following a talk by the third author on the solution of the displaced disks problem, that the spectral norm $\gamma$ could be equal to  the quantity
 $$
 \inf \left\{ \left( \sup_{x \in X} \cA_H(x) - \inf_{x \in X} \cA_H(x) \right) , X  \mbox{ maximal unlinked set for } \phi_{H}^1\right\}
   $$
(see below for the definitions). This formula is still a conjecture. In addition to this seminal proposal, Patrice's theory of equivariant Brouwer foliations is both a powerful tool and an exciting motivation to understand the link between unlinked sets and spectral invariants. We also owe him the suggestion that the existence of maximal unlinked sets for diffeomorphisms could be proved using Handel's lemma. We warmly thank him for all this!

We would like to thank Sylvain Crovisier, Michael Entov, R\'emi Leclercq, Alex Oancea, Leonid Polterovich, Claude Viterbo and  Frol Zapolsky for helpful comments and conversations.

SS: Most of the research leading to this article was carried out while I was a postdoctoral member of the  D\'epartement de Math\'ematiques et Applications at Ecole Normale Sup\'erieure in Paris.  I wish to express my deepest gratitude to members of DMA for their warm hospitality during my two year stay there.

  The first and second author were  partially supported by the ANR Grant ANR-11-JS01-010-01.  The third author was partially supported by the NSF Postdoctoral Fellowship Grant No. DMS-1401569 and the European Research Council under the European Union's Seventh Framework Programme (FP/2007-2013) / ERC Grant Agreement  307062.

\subsection{The invariant $\m N$}
We work with a symplectic surface $\Sigma$ which is either the plane $\R^2$ or a closed surface other than the sphere.  We denote the set of compactly supported Hamiltonians on $\Sigma$ by $C^{\infty}([0,1] \times \Sigma)$, and the set of compactly supported autonomous Hamiltonians by $C^{\infty}(\Sigma)$.
Let $H \in C^{\infty}([0,1] \times \Sigma)$. Our sign convention is that the Hamiltonian vector field $X_H$ is induced by $H$ via $\omega(X_H^t, \cdot\,)=-dH_t$ for all $t$. Integrating the time-dependent vector field $X_{H}^t$ yields a Hamiltonian isotopy $(\phi_{H}^{t})_{t \in [0,1]}$.

The main ingredient in the definition of $\m N$ is the notion of \emph{unlinked sets}. 
We consider fixed points of $\phi_{H}^1$ whose trajectories under the Hamiltonian isotopy are contractible in $\Sigma$; these are referred to as contractible fixed points.
An \emph{unlinked set} is a set $X$ of contractible fixed points of $\phi_{H}^1$ for which there exists an isotopy $(f_{t})_{t\in [0,1]}$, $f_{0}= \mathrm{Id}$, $f_{1} = \phi_{H}^1$, such that every point of $X$ is fixed by every $f_{t}$.
Unlinked sets play a crucial role in Le Calvez's theory of transverse foliations for dynamical systems on surfaces. We will say an unlinked set $X$ is  \emph{negative} if for every $x$ in $X$, the direction of every tangent vector at $x$ is either fixed or turns in the negative direction by the isotopy $(f_{t})_{t\in [0,1]}$. We denote by $\mathrm{mnus}(H)$ the family of negative unlinked sets that are maximal for the inclusion among negative unlinked sets. Finally, the invariant $\m N$ is defined by the formula
$$
\cN(H) = \inf_{X \in \mathrm{mnus}(H)} \sup_{x \in X} \cA_H(x)
$$
where $\cA_{H}(x)$ denotes the symplectic action  (see Section~\ref{sec:preliminary-N} for details).

An interesting aspect of the invariant $\m N$ is that it is defined directly for all smooth Hamiltonians while spectral invariants are first constructed  for non-degenerate Hamiltonians and then extended to all Hamiltonians by a limiting process. Regarding computational issues, note that for a generic Hamiltonian function $H$, the map $\phi_{H}^1$ has a finite number of fixed points. For a finite set $X$ of contractible fixed points, the unlinkedness is equivalent to the triviality of the braid $(\phi_{H}^t X)_{t \in [0,1]}$ (see Section~\ref{sec:preliminary-N}). Then the value of $\cN$ depends only on the total braid associated to the set of all contractible fixed points, colored with the value of the action at each contractible fixed point. In particular, the types of braids generated by autonomous Hamiltonian functions are very constrained,  and we provide a recursive formula that makes explicit computations easy (see Propositions~\ref{prop:recurs-form-disk} and~\ref{prop:form_N_general} below).

In the following theorem the function $c: C^{\infty}([0,1] \times \Sigma) \rightarrow \R$ denotes either the spectral invariant defined via generating function theory (when $\Sigma= \R^2$) or the spectral invariant defined via Hamiltonian Floer theory; these spectral invariants are defined in Sections \ref{sec:def_c_R2n} and \ref{sec:def_c_aspherical}, respectively. Here is the main result of this article.
\begin{theo}\label{theo:main}
$c(H) = \m N (H)$ for every autonomous $H \in C^{\infty}(\Sigma).$
\end{theo}

This theorem immediately gives rise to the following question:
\begin{question}
Is it true that $c(H) = \m N(H)$ for all $H \in C^{\infty}([0,1] \times \Sigma)$?
\end{question}

\subsection{Max Formulas and formal spectral invariants}\label{sec:max-formal}
We now outline the two main components of the proof of Theorem \ref{theo:main}.   

\medskip

\noindent \textbf{Max Formula:}  In Section \ref{sec:max_form}, we prove max formulas for the spectral invariant $c$  which hold on higher dimensional symplectic manifolds as well as surfaces.  Here, we will state a simplified version of these max formulas and refer to Theorems \ref{theo:max-formula-R2n} and \ref{theo:max-formula-aspherical} in Section \ref{sec:max_form_aspherical} for the more general statements.

Below,  the function $c: C^{\infty}([0,1] \times M) \rightarrow \R$ denotes either the spectral invariant  defined via generating function theory, when $M= \R^{2n}$, or the one defined via Hamiltonian Floer theory on closed aspherical manifolds.  We denote by $U_1, \ldots, U_N$ disjoint open subsets of $M$ each of which is symplectomorphic to a Euclidean ball.  
 \begin{theo}\label{theo:max-formula-balls}
   Suppose that $H_1,\ldots, H_N$ are  Hamiltonians whose supports are contained, respectively, in the symplectic balls $U_1 ,\ldots, U_N$. Then,
 $$c(H_1 +\ldots + H_N) =  \max\{c(H_1), \ldots, c(H_N)\}.$$
  \end{theo}
  
 The assumption that $M$ is aspherical is crucial.  In Section \ref{sec:countr_exmpl}, we give a counter example to the above max formula on the $2$--sphere.  

 The above theorem and its more general version, Theorem \ref{theo:max-formula-aspherical}, relate to questions which arise from the recent work of Polterovich on Poisson bracket invariants of coverings \cite{polterovich}; see also Question 1 in \cite{sey_killer}.  We will not delve into this topic as it goes beyond the intended scope of this article.
 
 \medskip

\noindent \textbf{Formal Spectral Invariants:}  Although the following definition makes sense on any symplectic manifold we will restrict our attention here to the case of a surface $\Sigma$ which is either the plane $\R^2$ or  is closed and aspherical.  

\begin{defi} \label{def:formal_spec}  A function $c: C^{\infty}([0,1] \times \Sigma) \rightarrow \R$ is a \emph{formal} spectral invariant if it satisfies the following four axioms:
\begin{enumerate}
\item (Spectrality) $c(H) \in \spec(H)$ for all $H \in  C^{\infty}([0,1] \times \Sigma)$, where $\spec(H)$, the \emph{spectrum} of $H$, is the set of critical values of the Hamiltonian action, that is, the set of actions of fixed points of $\phi_{H}^1$.
\item (Non-triviality) There exists a topological disk $D \subset \Sigma$ and $H$ supported in $D$ such that $c(H) \neq 0$.
\item (Continuity) $c$ is continuous with respect to the $C^{\infty}$ topology on $C^{\infty}([0,1] \times M)$.
\item (Max formula) $c(H_1 +\ldots + H_N) =  \max\{c(H_1), \ldots, c(H_N)\}$ if  $H_i \in C^{\infty}([0,1] \times M)$ are supported in pairwise disjoint disks. 
 \end{enumerate}

\end{defi}

The fact that the invariant $\m N$ satisfies the spectrality and non-triviality axioms is an immediate consequence of its definition.  It is also not difficult to check that $\m N$ satisfies the max formula.  However, we do not know if $\m N$ satisfies the continuity axiom and thus we do not know if $\m N$ is a \emph{formal} spectral invariant.   The two spectral invariants constructed by Viterbo and Schwarz satisfy a long list of well known properties which include the above spectrality, non-triviality and continuity axioms.  It is a consequence of Theorem \ref{theo:max-formula-balls} that these two spectral invariants are indeed \emph{formal}.  Theorem \ref{theo:main} is now an immediate consequence of the following theorem.

\begin{theo}\label{theo.axiomatic-c=n}  Let  $c : C^{\infty}([0,1] \times \Sigma) \to \R$ denote a formal spectral invariant.
Then, $c (H)= \cN(H)$ for every $H \in C^{\infty}(\Sigma)$.
\end{theo}

 See Section \ref{sec:idea_of_proof} for an overview of the proof of the above theorem. An interesting feature of Theorem \ref{theo.axiomatic-c=n} is that it establishes a partial uniqueness result for spectral invariants which relies only on the above four axioms. As mentioned earlier the spectral invariants constructed via Floer and generating functions theories satisfy many properties.  We will prove in Section \ref{sec:formal-spec} that \emph{formal} spectral invariants share some of the  same properties such as Lipschitz continuity, monotonicity, conjugation invariance, and the energy-capacity inequality.  We do not know if formal spectral invariants satisfy the triangle inequality, or the property that $c(H)$ is attained by an orbit of Conley--Zehnder index $2n$.  However, it is a consequence of Theorems \ref{theo:main} and \ref{theo.axiomatic-c=n} that at the level of autonomous Hamiltonians the triangle inequality and the index property are satisfied by \emph{formal} spectral invariants.  It would be interesting to see if this can be extended to non-autonomous Hamiltonians or higher dimensional manifolds.

In light of the counterexample of Section \ref{sec:countr_exmpl}, we see that the spectral invariant constructed by Oh on the $2$--sphere is not a \emph{formal} spectral invariant.
  
\subsection{Further Consequences}
We now describe some consequences of the work carried out in this article.  

\subsubsection*{A simple description of the Entov-Polterovich quasi-state.} In this portion of the paper, $\Sigma$ denotes a closed surface other than $\S^2$.   Take $c$ to be any \emph{formal} spectral invariant on $C^{\infty}([0,1] \times \Sigma)$ and define
 \begin{equation}\label{eq:def-quasi-state}
 \zeta(H)=\lim_{k\to\infty}\frac1k c(kH),
 \end{equation} 
 for any autonomous function $H$.  The functional $\zeta$ was introduced by Entov and Polterovich in \cite{entov-polterovich06} and in their terminology it is referred to as a (partial) symplectic quasi-state.  Partial and genuine symplectic quasi states have been constructed on a large class of symplectic manifolds; see \cite{entov} for a survey of the subject.   It is well-known that the quasi-state on $\S^2$ admits a very simple description \cite{entov-polterovich03}. We will now give a simple description of $\zeta$ on aspherical surfaces.

Let $H$ be a Morse function on $\Sigma$ and suppose that $s \in \Sigma$ is a saddle point of $H$. Note that the connected component of $s$ in $H^{-1}(H(s))$ is a ``pinched'' loop.  We will call $s$ an \emph{essential} saddle if this pinched loop is not contractible in $\Sigma$.  In Section \ref{sec:quasi-states}, we will prove the following theorem.

\begin{theo}\label{theo:quasi-state}
For any Morse function $H$ on $\Sigma$, $\zeta(H)$ is the maximum of $H$ over all of its essential saddles.
 More generally, for any continuous function $H: \Sigma \to \R$,
\begin{equation}\label{for:quasi-state}
\zeta(H) = \inf\left\{h_{0}: H^{-1}(h_{0}, +\infty) \mbox{ is contractible in } \Sigma\right\}.\end{equation} 
 \end{theo}
 
 The quantity on the right hand side of  Formula \eqref{for:quasi-state} has already appeared in the literature in a different (but related) context:  it was introduced by Polterovich and Siburg in \cite{PS00} to study the asymptotic behavior of Hofer's metric on open surfaces with infinite area.

  A rather surprising consequence of the above result is that the functional $\zeta$ which is constructed via symplectic techniques, namely Floer theory, is in fact invariant under the action of all diffeomorphisms, i.e. $\zeta(f \circ \phi) =\zeta(f)$ for any diffeomorphism $\phi$.  Building on the works of Py \cite{Py05, Py06}, Zapolsky (in \cite{Zap}) and Rosenberg (in \cite{Rosen}) constructed genuine (and not partial) quasi-states on the torus and surfaces of genus higher than one, respectively.   Like $\zeta$, both of these quasi-states can be described by simple formulas which are different than the formula for $\zeta$.  The quasi-state on  the torus is only invariant under the action of symplectomorphisms while the other one is invariant under the action of all diffeomorphisms, like $\zeta$.

The above theorem has some interesting corollaries.  In \cite{entov-polterovich09}, Entov and Polterovich introduced the notions of heaviness and super-heaviness. 
A closed subset $X \subset \Sigma$  is called \emph{heavy} if $\zeta(H) \geq \inf(H|_X)$ for every function $H$. A closed subset $X$ is called \emph{superheavy} if $\zeta(H) \leq \sup(H|_X)$ for every function $H$.\footnote{Although it is not obvious from the definition, every superheavy set is necessarily heavy; see \cite{entov-polterovich09}.}   In \cite{kawasaki}, Kawasaki proves that the union of a longitude and a meridian in the torus $T^2$ is superheavy. \footnote{ We have been informed by Kawasaki that he is able to generalize the methods of \cite{kawasaki} to recover Proposition \ref{prop:heaviness}.} Using the above theorem we generalize Kawasaki's result and give the following characterization of heavy and super-heavy subsets of closed aspherical surfaces; see Section \ref{sec:quasi-states} for the proof.

\begin{prop}\label{prop:heaviness}
Let $X \subset \Sigma$ be a closed subset. Then, 
\begin{enumerate}
\item $X$ is heavy if and only if $X$ is not included in a disk.
\item $X$ is super-heavy if and only if any closed curve included in its complement is contractible in $\Sigma$.
\end{enumerate}
\end{prop}

Since the product of two super-heavy sets is super-heavy, the above result can be used to construct new examples of (strongly) non-displaceable sets.  We refer the reader to \cite{kawasaki} for a sample of such non-displaceability results.  

\subsubsection*{Dispersion free quasi states.}  A symplectic quasi-state $\zeta$ is said to be dispersion free if $\zeta(H^2) = \zeta(H)^2$ for any function $H$. It is known that the Entov-Polterovich quasi-state on $\S^2$ is dispersion free.  The functional $\zeta: C^{\infty}(\Sigma) \rightarrow \R$, defined above, is not dispersion free: using Theorem \ref{theo:quasi-state}, for example, one can find $H$ such that $\zeta(H) \neq \zeta(-H)$.  However, it follows immediately from Theorem \ref{theo:quasi-state} that $\zeta$ is dispersion free on the set of positive or negative functions and more generally $$\zeta(H^2) = \max\{\zeta(H)^2, \zeta(-H)^2\}.$$  This gives a partial answer to Question  3.4 of \cite{entov} and Question 8.5 of \cite{entov-polterovich09}.

\subsubsection*{Non-closed surfaces.}
It is not difficult to see that the invariant $\m N$ can be defined for compactly supported Hamiltonians  on any  surface.  Indeed, the definition does not rely on $\Sigma$ being closed. Following the work of  Frauenfelder and Schlenck \cite{FrSc} (see also \cite{Lan}) one can construct a formal spectral invariant $c$ on compact surfaces with boundary.  We expect that the equality $c = \m N$ continues to hold, for autonomous Hamiltonians, in this setting.

  In \cite{Lan}, Lanzat  constructs (partial) quasi-states on a class of non-closed symplectic manifolds which includes compact surfaces with boundary.  Now, given a formal spectral invariant $c$ on a non-closed surface one can define the functional $\zeta$ via Equation \ref{eq:def-quasi-state}.  We expect that $\zeta$ will continue to satisfy Formula \ref{for:quasi-state}.  Furthermore, we anticipate that the proof of Theorem \ref{theo:quasi-state} can be adapted to show that the partial quasi-state constructed by Lanzat coincides with $\zeta.$  Lastly, note that one  could  directly define $\zeta$ on any aspherical surface (closed or not) via Equation \eqref{for:quasi-state}. It can be checked that $\zeta$ (defined via Equation \eqref{for:quasi-state}) is a partial quasi-state in the sense of Lanzat \cite{Lan}.

\subsection{An overview of the proof of Theorem~\ref{theo.axiomatic-c=n}}\label{sec:idea_of_proof}
The strategy for proving that $c=\m N$ for autonomous Hamiltonians  consists of three main steps. First, in Section \ref{sec:c=N-Morse-plane}, we prove it for Morse functions on the plane.  This is achieved by proving that $\m N$ and $c$ satisfy the same recursive relation. 
  
The second main step is carried out in Section \ref{sec:c=N-Morse-surface} where we prove the equality for Morse functions on closed surfaces.  This is done by relating the values of both  $\m N$  and $c$ to their values on the plane. 

Finally, in Section \ref{sec:c=N-nonMorse}, we complete the proof by perturbing a general Hamiltonian to a carefully chosen nearby Morse Hamiltonian; the non-triviality of this final step stems from the fact that we do not know if $\m N(H)$ depends continuously on $H$.

The most difficult step is (perhaps) the proof of Proposition \ref{prop:recurs-form-c-disk} which establishes the aforementioned recursive formula for $c$.   Essentially, the argument consists in considering a continuous deformation from the zero Hamiltonian to $H$ and following the value of $c$ using the continuity axiom and a careful analysis of the deformation of the spectrum.  An important simplifying factor here is that, having obtained the recursive formula for $\m N$, we already know what it is that we are searching for.

Following the value of $c$ during deformations is facilitated by the tools developed in Section \ref{sec:formal-spec}. In particular, we prove that every \emph{formal} spectral invariant $c$ is monotone and Lipschitz continuous with respect to $H$, and satisfies the Energy-Capacity inequality: the value of $c$ for functions supported on a disk is bounded by the area of the disk.

\begin{figure}[h!]
\centering

\includegraphics[width=10cm]{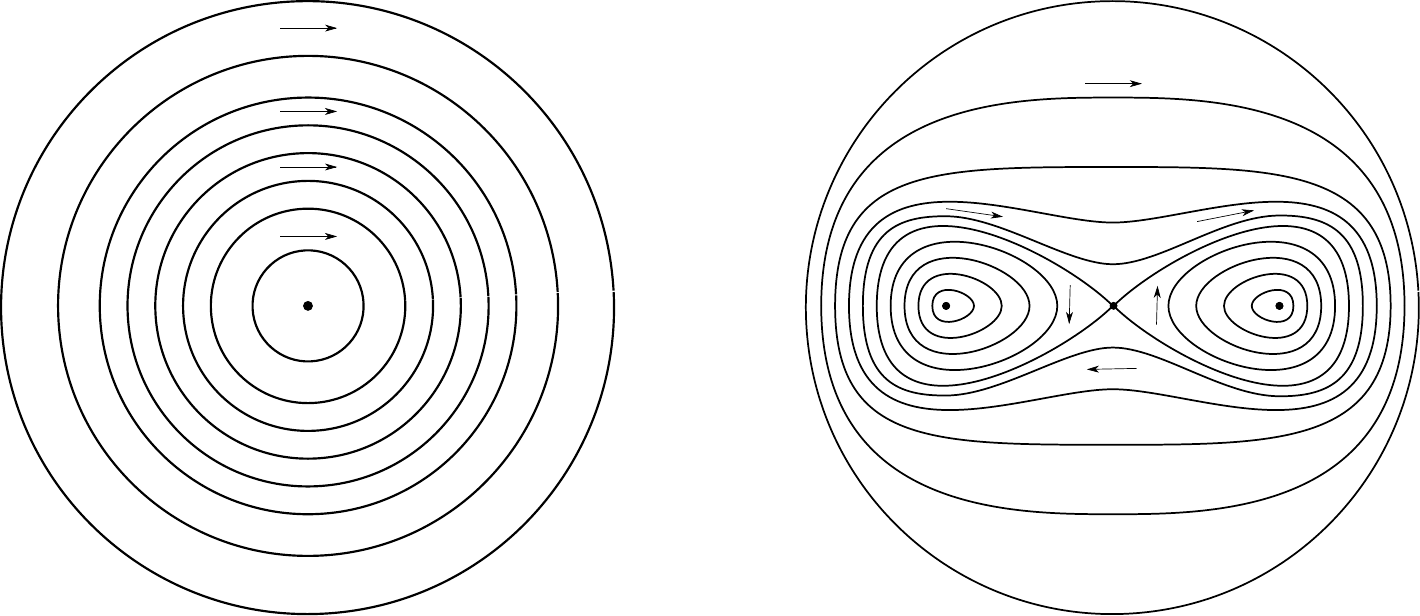}
\caption{Two simple examples of Morse Hamiltonians on the plane: a ``single mountain'' and a ``double mountain.''}
\label{fig:example-hamiltonian-intro}
\end{figure}

To get a taste for the real work, we shall consider here the two simplest scenarios; see Figure \ref{fig:example-hamiltonian-intro}. 
We focus on a non-negative Morse function $H$ on the plane.  
The first and easiest scenario is that of a function without any saddle point; the graph of such function looks like a ``single mountain.''
Let us call trivial the fixed points lying outside the support of $H$. 
Then every two non trivial fixed points of $\phi^1 _{H}$ are linked, and the definition of  $\m N$ entails that it coincides with the minimum value, say $a$, of the actions of its non trivial fixed points. 
By spectrality, the value of $c$ cannot be less than $a$. On the other hand we can bound $H$ from above by a function $G$ which still has $a$ as the minimal positive action, and whose other action values are larger than the area of its support. By the Energy-Capacity inequality $c(G)$ must be equal to $a$, and by monotonicity we get $c(H) \leq c(G) =a$, as wanted.

In the second simplest scenario, $H$ has a single saddle point $s$, and is larger than $H(s)$ on the two disks $T_{0}, T_{1}$ bounded by the level set of $s$.  In this case the graph of $H$ looks like a ``double mountain''.
Again the list of all maximal negative unlinked sets is easy to establish. Mnus's are of two kinds: in addition to the set of trivial fixed points which is contained in every mnus, the first kind consists of a single fixed point of $\phi^1 _{H}$ whose orbit surrounds the saddle point, and the second kind consists of the saddle together with one fixed point in each of the two disks $T_{0}, T_{1}$.
Denoting by $b, a_{0}, a_{1}$ the minimal positive values of the action respectively outside the saddle level and inside $T_{0}$ and $T_{1}$, the definition of $\m N$ yields 
$$\m N(H) = \min (b, \max(a_{0}, a_{1})).$$
Now we try to prove that $c(H)=\m N(H)$. Proving the upper bound $c(H) \leq \m N(H)$ is not much more difficult than in the first scenario (but it does rely on the max formula).  
The lower bound is the most delicate step of the proof, and goes as follows. First, we consider the case when the value of $c$ is attained outside the saddle level set. Here the definition of $b$ gives $c(H) \geq b \geq \m N(H)$, which lets us conclude the equality in this case. In the remaining case we have $c(H) \neq b$; since by  the upper bound $c(H) \leq b$, we get  $c(H) < b$. Now let us write
$
H = F+ H_{ T_{0}} + H_{ T_{1}},
$
where $F$ equals the constant value $H(s)$ on $T_{0} \cup T_{1}$, and $H_{T_{0}}$ and $H_{ T_{1}}$ are supported respectively on $T_{0}$ and $T_{1}$. In this outline we will pretend that these are smooth functions; note that $H_{T_{0}}$ and $H_{ T_{1}}$ have no saddle points and hence they are both ``single mountains.''
By a careful analysis of the action values, we construct a deformation
 $H_{\sigma}$ from $H_{0} = H$ to $H_{1} = H_{T_{0}} + H_{T_{1}}$  with the following properties: during the deformation,
\begin{itemize}
\item the part of the action spectrum corresponding to  orbits in  $T_{0} \cup T_{1}$, which we will refer to the ``inside'' spectrum, decreases at the constant speed $v=H(s)$,
\item the remainder of the spectrum, which we will refer to the ``outside'' spectrum, does not decrease faster than $v$.
\end{itemize}
Now the crucial point is that $c(H_{0}) < b$, whereas the ``outside'' spectrum for $H_{0}$ is no smaller than $b$.  Thus in the bifurcation diagram $\sigma \mapsto \spec(H_{\sigma})$, the connected component of $c(H_{0})$ is disjoint from the connected components of the ``outside'' spectrum, and this component  is a single line with slope $-H(s)$ (this will be clear in Figure \ref{fig:bif-fancy-def} in Section \ref{sec:proof-c=N-Morse-disk}).
 By continuity, we get that $c(H_{0}) = c(H_{1})+H(s)$. Then the max formula and the ``single mountain'' scenario give
$$
c(H_{1}) = \max( c(H_{T_{0}}), c( H_{T_{1}})) = \max(a_{0}-H(s), a_{1}-H(s)).
$$
We conclude that $c(H) = \max(a_{0}, a_{1}) \geq \m N(H)$, as wanted.

\subsection*{Organization of the paper}
In Section~\ref{sec:preliminary-N}  we give the precise definition of the invariant $\m N$ and discuss some of its properties. In Section ~\ref{sec:formal-spec} we  establish those properties of formal spectral invariants which will be used later on in the paper.  Section~\ref{sec:proof_c=N} is devoted to the proof of the main theorem, namely that every formal spectral invariant is equal to $\m N$ on the set of autonomous Hamiltonians. The ``max formulas'', which show that the Viterbo and Schwarz spectral invariants are indeed formal, are proved in Section~\ref{sec:max_form}. Section \ref{sec:countr_exmpl} contains a counter-example for a max formula on the sphere. Finally, a fundamental characterization of unlinked sets, a key ingredient in the definition of $\m N$, is proved in the appendix.

\section{Preliminaries: definition of $\m N$}\label{sec:preliminary-N}

In this section we introduce the notions of unlinked sets, rotation number of a fixed point, and Hamiltonian action that lead to the definition of our invariant $\m N$.
Many of the definitions and results of this section hold for general surface diffeomorphisms, not just Hamiltonian diffeomorphisms, and thus in Sections \ref{sec:unlinked_sets} and \ref{sec:rot_number} we work in this more general context.
\subsection{Unlinked sets}\label{sec:unlinked_sets}

\paragraph{Definition and characterisations.}
We consider an orientable surface $\Sigma$ which may be non-compact  but has empty boundary.
We denote by $\mathrm{Diff}_{0}(\Sigma)$ the group of diffeomorphisms that are the time one of a compactly supported isotopy.
Let  $(\phi^t)_{t \in [0,1]}$ be a compactly supported isotopy in $\Sigma$, and denote its time-one $\phi^1$ by $\phi$. A \emph{contractible fixed point} for the isotopy  is a fixed point $x$ of $\phi$ whose trajectory under $(\phi^{t})_{t \in [0,1]}$ is a contractible loop in $\Sigma$. If in addition $\phi^t(x) = x$ for every $t \in [0,1]$, we say that the isotopy \emph{fixes $x$}. 
\begin{defi}
A set $X$ of contractible fixed points of $(\phi^{t})_{t \in [0,1]}$ is \emph{unlinked} if there exists another isotopy $I$ whose time-one is $\phi$, which is homotopic to $(\phi^{t})_{t \in [0,1]}$ as a path in 
$\mathrm{Diff}_{0}(\Sigma)$ with fixed end-points, and that fixes every point of $X$.
\end{defi}
Note than when $\mathrm{Diff}_{0}(\Sigma)$ is simply connected, the notion of unlinkedness depends only on $\phi^1$. This includes the case when $\Sigma$ is the disk, the plane or any closed orientable surface except the sphere and the torus (\cite{gramain1973}). Likewise, on the torus, since $\mathrm{Ham}(\mathbb{T}^2)$ is simply connected (\cite{Pol_book}, Section 7.2), it depends only on $\phi^1$ if we restrict ourselves to Hamiltonian isotopies.\footnote{The space $\mathrm{Diff}_{0}(\Sigma)$ is also most probably simply connected when $\Sigma$ is any non compact surface, but we have no reference for this fact.}

The basic result on unlinked sets is the following.
\begin{theo}~\label{theo.unlinked-criteria1}
A set $X$ of contractible fixed points of $(\phi^{t})_{t \in [0,1]}$  is unlinked if and only if every finite subset of $X$ is unlinked.
\end{theo}
An important corollary of this theorem is the existence of unlinked sets that are maximal for inclusion (Corollary~\ref{coro.maximal-unlinked-sets}). In this paper, we will use the theorem to prove the existence of maximal \emph{negative} unlinked sets (see Corollary~\ref{coro.existence-mnus} below).
Theorem~\ref{theo.unlinked-criteria1} is proved in the Appendix, as well as Proposition~\ref{prop.unlinked-criteria2} below. Note that the existence of maximal unlinked sets for homeomorphisms is discussed in~\cite{Jaulent:2012aa}.

Theorem~\ref{theo.unlinked-criteria1} is complemented by a geometric characterization of unlinkedness for finite sets, which we describe now.  
Let $X$ be a finite set of contractible fixed points. 
A \emph{geometric pure braid} (based on $X$) is a map 
$b : X \times [0,1] \to \Sigma$ such that $b(x,0)=b(x,1)$ for every $x$ in $X$, and $x \mapsto b(x,t)$ is injective for every $t$. The isotopy $(\phi^{t})_{t \in [0,1]}$ generates the geometric pure braid 
$$
b_{X,(\phi^{t})} = (x,t) \mapsto \phi^t (x).
$$
We will say that this geometric braid \emph{represents the trivial braid} if there exists a continuous map $B : X \times [0,1] \times [0,1]  \rightarrow \Sigma$ such that $B(\cdot,\cdot,0)$ is the constant braid $(x,t) \mapsto x$, $B(\cdot,\cdot,1)=b_{X,(\phi^{t})}$, and $B(\cdot,\cdot,s)$ is a geometric braid for every $s$. 

\begin{prop}~\label{prop.unlinked-criteria2}
 A  finite set $X$  of contractible fixed points of $(\phi^{t})_{t \in [0,1]}$ is unlinked if and only if the geometric braid $b_{X,(\phi^{t})}$ represents the trivial braid.
\end{prop}
If $x$ is a contractible fixed point, then the geometric pure braid $b_{\{x\},(\phi^{t})}$ clearly represents the trivial braid. As a consequence of the proposition, the set $\{x\}$ is unlinked. For a more interesting example,
let us consider a pair $\{x,y\}$ of contractible distinct fixed points in $\Sigma = \R^2$. One can define the \emph{linking number} $\ell(x,y)$ 
as the degree of the circle map
$$
t \mapsto \frac{\phi^t(x) - \phi^t(y)}{\| \phi^t(x) - \phi^t(y) \|}.
$$
Then the pair $\{x,y\}$ is unlinked if and only if $\ell(x,y)=0$.

\paragraph{Unlinked subsets of disks.} 

Another consequence of the above results concerns the following situation. Assume that $\Sigma$ is not the sphere, and that our isotopy $(\phi^t)_{t \in [0,1]}$ fixes every point in some neighborhood of the boundary $\partial D$ of some open disk  $D$  in $\Sigma$. Let $X$ be a set of contractible fixed points of $\phi$ that is included in $D$. We will say that $X$ is \emph{unlinked in $D$} if there is  an isotopy in $\mathrm{Diff}_{0}(D)$ whose time one is  $\phi_{\mid D}$ that fixes every point of $X$.

\begin{corol}\label{cor:mnus_closed_surface}
In this situation, $X$ is unlinked if and only if it is unlinked in $D$.
\end{corol}

\begin{proof}
By Theorem~\ref{theo.unlinked-criteria1} it suffices to consider the case when $X$ is finite. 
If $X$ is unlinked in $D$, then the isotopy in $D$ given by the definition may be glued with the restriction of $(\phi^t)$ outside $D$ to provide an isotopy in $\Sigma$ which fixes every point of $X$, and we get that $X$ is unlinked. Now assume $X$ is unlinked. This means that the geometric braid $b_{X,(\phi^{t})}$ may be deformed into the trivial braid in $\Sigma$. 
The deformation starts with a braid included in $D$ and ends with the trivial braid in $D$, but the braid may go out of $D$ during the deformation.  Let us consider the situation in the universal cover $\tilde \Sigma$. We lift $D$ to a disk $\tilde D$, and let $\tilde X$ be the pre-image of $X$ in $\tilde D$.
The braid $b_{X,(\phi^{t})}$ lifts to a braid $\tilde b$ based on $\tilde X$ in $\tilde D$, and the deformation of $b_{X,(\phi^{t})}$ to the trivial braid lifts to a deformation of $\tilde b$ to the trivial braid in $\tilde \Sigma$.  Since the universal cover $\tilde \Sigma$ is contractible, it is easy to modify the deformation so that it takes place entirely in $\tilde D$. 
Now we project this new deformation down to $D$, and we see that the braid is trivial in $D$. Finally we apply Proposition~\ref{prop.unlinked-criteria2} in $D$ to get that $X$ is unlinked in $D$.
\end{proof}

\begin{remark}It can easily be seen that the above corollary still holds if $D$ is replaced with any compact incompressible subsurface of $\Sigma$.  We do not use the corollary in this generality.
\end{remark}

\paragraph{Unlinked sets for autonomous systems.} In this paragraph we assume again that $\Sigma$ is not the sphere.

We call an isotopy \emph{autonomous} if it is the flow of a time-independent vector field.  Let $(\phi^{t})_{t \in [0,1]}$ be an autonomous isotopy, and $x$ a contractible fixed point of $\phi = \phi^1$ which is not fixed by the isotopy. Then the trajectory of $x$ is a simple closed curve which bounds a unique disk, we denote this disk by $D(x)$. 
\begin{corol}\label{corol:unlinked-sets-autonomous}
 Let $X$ be a set of contractible fixed points of the autonomous isotopy $(\phi^{t})_{t \in [0,1]}$. Then $X$ is unlinked if and only if $X \cap D(x) = \{x\}$ for every point $x$ of $X$ which is not fixed by the isotopy.
\end{corol}

\begin{proof}
First assume $y$ is a point in $X \cap D(x)$ distinct from $x$. In the universal cover of $\Sigma$,  the lifts of $x$ and $y$ in some lift of $D(x)$ have a non zero linking number, and thus they are linked. An argument similar to the proof of Corollary~\ref{cor:mnus_closed_surface} shows that $\{x,y\}$ is linked in $\Sigma$. This proves the direct implication.

 The reverse implication goes as follows. Assume that for every point $x$ in $X$ which is not fixed by the isotopy, $X \cap D(x) = \{x\}$, and let us prove that $X$ is unlinked.  According to Theorem~\ref{theo.unlinked-criteria1} and Proposition~\ref{prop.unlinked-criteria2},  it suffices to prove that every geometric braid generated by a finite subset $X'$ of $X$  represents the trivial braid. For a point $x$ in $X'$ which is not fixed by the isotopy,  the single-strand braid generated by  $\{x\}$ represents the trivial braid, and we can choose the map $B$ deforming the braid so that it is supported in $D(x)$. Due to the hypothesis on $X$, all these deformations do not interfere, and together they give rise to a deformation of the braid $b_{X',(\phi^{t})}$ into the trivial braid (the strands corresponding to the fixed points of the isotopy stay still during the deformation).
\end{proof}

\subsection{Rotation number and negative unlinked sets}\label{sec:rot_number}

\paragraph{Definitions.}

For simplicity we restrict ourselves to a surface $\Sigma$ which is either the plane or a closed surface which is not the sphere. Consider a contractible fixed point $x$ for an isotopy $(\phi^t)_{t \in [0,1]}$ as before.
 Since $x$ is a contractible fixed point, there exists a ``capping disk'', i.e. a smooth map $u: \D^2 \to \Sigma$ from the unit disk $\D^2$ whose restriction to the unit circle is (a parametrization of) the trajectory $t \mapsto \phi^t(x)$. Since $\D^2$ is contractible, the pullback of 
$T\Sigma$ under $u$ may be identified with the trivial bundle $\D^2 \times \R^2$. 
Given a unit vector $v$ in $\R^2 \simeq \{x\} \times \R^2 \simeq T_{x} \Sigma$, the pullback of the path $t \mapsto (\phi_{H}^t(x),D_{\phi_{H}^t(x)}\phi_{H}^t .v)$ is a path $(t, v_{t})$ in $\D^2 \times (\R^2 \setminus \{0\})$. The map 
$$
(t,v) \mapsto \frac{v_{t}}{\norm{v_{t}}}
$$
is an isotopy in the circle. We call \emph{rotation number of $x$} and denote by $\rho(x)$ the rotation number of this isotopy, which is a real number defined as follows. We lift the isotopy to an isotopy $(F_t)_{t\in[0,1]}$ of $\R$, whose time one map $F_1$ is a homeomorphism of the line that commutes with the translation $s \mapsto s+1$; the rotation number of the isotopy is, by definition, the translation number of $F_1$, 
$$
\lim_{n \to +\infty} \frac{1}{n} (F_{1}^n(s)-s)
$$
for any $s \in \R$ (see for example~\cite{katok_hasselblatt}). 

\begin{lemma}
The rotation number $\rho(x)$ depends only on $x$ and $\phi^1$.
\end{lemma}
Here is a sketch of the proof. Notice that the rotation number of a circle homeomorphism is well defined as a real number modulo one. Thus, modulo one, $\rho(x)$ depends only on $x$ and $\phi^1$.
From this we first deduce that $\rho(x)$ does not depend on the trivialization of the tangent bundle over $u$. Then, since  $\pi_{2}(\Sigma)=0$,  we conclude that it does not depend  on the choice of the capping disk $u$ either. Likewise, we see that it depends only on the homotopy class of $(\phi^t)_{t \in [0,1]}$ as a path of diffeomorphisms. 
If $\Sigma$ is not the torus then $\mathrm{Diff}_{0}(\Sigma)$ is simply connected and we are done.
 It remains to take care of the torus. First note that on any surface, since according to Proposition~\ref{prop.unlinked-criteria2} the set $\{x\}$ is unlinked, the homotopy class of $(\phi^t)_{t \in [0,1]}$ contains an isotopy $I = (f_{t})_{t \in [0,1]}$ that fixes the point $x$. Thus we can use the isotopy $I$ and the trivial capping to define $\rho(x)$, and we see that $\rho(x)$ equals the rotation number of the action of the differential of this isotopy on the unit tangent bundle at $x$. Finally, when $\Sigma$ is the torus, we can conclude since the subgroup of elements of $\mathrm{Diff}_{0}(\Sigma)$ fixing $x$ is simply connected (\cite{gramain1973}, Th\'eor\`eme 2 and Proposition 2).


In the Hamiltonian context the rotation number may be generalized to higher dimensions, and is called the mean index, see for example \cite{SZ92, GiGu10}.

%

\begin{defi}
We say that an unlinked set $X$ is \emph{negative} if $\rho(x) \leq 0$ for every $x \in X$.
We say that a negative unlinked set $X$ is \emph{maximal} if there is no negative unlinked set $X'$ strictly containing $X$.
\end{defi}
Note that the rotation number, and hence being negatively unlinked,  is invariant under conjugation in the group $\mathrm{Diff}_{0}(\Sigma)$. Theorem~\ref{theo.unlinked-criteria1} has the following important consequence.

\begin{corol}\label{coro.existence-mnus}
Every negative unlinked set is contained in a maximal negative unlinked set.
Furthermore, the closure of a negative unlinked set is still a negative unlinked set, and maximal negative unlinked sets are closed.
\end{corol}
\begin{proof}
For the first part, we provide a short argument relying on Zorn's Lemma (a more constructive proof may be obtained by adapting the proof of Corollary~\ref{coro.maximal-unlinked-sets}). It suffices to consider a family $\cF$ of negative unlinked sets which is totally ordered by inclusion, and check that $\cF$ has an upper bound. Consider the union $X$ of all elements of $\cF$. Theorem~\ref{theo.unlinked-criteria1} entails that $X$ is unlinked, and clearly every point of $X$ has non positive rotation number. Thus $X$ is an upper bound for $\cF$. This proves the first sentence. For the second sentence, consider an unlinked set $X$, and let $I = (f_{t})_{t \in [0,1]}$ be an isotopy fixing every point of $X$. Then $I$ also fixes every point of $\overline{X}$, which shows that $\overline{X}$ is unlinked. Furthermore, at every point $x \in \overline{X} \setminus X$ the differential $D_{x}f_{t}$  has a fixed vector $v$ that does not depend on $t$; thus the rotation number $\rho(x)$ vanishes. This proves that $\overline{X}$ is a negative unlinked set. The closedness of maximal negative unlinked sets follows immediately.
\end{proof}

\paragraph{Negative unlinked sets for Hamiltonian isotopies.}
We now describe some properties that are specific to Hamiltonian systems.
 Let $\Sigma$ be equipped with a symplectic form $\omega$,
 consider a time-dependent Hamiltonian function $H \in C^{\infty}([0,1] \times \Sigma)$, and the corresponding Hamiltonian isotopy $(\phi_{H}^{t})_{t \in [0,1]}$.

\begin{lemma}\label{lemma:existence-negative-fixed-point}
There exists a negative contractible fixed point. 
\end{lemma}

\begin{proof}
In the case when $\Sigma$ is not compact, every point outside the support of $\phi$ is a negative contractible fixed point.

 In the case when $\Sigma$ is a compact surface, a negative contractible fixed point is provided by P. Le Calvez's proof of the Arnol'd conjecture. Let us recall the outline of the proof (see~\cite{lecalvez2005} for more details).  According to Corollary~\ref{coro.maximal-unlinked-sets} in the Appendix, there exists a maximal unlinked set $X$ for $\phi$. 
Since every accumulation point of $X$ has zero rotation number, we may assume that $X$ is finite.
Le Calvez's Brouwer foliated equivariant theorem provides an oriented foliation on $\Sigma \setminus X$ which is ``homotopically transverse'' to the flow $(\phi_{H}^{t})_{t \in [0,1]}$, which means that  every trajectory of the flow is homotopic in $\Sigma \setminus X$, with fixed end-points, to a curve which is positively transverse to the foliation (in other words, ``every leaf is pushed towards its right''). Such a foliation is ``gradient like'' : in particular, for every leaf $L$, there exists two distinct points  $\alpha(L), \omega (L)$ in $X$ such that the closure of $L$ equals $L \cup \{\alpha(L),\omega (L) \}$. By transversality, the point $\alpha(L)$ has non positive rotation number, and the point $\omega(L)$ has non negative rotation number.

 The existence of a fixed point with non-positive rotation number also follows from Floer's proof of the Arnold conjecture \cite{Floer88}. Indeed, if $\phi$ is non-degenerate it guarantees  the existence of a fixed point with Conley-Zehnder index 2. Such a point has a non positive rotation number according to the next remark, which we include only for the reader's convenience since it is not used in the paper. For degenerate $\phi$ the existence can be obtained by approximating $\phi$ in the $C^1$ topology with a sequence of non-degenerate diffeomorphisms. 
\end{proof}

\begin{remark}[ (Relation with the Conley--Zehnder index)]
   When $\phi^1_{H}$ is non-degenerate 
   its 1--periodic orbits can be indexed by the well known Conley--Zehnder index  $\CZ$ which takes values in the integers.  Many conventions are used for normalizing $\CZ$.  Our convention is as follows: Suppose that $H: \Sigma \to \mathbb{R}$ is a non-degenerate $C^2$--small Morse function. We normalize the Conley--Zehnder index so that for every critical point $p$ of $H$, 
$$ \mu_\mathrm{CZ}(p) = i_{\text{Morse}}(p),$$
where $i_{\text{Morse}}(p)$ is the Morse index of $p$. To be specific, in this case
$\mu_\mathrm{CZ}(p)$ is equal to $2$ if $p$ is a local maximum of $H$, $1$ if it is a saddle point and $0$ if it is a local minimum. Note that in the first case the rotation number $\rho(x)$ belongs to $(-1,0)$, it vanishes in the second case, and in the last case it belongs to $(0,1)$.
In general, the Conley--Zehnder index and the rotation number are related by the following formula: if $p$ is any contractible fixed point of a non-degenerate $\phi^1_{H}$, then
\begin{itemize}
\item  If $\mu_\mathrm{CZ}(p)$ is odd, then 
 $
 \displaystyle
 \rho(p) = \frac{-\mu_\mathrm{CZ}(p)+1}{2}.
 $
\item   If $\mu_\mathrm{CZ}(p)$ is even, then 
 $
 \displaystyle
 \rho(p) \in \left( \frac{-\mu_\mathrm{CZ}(p)}{2}, \frac{-\mu_\mathrm{CZ}(p)+2}{2}\right). 
 $
\end{itemize}

\end{remark}

\subsection{Action functional}

The definitions of this section are valid on every symplectic manifold $(M^{2n},\omega)$ which is symplectically aspherical, i.e., $\langle \omega,\pi_2(M)\rangle=0$. 
Given a (time-dependent) Hamiltonian function $H:[0,1]\times M\to\R$, the \emph{action functional} is the function $\cA_H$ defined on the space of contractible loops in $M$ by the formula $$\cA_H(x)=\int_0^1H(t, x(t))dt-\int_{\D^2}u^*\omega,$$ 
where $u$ is a capping disk of the loop $x$, i.e., a map $u:\D^2\to S$ such that $u|_{\partial\D^2}=x$. In other words, the term $\int_{\D^2}u^*\omega$ is the algebraic area enclosed by $x$. Since the manifold is assumed symplectically aspherical, this term does not depend on the choice of the capping disk.
Moreover, if one only allows \emph{mean normalized} Hamiltonians, i.e. Hamiltonians which are normalized by the condition $\forall t\in[0,1],\int_0^1H(t,x)\omega^n=0$, then the value of the action on periodic orbits does not depend on the choice of the generating Hamiltonian but only on the time one map $\phi_H^1$. This means that the action functional is well defined for fixed points of Hamiltonian diffeomorphisms. 
 If $x,y$  are two points that are fixed under the Hamiltonian flow, then $\cA_{H}(y)-\cA_{H}(x)$ can be geometrically interpreted as the quantity of area flowing through any curve joining $x$ to $y$ under the isotopy $(\phi_{H}^t)_{t \in [0,1]}$.

The most important feature of the Hamiltonian action is that its critical points are exactly the 1--periodic orbits of the Hamiltonian flow $\phi_H^t$ (by this we mean the periodic orbits whose period divides 1). The set of critical values of the action, i.e.,  values  on 1--periodic orbits, is called \emph{spectrum} of the Hamiltonian $H$ and is denoted $\spec(H)$.  It has Lebesgue measure zero. 
See Section \ref{sec:exampl-radi-hamilt} for an example of computation.

\subsection{Definition of $\m N$}
For simplicity again we restrict ourselves to a surface $\Sigma$ which is either the plane $\mathbb{R}^2$, the interior of a closed disk in the plane, or a closed surface which is not the sphere, although everything works on any surface $\Sigma$ for which the inclusion of $\mathrm{Ham}(\Sigma)$ into $\mathrm{Diff}_{0}(\Sigma)$ is trivial at  the level of the fundamental groups (see the footnote above). 

Let us consider a time-dependent Hamiltonian function $H \in C^{\infty}([0,1] \times \Sigma)$, the corresponding Hamiltonian isotopy $(\phi_{H}^{t})_{t \in [0,1]}$, and $\phi= \phi^1$. Remember that the notions of unlinkedness and rotation number depend only on $\phi^1$ and not on the isotopy.

For short we write $mnus$ for ``maximal negative unlinked set'', and we denote the family of mnus's by $\mathrm{mnus}(\phi)$ or $\mathrm{mnus}(H)$. According to Corollary~\ref{coro.existence-mnus}, there exists at least one mnus. Furthermore, since by Lemma~\ref{lemma:existence-negative-fixed-point} there exists a negative contractible fixed point, every mnus is non-empty. Hence the following definition is valid.
\begin{defi}
$$
\cN(H) = \inf_{X \in \mathrm{mnus}(\phi)} \ \ \sup_{x \in X} \cA_H(x).
$$
\end{defi}

If $\Sigma=\R^2$ then $\cA_H(x) = \cA_G(x)$ for every (compactly supported) Hamiltonian function $G$ whose time one is $\phi$. If $\Sigma$  is a closed surface then the same equality holds if $\int_{S} G d\omega= \int_{S} H d\omega$. In particular we may give the following definition.
\begin{defi}
Let  $\cN(\phi)$ be $\cN(H)$ where $H$ is any Hamiltonian function whose time one is $\phi$, normalized by  the condition $\int_{S} H d\omega = 0$ in the case $\Sigma$ is a closed surface.
\end{defi}
Note  that $\m N$ is invariant under conjugation by symplectic diffeomorphisms.

\subsection{Example: radial Hamiltonians}\label{sec:exampl-radi-hamilt}

In this subsection, we illustrate the notions introduced above on a basic but fundamental example. We will make intensive use of this example in the proof of Theorem \ref{theo.axiomatic-c=n}.
Let $H \in C^\infty(\R^2)$ be a smooth autonomous Hamiltonian on the plane, that only depends on the distance to the origin. It will be convenient to write $H$ in the form 
$$\forall x,y\in\R,\  H(x,y)= f(\pi(x^2+y^2)),$$
for some function $f:[0,+\infty)\to\R$. 

\paragraph{Fixed points.} The Hamiltonian vector field is given by 
$$
X_H=(-2\pi yf'(\pi(x^2+y^2)),2\pi xf'(\pi(x^2+y^2)))
$$
 and we see that the flow restricted to the circle of radius $r$ is the rotation by $2\pi f'(\pi r^2)$.
Thus, the fixed points of $\phi_H^1$ are, besides the origin, the points of $\R^2$ whose distance to the origin $r$ is such that $f'(\pi r^2)$ is an integer. 

\paragraph{Rotation numbers.}
Let $(x,y)$ be such a point, denote $s=\pi(x^2+y^2)$ and set $k=f'(s)$.  The orbit of $H$ makes exactly $k$ oriented turns along the circle centered in the origin and passing through $(x,y)$. 
The linearized flow of $H$  along the orbit, i.e. the linear map $D\phi_H^t(x,y)$, acts on a vector $\vec v$ tangent to the circle as the rotation by angle $2\pi kt$,

  thus $\rho(x,y)=k$. Therefore the fixed points with non-positive rotation number correspond to values of $s$ where $f$ is non-increasing. Note that the rotation number of the origin is $f'(0)$.

\paragraph{Mnus's.}
Let $p_{1}, p_{2}$ be two distinct fixed points of $\phi^1_H$. To fix ideas, assume that $p_{2}$ is no closer to the origin than $p_{1}$. Then the linking number $l(p_1, p_2)$ equals the rotation number $\rho(p_{2})$. This immediately leads to the following complete description of the mnus's. Let $X$ denotes the set of critical points of $H$. For every point $p=(x,y)$ such that $f'(\pi(x^2+y^2))$ is a negative integer, let $X_{p}$ denotes the union of $\{p\}$ and of the critical points of $H$ farther than $p$ from the origin. The sets $X_{p}$ are mnus's.
 If $f'(0) \leq 0$ then $X$ is a mnus, in the opposite case $X \setminus\{0\}$ is a mnus (note that, by the intermediate value theorem, in this case this last set is not included in any of the $X_{p}$'s).

\paragraph{Reading the Hamiltonian action on diagrams.}
The Hamiltonian action of these fixed points  is given by $$\cA_H(x,y)=f(s)-sk=f(s)-sf'(s).$$
It corresponds to the intersection of the vertical axis $\{0\}\times \R$ with the tangent to the graph of $f$ at the point $(s,f(s))$, see Figure \ref{fig:reading-action-1}.
The action can also be seen on the graph of minus the rotation number $-f'$. With the above notations, 
$\cA_H(x,y)=-(ks+\int_s^{+\infty}f'(\sigma)d\sigma)$.
This corresponds to the grey area in Figure~\ref{fig:action-rotation-number}.

\begin{figure}[h!]
\centering
\def\svgwidth{0.8\textwidth}
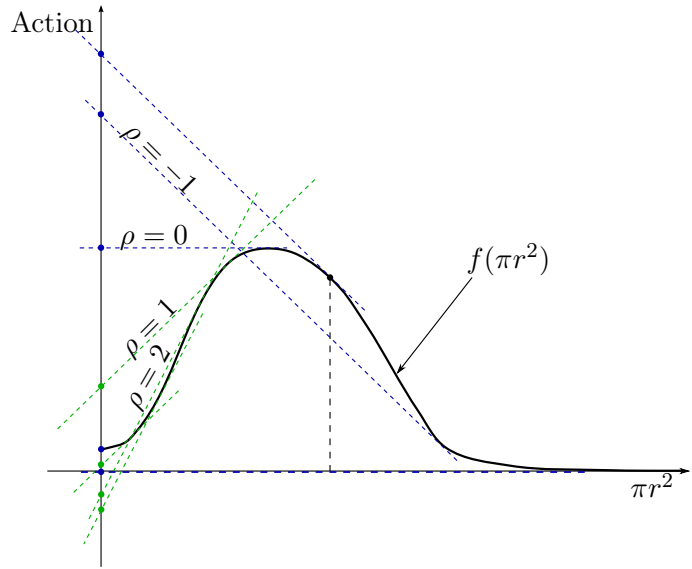
\caption{The dotted lines are the tangents to the graph of $f$ with integer slope. Their tangency points correspond to the fixed points of $\phi_H^1$. The intersections of these lines with the vertical axis (represented by thick dots) give the action. The points with non-positive rotation numbers are in blue.}

\label{fig:reading-action-1}
\end{figure}
\begin{figure}[h!]
\centering
\def\svgwidth{0.8\textwidth}
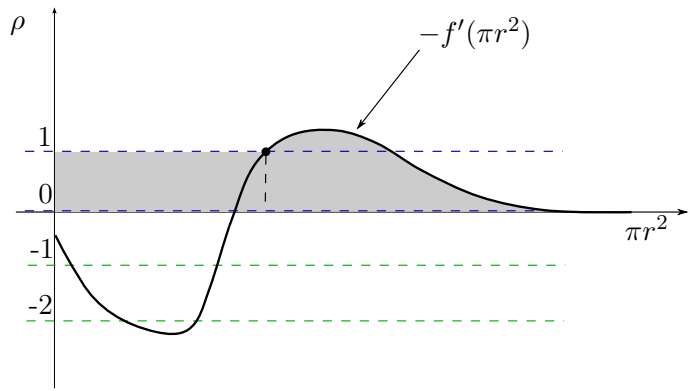
\caption{The fixed point correspond to intersections of the graph $\rho=-f'$ with the horizontal lines ``$\rho\,=$ integer constant''. 
The action of the thick black dot is the area of the grey region. This thick black dot corresponds to the thick black dot on Figure \ref{fig:reading-action-1}.}
\label{fig:action-rotation-number}
\end{figure} 

\paragraph{Computing $\cN$.}
First assume that the function $f:[0,+\infty)\to\R$ is decreasing and has non vanishing derivative on $(0,r_{0})$, where $[0,r_{0}]$ is the support of $f$. Let $Y$ be the complement in the plane of the open disk with radius $r_{0}$. 
The mnus's are the sets of the form $\{x\} \cup Y$
where $x$ is any fixed point not in $Y$. Finally we get 
\begin{align}\label{eq:example-single-mountain-N}
\cN(H) = \min_{x} \cA_H(x),
\end{align}
where the minimum runs on all fixed points of $\phi_{H}^1$ that are not in $Y$. With the interpretation of the action explained above, we see that it is a positive number, attained at a periodic orbit of period exactly one (which is not necessarily the action of the outermost periodic orbit).

Another case when $\cN$ is easy to compute is when $f$ takes only non positive values. Indeed, remember that the set of all critical points of $H$, taking out the origin in case $f'(0)>0$, is a mnus. Since every critical point has a non positive action, we see that $\cN(H) = 0$.
There does not seem to be any easy formula in the case of a general radial Hamiltonian.

\subsection{Max formula for $\cN$}
Here again we assume that $\Sigma$ is the plane or a closed aspherical surface.

\begin{lemma}[Max formula for $\cN$] 
   Suppose that $H_1,\ldots, H_N \in C^\infty([0,1] \times \Sigma, \R)$ are  Hamiltonian functions whose supports are contained in pairwise disjoint open disks $U_{1}, \dots , U_{N}$. Then
 $$\cN(H_1 +\ldots + H_N) =  \max\{\cN(H_1), \ldots, \cN(H_N)\}.$$
\end{lemma}
\begin{proof} 
By an easy induction, the proof boils down to the $N=2$ case.
Let $Y_{i}$ denotes the complement of the disk $U_{i}$ in $\Sigma$, and $Y$ be the complement of $U_{1} \cup U_{2}$. The crucial remark is the following:

\emph{The unlinked sets (resp. negative unlinked sets) of $H_{1}+H_{2}$ are the sets of the form
$$
Y' \cup X_{1} \cup X_{2}
$$
where 
\begin{itemize}
\item $Y'$ is included in $Y$,
\item  $X_{i}$, $i=1,2$ is included in $U_{i}$,
\item  $X_{i} \cup Y_{i}$ is an unlinked set (resp. negative unlinked set)  of $H_{i}$.
\end{itemize}
 The mnus's of $H_{1}+H_{2}$ have the same form with $Y'=Y$ and $X_{i} \cup Y_{i}$ is a mnus of $H_{i}$. }

The proof of this remark is a consequence of Corollary~\ref{cor:mnus_closed_surface}. 

We first check that unlinked sets correspond.
If $X_{1}$ is a subset of $U_{1}$ which is unlinked for $H_{1}$, then by Corollary~\ref{cor:mnus_closed_surface} it is unlinked for $H_{1}$ in $U_{1}$; this provides us with some isotopy which is compactly supported in $U_{1}$. If likewise $X_{2}$ is unlinked for $H_{2}$ in $U_{2}$ we get a second isotopy, and we can glue the two isotopies with the identity on $Y$ into an isotopy on $\Sigma$, yielding that $X_{1} \cup X_{2} \cup Y$ is unlinked for $H_{1}+H_{2}$. The proof of the converse implication is similar. For the converse implication, let $X$ be unlinked for $H_{1}+H_{2}$. Then the set $X_{i} = X \cap U_{i}$ is also unlinked, thus by Corollary~\ref{cor:mnus_closed_surface} it is unlinked in $U_{i}$ for $H_{1}+H_{2}$, but this is exactly the same thing as being unlinked in $U_{i}$ for $H_{i}$. Then obviously $X_{i} \cup Y_{i}$ in unlinked for $H_{i}$, and we get $X = Y' \cup X_{1} \cup X_{2}$ as wanted.
The correspondences between negative unlinked sets and mnus's follow immediately.
\end{proof}


\section{Preliminaries: properties of formal spectral invariants}\label{sec:formal-spec}
The main goal of this section is to establish certain properties of \emph{formal} spectral invariants which will be used later on in the paper.  Throughout the section $c$ denotes a \emph{formal} spectral invariant in the sense of Definition \ref{def:formal_spec}.  In Section \ref{sec:standard-prop}, we present those properties of $c$ which are standard in the sense that they are known to hold for  the Floer and generating-function theoretic spectral invariants.  In Section \ref{sec:symp-cont}, we introduce the symplectic contraction principle which provides a powerful tool in the study of spectral invariants on aspherical manifolds.

\subsection{The standard properties of $c$}\label{sec:standard-prop}
Properties 1--6 listed below are among the standard properties which are known to hold for the Floer and generating-function theoretic spectral invariants; see for example \cite{viterbo, schwarz, Oh05b}. The proofs we give in this section for the first six properties are similar to those presented in \cite{viterbo}.   It is interesting to observe that the proofs of the first five properties rely solely on  Spectrality and Continuity of \emph{formal} spectral invariants.  The last two  properties, which prove that $c(H)$ is positive for a large  class of Hamiltonians, rely on the Max formula.  Lastly, we should mention that one standard property of  Floer and generating function theoretic spectral invariants which we have not been able to prove is the triangle inequality.  

\bigskip

\noindent\textbf{1. Symplectic invariance:}   $c(H) = c(H \circ \psi) \;\;\forall H \in C^{\infty}([0,1] \times  \Sigma ),   \forall \psi \in Symp_0,$ where $Symp_0$ denotes the path component of the Identity in $Symp(\Sigma, \omega)$.

\begin{proof} It is a classical fact that $\spec(H\circ \phi) = \spec(H)$ for any symplectomorphism $\phi$.  Let $\psi_s$ denote a path in $Symp_0$ such that $\psi_0 = Id$ and $\psi_1 = \psi.$   Now, the continuous function $s \mapsto c(H \circ \psi_s)$ takes values in the measure zero set $\spec(H)$ and hence it must be constant.
\end{proof}

\medskip
\noindent\textbf{2. Shift:}  $c(H  + r) = c(H) + \int_0^1 r(t) \, dt, $ where $r:[0,1] \rightarrow \R$ is a function of time.
\begin{proof}
For $s \in [0,1]$ let $H_s = H + s r$.  Note that $\spec(H_s) = \spec(H) + s \int_0^1 r(t) \, dt.$   Hence, by the Continuity and Spectrality axioms, the function $s \mapsto c(H_s) - s \int_0^1 r(t)$ is continuous and takes values in the measure-zero set $\spec(H)$.  Hence, it must be constant.  The Shift property follows immediately.
\end{proof}

\medskip
\noindent \textbf{3. Monotonicity:} $c(H) \leq c(G)$ if $H\leq G$.

\begin{proof}  See the proof of Lipschitz continuity. \end{proof}

\medskip
\noindent\textbf{4. Lipschitz continuity:} $ \displaystyle \int_{0}^{1} \min_{x \in M } (H_t-G_t) \, dt \leq c(H) - c(G) \leq \int_{0}^{1} \max_{x \in M } (H_t-G_t) \, dt.$

\begin{proof} We will simultaneously prove monotonicity and  Lipschitz continuity. The continuity axiom implies that it is sufficient to prove these properties in the special case where both $H$ and $G$ are non-degenerate.  For non-degenerate Hamiltonians both of these properties follow from Lemma \ref{lem:c_rate_change}, stated below: take $F_s = G + s (H - G)$ and note that  $\frac{\partial F_s}{\partial s} = H-G$. If $F_s$ is an admissible family in the sense of Lemma \ref{lem:c_rate_change} then both results follow immediately.  If $F_s$ is not admissible then we can perturb it by a $C^2$--small amount and obtain an admissible family such that  $\frac{\partial F_s}{\partial s} \approx H-G$.  We leave the details of this to the reader.

We will now state and prove Lemma \ref{lem:c_rate_change}.  Consider a $1$--parameter family of time dependent Hamiltonians $H_s(t,x), \, s \in [0,1]$ which depends smoothly on $s$.  We call $H_s$ \emph{admissible} if there exists a finite (possibly empty) set of points $\{s_1, \cdots, s_k \}\subset (0,1)$:
\begin{enumerate}
\item The set of  fixed points of $\phi^1_{H_s}$ is finite $\forall s \in [0,1]$,
\item  $\forall s \in [0,1] \setminus \{s_1, \cdots, s_k\},$ the Hamiltonians $H_s$ is non-degenerate and no two fixed points of $\phi^1_{H_s}$ have the same action.
\end{enumerate}
A generic (in the sense of Baire) $1$--parameter family of Hamiltonians is admissible.
\begin{lemma}\label{lem:c_rate_change}Let $H_s(t,x), \, s \in [0,1]$ denote an \emph{admissible} family of Hamiltonians. The function $s \mapsto c(H_s)$ is differentiable at every $s \in [0,1]$ except the finite set of points $\{s_1, \cdots, s_k\}$ where $H_s$ is degenerate and furthermore,
$$  \int_{0}^{1} \min_{x \in M }\frac{\partial H_s}{\partial s}(t, x) dt\leq \frac{d}{ds}c(H_s) \leq \int_{0}^{1} \max_{x \in M }\frac{\partial H_s}{\partial s}(t, x) dt.$$
\end{lemma}
\begin{proof}[Proof of Lemma \ref{lem:c_rate_change}]
Let $I_k$ denote the open interval $(s_k, s_{k+1})$ and consider $s \in I_k$.  There exists a 1-- periodic orbit $x_s$ of $\phi^1_{H_s}$ such that $c(H_s) = \m A_{H_s}(x_s)$.  The admissibility condition implies that the fixed point $x_s$ varies smoothly on the entire interval $I_k$; indeed no bifurcations take place in this interval.  Furthermore, since $c$ is continuous it must be the case that $c(H_s) = \m A_{H_s}(x_s)$.  This implies that in fact $c$ is smooth in the interval $I_k$ hence we can differentiate: We will use the symbol $x_s$ to denote the 1--periodic orbit associated to the fixed point $x_s$.
$$\frac{d}{ds}c(H_s) = \frac{d}{ds} A_{H_s} (x_s) = \frac{\partial}{\partial r} \m A_{H_r}(x_s) + \frac{\partial}{\partial r}  \m A_{H_s}(x_r).$$
Now, $\frac{\partial}{\partial r} \m A_{H_s}(x_r) = 0$ because $x_s$ is a critical point of $ \m A_{H_s}$.  A simple computation yields
 $$\frac{\partial}{\partial r} \m A_{H_r}(x_s) =  \int_{0}^{1} \frac{\partial H_s}{\partial s}(t, x_s(t)) \, dt.$$  The result follows immediately.
\end{proof}\end{proof}
\begin{remark} \label{rem:c-for-cts-fns}
Observe that the Lipschitz continuity property of $c$ allows us to extend $c$ to all continuous functions.
\end{remark}
\medskip
\noindent \textbf{5. Energy-Capacity inequality:} Let $K, H$ be two Hamiltonians such that $\phi^1_K$ displaces the support of $H$.  Then, $|c(H)| \leq \int_{0}^{1} (\max_{x \in M } K_t - \min_{x \in M } K_t ) \, dt.$

\begin{proof}
For each $s \in [0,1]$ consider the Hamiltonian $F_s(t,x) = sH(st,x) + K(t, (\phi^{st}_H)^{-1}(x)$.  The time-1 map of the flow $F_s$ is given by $\phi^s_H \circ \phi^1_K$.  Using the fact that $ \phi^1_K$ displaces the support of $H$ one can prove that the fixed points of $\phi^s_H \circ \phi^1_K$ are precisely the fixed points of $\phi^1_K$ and furthermore for each fixed point $x$ we have $\m A_{F_s} (x) = \m A_K(x)$.  Hence, $\spec(F_s) = \spec(K)$.  It then follows that the continuous function $s \mapsto c(F_s)$ is constant and thus, $$c(K) = c(H(t,x) + K(t, (\phi^{t}_H)^{-1}(x)).$$  This, combined with the Lipschitz continuity of $c$, yields:
$$  c(H) - c(K) \leq \int_{0}^{1} \max_{x \in M } (-K_t(\phi^{t}_H)^{-1}(x))) \, dt = - \int_{0}^{1} \min_{x \in M } K_t \, dt,$$
and thus $c(H) \leq c(K) - \int_{0}^{1} \min_{x \in M } K_t \, dt.$ Using the Lipshitz continuity again we obtain:
$c(H) \leq  \int_{0}^{1} (\max_{x \in M } K_t - \min_{x \in M } K_t) \, dt.$  Similarly, one proves that $- \int_{0}^{1} (\max_{x \in M } K_t - \min_{x \in M } K_t) \, dt \leq c(H).$
\end{proof}

We do not use the following property in this article.  However, we state it as it is one of the standard properties of the Floer and generating function theoretic spectral invariants.

\medskip
\noindent \textbf{6. Path independence:} Suppose that $\phi^1_H = \phi^1_G$.  If $\Sigma = \R^2$  then $c(H) = c(G)$.  In the case where $\Sigma \neq \R^2$ then $c(H) = c(G)$ if we assume additionally that $\int H_t \omega^2 = 0 = \int G_t \omega^2$ for each $t \in [0,1].$
\begin{proof}
Since $Ham(\Sigma)$ is simply connected there exists a path of Hamiltonians $F_s$ such that $F_0 = H, F_1 = G$ and $\phi^1_{F_s} = \phi^1_H = \phi^1_G$.  It follows from our assumptions that $\spec(F_s) = \spec(H) = \spec(G)$ and hence the function $s \mapsto c(F_s)$ is constant.
\end{proof}

\medskip

\noindent\textbf{7. Positivity:} If $H$ is supported in a disk, then $c(H)\geq 0$.

\begin{proof} This follows readily from the Max formula applied to $H$ and 0 which gives: $c(H)=\max(c(H),0)$.
\end{proof}

\medskip
\noindent \textbf{8. Non-degeneracy:}  If $H \neq 0$ and $H\geq 0$, then $c(H)>0$.
 \begin{proof}
   One can find a small disk $D$, a short time interval $[t_0, t_1]$, and a positive constant $m$ such that $H(t,x) \geq m$ for all $(t, x) \in [t_0, t_1] \times D.$  By Lemma \ref{lem:c_pos}, there exists a Hamiltonian $F$ such that
   \begin{itemize}
   \item $F$ is supported in $D$,
   \item $ F(t, x) <\frac{(t_1 -t_0)}{2}m$ for each $x$ in the interior of $D$,
   \item $c(F) > 0.$
   \end{itemize}
We will show that $c(H) \geq c(F)$.   
Let $\alpha:[0,1] \rightarrow [0,1]$ denote a smooth reparametrization of $[0,1]$ such that 
   \begin{itemize}
   \item $\alpha(t) = 0 $ for all $t \leq t_0,$ 
   \item $\alpha'(t) < \frac{2}{(t_1 -t_0)},$
   \item $\alpha(t)  = 1$ for all $t \geq t_1$.  
   \end{itemize}
   Set $G(t,x) = \alpha'(t) F(\alpha(t), x)$. The flow of $F$ is given by $\phi^t_F(x) = \phi^{\alpha(t)}_f(x)$ and so by the path independence property $c(G) = c(F)$.  On the other, $G(t,x) \leq H(t,x)$ and hence, by monotonicity, $c(G) \leq c(H)$.  
    It remains to prove the following lemma:
   \begin{lemma}\label{lem:c_pos}
   For any disk $D \subset \Sigma$ and any positive constant $\epsilon$, there exists a Hamiltonian $F$ such that $F \leq \epsilon$,  support of $F$ is contained in $D$ and  $c(F)> 0.$
   \end{lemma}
   The proof of this lemma uses the ``symplectic contraction'' principle which is described in Section \ref{sec:symp-cont}.  We postpone the proof to the end of that section.
 \end{proof}

\subsection{The symplectic contraction principle} \label{sec:symp-cont}
In this section we introduce the ``symplectic contraction'' technique which describes the effect of the flow of a Liouville vector field on a formal spectral invariant $c$. This technique has been used by Polterovich in \cite{polterovich}.  Throughout this section we will work on a general aspherical symplectic manifold $M$ and  we suppose that $c: C^{\infty}([0,1] \times M)$ is any function satisfying the three axioms of Definition \ref{def:formal_spec}.  The reason for working in this generality is that the symplectic contraction technique is used in our proof  Theorem \ref{theo:max-formula-aspherical} which holds for aspherical manifolds of higher dimensions.

   Recall that a domain $U \subset M$ is said to be \emph{Liouville domain} if the closure of $U$ admits a vector field $\xi$ which is transverse to the boundary $\partial U$ and satisfies $L_{\xi} \omega = \omega$, where $L$ is the Lie derivative. The vector field $\xi$ is referred to as the Liouville vector field of the domain $U$.  Note that the Liouville vector field $\xi$ necessarily points outward along $\partial U$ and therefore the flow $A_t: U \rightarrow U$ of  $\xi$ is defined $\forall t \leq 0.$  This flow ``contracts'' the symplectic form $\omega: \, \, A_t^* \omega = e^t \omega.$  Recall also that  $U \subset M$ (not necessarily Liouville)  is called incompressible if the map  $i_* : \pi_1(U) \rightarrow \pi_1(M)$, induced by the inclusion $i: U \rightarrow M,$ is injective.

Let $U$ denote an \emph{incompressible Liouville} domain in  $M$ and let $F:[0,1]\times M \rightarrow \R$ be a Hamiltonian supported in $U$.  For each fixed $s \leq 0$ consider the Hamiltonian

 $$F_s(t,x):=  \begin{cases}
e^s F(t, A_s^{-1}(x)) &\mbox{ if } x \in A_s(U),\\
0 &\mbox{ if } x\notin A_s(U).
\end{cases} $$
It can be checked that the Hamiltonian flow of $F_s$ is given by $$\phi^t_{F_s}(x):=  \begin{cases}
A_s \phi^t_F A_s^{-1}(x) &\mbox{ if } x \in A_s(U),\\
x &\mbox{ if } x\notin A_s(U).
\end{cases} $$

It follows that there exists a 1--1 correspondence between the 1--periodic orbits of $F$ and $F_s$. Indeed, if $x(t)$ is a 1--periodic orbit of $F$ then $x_s := A_s(x(t))$ is a 1--periodic orbit of $F_s$.   Next, we claim that  $x$ is contractible if and only if $x_s$ is and furthermore
$\m A_{F_s}(x_s) = e^s \m A_{F}(x)$:    Since $U$ is incompressible we can pick a capping disk $D$ contained in $U$ for the orbit $x$.  Then, $A_s(D)$ is a capping disk for $x_s$.  Now, we compute 
 $$\m A_{F_s}(x_s) = \int_{0}^{1} e^s F(t, A_s^{-1}(x_s(t))) - \int_{A_s(D)} \omega$$ 
 $$ = \int_{0}^{1} e^s F(t, (x(t))) - \int_D e^s \omega = e^s \m A_{F}(x).$$
It follows that 
\begin{align} \label{eq:spec_dial1}
\spec(F_s) = e^s \spec(F).  
\end{align}
Using the spectrality and continuity properties of spectral invariants we conclude that
\begin{align} \label{eq:spec_dial2}
c(F_s) =  e^s c(F).  
\end{align}
We have symplectically contracted the Hamiltonian $F$. \\

We end this section with a proof of Lemma \ref{lem:c_pos}.
\begin{proof}[Proof of Lemma \ref{lem:c_pos}]
By the non-triviality axiom there exists a disk $D_0$ and a Hamiltonian $H$ supported in $D_0$ such that $c(H) \neq 0$.  By the max formula $c(H)$ is necessarily positive.  
Note that a disk is a Liouville domain and so we can apply the symplectic contraction principle.  Let $H_s$ denote a symplectic contraction of $H$ as described above.  Picking $s$ to be sufficiently negative yields $|H_s| \leq \epsilon.$
Observe that  $H_s$ is supported in the disk $A_s(D_0)$ whose area is $e^s Area(D_0)$.  Hence, by picking $s$ to be sufficiently negative we can ensure that the area of the support of $H_s$ is smaller than the area of the disk $D$ and so we can find a Hamiltonian diffeomorphism $\psi$ which maps the support of $H_s$ into $D$.  Set $F = H_s \circ \psi$.    The Hamiltonian $F$ is supported in $D$, is bounded above by $\epsilon$ and, using the symplectic invariance property and the symplectic contraction principle, we see that $c(F) = c (H_s) = e^s c(H) >0.$

\end{proof}

\section{Proof of Theorem \ref{theo.axiomatic-c=n}}\label{sec:proof_c=N}
As mentioned in the introduction, Theorem \ref{theo:main} is an immediate consequence of Theorems \ref{theo:max-formula-balls} and \ref{theo.axiomatic-c=n}.  The main goal of this Section is to prove Theorem \ref{theo.axiomatic-c=n}.  This is done in three stages.  In Sections \ref{sec:c=N-Morse-plane} and \ref{sec:c=N-Morse-surface}, we prove the theorem for Morse functions on the plane and closed surfaces of positive genus, respectively.  In Section \ref{sec:c=N-nonMorse}, we explain how one can pass from Morse functions to general autonomous Hamiltonians. 

Along the way, we obtain several results which may be of independent interest as they describe  algorithms for computing \emph{formal} spectral invariants and the invariant $\m N$ on autonomous Hamiltonians.  For example, Propositions \ref{prop:recurs-form-disk} and \ref{prop:recurs-form-c-disk} provide recursive formulas for computing $\m N$ and $c$ on the plane.  Propositions  \ref{prop:form_N_general} and \ref{prop:c=max-on-disks} are quite surprising as they demonstrate that computing $\m N$ and $c$ on closed surfaces can easily be reduced to computations on the plane!  In Section \ref{sec:quasi-states}, we use Proposition \ref{prop:c=max-on-disks} to prove Theorem \ref{theo:quasi-state} and Proposition \ref{prop:heaviness} on the Entov-Polterovich quasi-state. 

\subsection{Theorem \ref{theo.axiomatic-c=n} for Morse functions on the plane} \label{sec:c=N-Morse-plane}

In this section, we prove the equality $c=\cN$ for Morse functions on the plane. Throughout this section, we call a function $H:\R^2\to\R$  a \emph{Morse function} if its support is a closed topological disk and it admits finitely many critical points in the interior of its support, all of which are non degenerate and corresponds to distinct values of $H$.
 
Proving that $c=\cN$ for such functions is done in two steps. We first establish a recursive formula for $\cN$; see Proposition \ref{prop:recurs-form-disk} in  Section \ref{sec:recurs-form-disk}.  We  then show that this relation is also satisfied by $c$; see Proposition \ref{prop:recurs-form-c-disk} in Section \ref{sec:proof-c=N-Morse-disk}.

\subsubsection{A recursive formula for $\m N$}\label{sec:recurs-form-disk}
The main goal of this section is to present and prove a recursive formula for $\m N$; this formula appears in Proposition \ref{prop:recurs-form-disk}.  Giving a precise statement of this formula will require some preparation.  Let $H: \R^{2} \to \R$ be a Morse function. Assume $H$ admits at least one saddle point. 
For a saddle point $s$ of $H$, we consider the level set $H^{-1}(H(s))$ and let $C(s)$ be the connected component of $s$ in this set; $C(s)$ is the union of the stable and the unstable manifold of $s$ for the flow $(\phi_{H}^{t})$, and it is homeomorphic to a bouquet of two circles (see Figure~\ref{fig:outer-saddle} below). 
 Let $d$ be one of the two bounded connected components of $\R^2\setminus C(s)$. The function $H_D=H|_{\overline D}-H(s)$ vanishes on the boundary of $D$. 
We would like to relate $\cN(H)$ to $\cN(H_D)$. However, $H_D$ is not smooth.  To circumvent this problem we will introduce an appropriate class $\bar H_{D}$ of smoothings of $H_D$. Then the recursive formula in Proposition~\ref{prop:recurs-form-disk} will express $\m N(H)$ in terms of $\m N(\bar H_{T_{0}}),\m N(\bar H_{T_{1}})$, where $T_{0}, T_{1}$ are the two bounded connected components of the complement of $C(s_{0})$ for the outermost saddle point $s_{0}$ of $H$ (see Notation~\ref{notation:proof}). We will keep the following notations throughout Section \ref{sec:proof_c=N}. 

\begin{figure}[h!]
\centering
\def\svgwidth{0.8\textwidth}
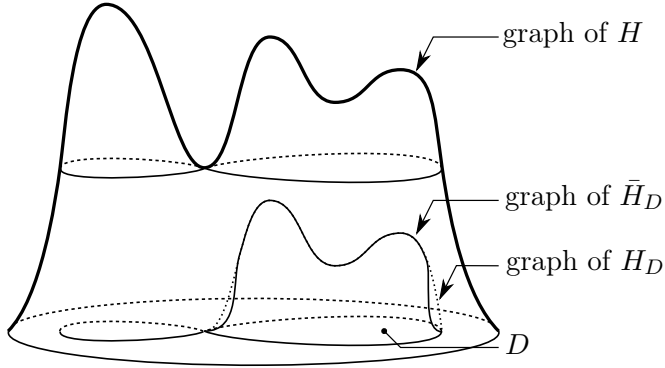
\caption{Graphs of $H, H_{D}, \bar H_{D}$}
\label{fig:H-bar-D}
\end{figure}

\begin{notation}\label{notation:bar-H} (See Figure~\ref{fig:H-bar-D})
  Assume that $H|_D>H(s)$ near the boundary of $D$ (we leave it to the reader to adapt the notations in the opposite case). We denote by $\m E_{D}$ the set of all functions $\bar H_D$ of the form $$\bar H_{D}(x)=\begin{cases}
0 &\mbox{ if } x \notin D,\\
\rho\circ H(x)-H(s)  &\mbox{ if } x\in D\setminus D', \\
H(x)-H(s)  &\mbox{ if } x\in D',
\end{cases}
$$
where:
\begin{itemize}
\item  $D'\subsetneq D$ is an open disk which contains all the 1--periodic orbits of $H$ in $D$ and such that for some constant $h> H(s)$,  $H|_{\partial D'}=h$,
\item  $\rho:(H(s),h)\to [H(s),h)$ is a smooth function such that for some $\eps>0$, 
\begin{itemize}
\item $\rho(t)=H(s)$ for all $t\in(H(s),H(s)+\eps]$,
\item $\rho(t)=t$ for all $t\in[h-\eps,h)$,
\item $0<\rho'(t)<\tau$ for all $t\in(H(s)+\eps,h-\eps)$, where $\tau>1$ denotes the smallest period of orbits of $H$ in $D\setminus D'$.
\end{itemize}
\end{itemize}  
\end{notation}

In the sequel, the notation $\bar H_D$ will be used for any function in $\m E_D$. 
The relevant properties of the functions $\bar H_D$ are summarized in the next lemma,  whose proof is straightforward.

\begin{lemma}\label{lemma:properties-bar-H} Every Hamiltonian $\bar H_{D}\in \m E_D$ is smooth and enjoys the following properties:
\begin{enumerate}
\item The support of $\bar H_{D}$ is a disk $D''$ included in $D$,
\item $\bar H_{D} = H - H(s)$ on a closed disk included in the interior of $D''$, which contains all the fixed points  of both $\phi_H^1$ and $\phi_{\bar H_{D}}^1$ that are contained  in the interior of $D''$.
\item For every such fixed point $x$ of $\phi_H^1$ in $D$, $ \cA_{\bar H_{D}}(x) =  \cA_H(x) -H(s)$, so that, in particular, all the elements in $\m E_D$ have the same spectrum. 
\end{enumerate}
Moreover, the set  $\m E_D$ is convex, and the continuous function $H_D$ that coincides with $H-H(s)$ on $D$ and vanishes elsewhere, belongs to its $C^0$-closure.
\end{lemma}



\begin{remark}\label{remark:c(barH)} The continuity and spectrality properties imply that the spectral invariant $c$ is constant on the set $\m E_D$. Moreover this constant value is $c(H_D)$. 
\end{remark}

Before giving a precise statement of the recursive formula promised at the beginning of this section, we need to introduce a new set of notations that will follow us throughout the proof. 


\begin{figure}[h!]
\centering
\def\svgwidth{0.8\textwidth}
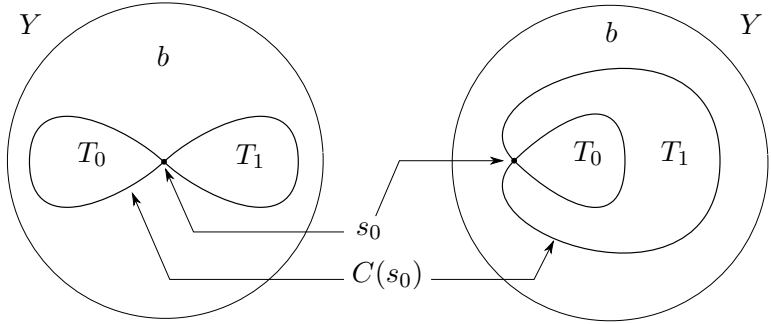
\caption{Notations $Y, s_{0}, C(s_{0}), b, T_{0}, T_{1}$: the two cases}
\label{fig:outer-saddle}
\end{figure}

\begin{notation}\label{notation:proof} (See Figure~\ref{fig:outer-saddle}) Let $H$ be a Morse function which admits at least one saddle point. 
\begin{enumerate} 
\item We denote by $\supp(H)$ the support of $H$ in $\R^2$, and by $Y$ the unbounded component of the closure of the complement of $\supp(H)$.  

\item There exists a saddle point of $H$, which we denote by $s_{0}$, such that  the interior of the outer component of $\supp(H)\setminus C(s_0)$ contains no critical point of $H$. 

Here is a brief argument as to why this outermost saddle $s_{0}$ 
must exist.  For every saddle $s$ choose a nearby periodic orbit surrounding $C(s)$, and  remove the (open) disk bounded by this orbit. Likewise for every local maximum or minimum  remove a small open disk bounded by a periodic orbit. We are left with a set $A$ which is a disk with a certain number of holes, foliated by level sets of $H$, containing  no critical point of $H$. According to the Poincar\'e-Hopf Theorem, the Euler characteristic of $A$ is zero, thus $A$ is an annulus, which means there was only one hole after all.  Hence, there was either only one critical point (local maximum or minimum), or there was an outermost saddle.
\item We denote by $b,T_{0}, T_{1}$ the three connected components of $\supp(H) \setminus C(s_{0})$, $b$ being the outer one. Note that $b$ contains no critical point of $H$.
\item If moreover $H>0$ on $b$, we set:
$$\m N_{b} = \min\{ \cA_H(x)\,|\, x \text{ fixed point of } \phi_{H}^{1} \text{ in }b\}.$$
\end{enumerate}
\end{notation}

In the case where $H$ has no saddle, we set $b$ to be the interior of the support of $H$ and define $\m N_{b}$ by the same formula  when $H >0$ on $b$.
Note that when $H$ is positive on $b$, the orbits inside $b$ turn in the negative direction and hence $\rho(x) \leq 0$ for every fixed point $x \in b$.

The above construction may be applied to $T_0$, $T_1$ giving rise to two sets $\m E_0$ and $\m E_1$ of functions $\bar H_{T_0}$ and $\bar H_{T_1}$. We are now ready to state our recursive formula.

\begin{prop}\label{prop:recurs-form-disk}
If $H$ has no saddle points, then 
$$\m N(H) =\begin{cases} 0&\text{ if }H|_b<0\\
 \cN_b&\text{ if }H|_b>0\end{cases}.$$ 
If $H$ has at least one saddle, then
$$
\m N(H) = \begin{cases}\max(0,H(s_0)+\max(\m N(\bar H_{T_{0}}),\m N(\bar H_{T_{1}})     ))&\text{ if }H|_b<0\\
\min(\m  N_{b} , H(s_0)+\max(\m N(\bar H_{T_{0}}),\m N(\bar H_{T_{1}})     ))&\text{ if }H|_b>0
\end{cases}.
$$
\end{prop}


\begin{proof} 
We first assume that $H$ has no saddle point. In this case $H$ has only one critical point $p$ not in $Y$, which is either a maximum or a minimum. The complement of $Y \cup\{p\}$ is foliated by invariant closed curves surrounding $p$, and it is well known that there exists a compactly supported symplectic diffeomorphism $\Psi$ such that $H \circ \Psi$ is a radial function as in Section ~\ref{sec:exampl-radi-hamilt}.
We have already computed the value of $\cN$ in this case; see Equation~\eqref{eq:example-single-mountain-N}.

Let us now assume that $H$ has at least one saddle and is negative on $b$. In that case, we need to prove that 
\begin{equation}\label{eq:2}
\m N(H) =  \max(0, H(s_{0})+\max(\m N(\bar H_{T_{0}}),\m N(\bar H_{T_{1}}))).
\end{equation}
Note that two  fixed points of $\phi_{H}^1$ that lie respectively in $T_0$ and $T_1$ are always unlinked. Such fixed points are also unlinked with the critical point $s_0$ and all the points in $Y$. Moreover, every orbit in $b$ rotates in the positive direction and hence the base $b$ contains no negative fixed point. It follows that the maximal negative unlinked sets (mnus's) of $H$ are exactly the sets of fixed points that are of the form $X=Y\cup \{s_0\}\cup X_0\cup X_1$, where $X_0$ is a set that is maximal for inclusion among the negative unlinked sets of $\phi_{H}^1$ that are included in $T_{0}$, and likewise for $X_{1}$; for short we say that $X_{0}$ and  $X_{1}$ are  mnus's for the restrictions $H|_{T_0}$ and $H|_{T_1}$. 
 As a consequence, 
\begin{equation}\label{eq:1}
\m N(H) = \max \left(0,H(s_0),\inf_{X_0}\sup_{x\in X_0}\cA_H(x),\inf_{X_1}\sup_{x\in X_0}\cA_H(x)\right),
\end{equation}
where infima are taken over the mnus's $X_0$ of $H|_{T_0}$ and the mnus's $X_1$ of $H|_{T_1}$.
For $i=1,2$, the Hamiltonian $\bar H_{T_i}$ has been built so that the mnus's of $\bar H_{T_i}$ are of the form $\bar X_i=Y_i\cup X_i$ where $X_i$ is a mnus's of $H|_{T_i}$ and $Y_i$ is the complement of the support of $\bar H_{T_i}$. We can compute the maximum of the action on such a set: 
\begin{align}\label{eq:3}\begin{split} \sup_{x\in\bar X_i}\cA_{\bar H_{T_i}}(x)&=\max\left(0,\sup_{x\in X_i}\cA_{\bar H_{T_i}}(x)\right)\\ 
&= \max\left(H(s_0),\sup_{x\in X_i}\cA_{H}(x)\right)-H(s_0). 
\end{split}
\end{align}
We then deduce \eqref{eq:2} from \eqref{eq:1} and \eqref{eq:3}.

We now assume that $H$ is positive on $b$;  recall that this means every fixed point in $b$ is a negative fixed point. A non-trivial orbit in $b$ is linked with any other fixed point of $H$ that it encloses. Therefore, the mnus's of $H$ are of two possible types:
\begin{itemize}
\item[(A):] $X=Y\cup \{x\}$  where $x$ is a  fixed point of $\phi_{H}^1$ in $b$. 
\item[(B):] $X=Y\cup \{s_0\}\cup X_0\cup X_1$ where $X_0$ is a mnus of $H|_{T_0}$ and $X_1$ is a mnus of $H|_{T_1}$.
\end{itemize}
Thus, 
\begin{equation*}
\m N(H)=\min\left(\inf_{X\text{ of type A}}\,\sup_{x\in X}\cA_H(x),\inf_{X\text{ of type B}}\,\sup_{x\in X}\cA_H(x)\right).
\end{equation*}
The same argument as in the case $H|_b< 0$ gives: 
\begin{align*}\inf_{X\text{ of type B}}\,\sup_{x\in X}\cA_H(x) &=\max(0, H(s_{0})+\m N(\bar H_{T_{0}}),H(s_{0})+\m N(\bar H_{T_{1}}))\\
&= H(s_0)+\max(\m N(\bar H_{T_{0}}),\m N(\bar H_{T_{1}}).
\end{align*}
The last equality follows from the fact that $H(s_0)>0$.
Note that since $H> 0$ on $b$, the non-trivial orbits in $b$ enclose disks with negative area and hence have positive actions. Therefore, for a mnus of the form  $Y\cup \{x\}$, where $x$ is fixed point of $\phi_{H}^1$ in $b$, the maximum of the action is precisely the action of $x$. Thus,
$$\inf_{X\text{ of type A}}\,\sup_{x\in X}\cA_H(x)=\m N_b,$$
and we get the equality $\m N(H) = \min(\m  N_{b} , H(s_0)+\max(\m N(\bar H_{T_{0}}),\m N(\bar H_{T_{1}})))$, as we wished.
\end{proof}


\subsubsection{Proof of $c=\m N$}
\label{sec:proof-c=N-Morse-disk}

Let $c$ be a formal spectral invariant on the plane $\R^2$. In this section we prove that $c=\cN$ for all Morse functions on the plane. The main step toward this will be to prove of the following proposition.

\begin{prop}\label{prop:recurs-form-c-disk}
If $H$ has no saddle points, then 
$$c(H) =\begin{cases} 0&\text{ if }H|_b<0\\
 \cN_b&\text{ if }H|_b>0\end{cases}.$$ 
If $H$ has at least one saddle, then 
$$
c(H) = \begin{cases}\max(0,H(s_0)+\max(c(\bar H_{T_{0}}),c(\bar H_{T_{1}})     ))&\text{ if }H|_b<0\\
\min(\m  N_{b} , H(s_0)+\max(c(\bar H_{T_{0}}),c(\bar H_{T_{1}})     ))&\text{ if }H|_b>0
\end{cases}.
$$
\end{prop}

Before giving the proof of this proposition, we explain how to deduce from it that $c=\cN$ for Morse functions on the plane.

\begin{proof}[Proof of $c=\cN$ for Morse functions on the plane.] We argue by induction on the number of saddles of $H$. First, it follows immediately from Propositions \ref{prop:recurs-form-disk} and \ref{prop:recurs-form-c-disk} that $\cN$ and $c$ coincide on functions having no saddle points. Then, assume that $c=\cN$ for all Morse functions having at most $k$ saddle points and let $H$ be a Morse function with $k+1$ saddle points. Then, $\bar H_{T_0}$ and $\bar H_{T_1}$ both have at most $k$ saddle points,  hence $c(\bar H_{T_0})=\cN(\bar H_{T_0})$ and  $c(\bar H_{T_1})=\cN(\bar H_{T_1})$. Now using Propositions \ref{prop:recurs-form-disk} and \ref{prop:recurs-form-c-disk} again, we deduce $c(H)=\cN(H)$.
\end{proof}

We now turn to the proof of Proposition \ref{prop:recurs-form-c-disk}. 
We note once and for all that since $\bar H_{T_{0}}$ and $\bar H_{T_{1}}$ are supported on disjoint disks, the max formula applies. Thus in the case when $H$ has at least one saddle, the formula we wish to prove reduces to 
$$
c(H) = \begin{cases}\max(0,H(s_0)+c(\bar H_{T_{0}} + \bar H_{T_{1}})   )&\text{ if }H|_b<0\\
\min(\m  N_{b} , H(s_0)+c(\bar H_{T_{0}}+\bar H_{T_{1}})     )&\text{ if }H|_b>0
\end{cases}.
$$

\begin{proof} The proof will be split into the two cases $H|_b<0$ and $H|_b>0$. 

\medskip
\noindent\textbf{Case 1: $H<0$ on $b$.} 

First assume that $H$ admits no saddle point. Then, $H\leq 0$ hence $c(H)\leq 0$ by monotonicity.
On the other hand, we have $c(H)\geq 0$ by positivity. Thus $c(H)=0$ as claimed.

We now assume that $H$ has a saddle point, so that we can use Notations \ref{notation:bar-H} and let $\eps>0$. Then, it is possible to find at a $C^0$-distance less than $\eps$ from $H$ a function that can be written as a sum $F+\bar H_{T_0}+\bar H_{T_1}$, where $F$ is a smooth non-positive function with exactly two critical values: 0 with critical locus $F^{-1}(0)=Y$, and $H(s_0)$ with critical locus $F^{-1}(H(s_0))=\overline{T_0\cup T_1}$. 
The Lipschitz property of $c$ yields: $$|c(H)-c(F+\bar H_{T_0}+\bar H_{T_1})|\leq\eps.$$
We will prove that 
\begin{equation}\label{eq:4}
c(F+\bar H_{T_0}+\bar H_{T_1})=\max(0,H(s_0)+c(\bar H_{T_0}+\bar H_{T_1})).
\end{equation}
By taking $\eps$ arbitrary small, it follows immediately that the formula of the Proposition holds in this case. 


We consider the 1-parameter family of functions $\sigma\mapsto K_\sigma= \sigma F+\bar H_{T_{0}}+\bar H_{T_{1}}$.
For $\sigma\in[0,1]$, the spectrum of $K_\sigma$ is given by: 
\begin{align*}\spec(K_\sigma) &= \spec(\sigma F) \cup(\sigma H(s_0)+\spec(\bar H_{T_{0}}+\bar H_{T_{1}})).
\end{align*}

Note that $\spec(\sigma F)$ only contains non-positive values. Therefore, by the positivity property of $c$, we know that $c(K_\sigma)$ is either 0 or belongs to $\sigma H(s_0)+\spec(\bar H_{T_{0}}+\bar H_{T_{1}})$. For $\sigma=0$ we have $c(K_0)= c(\bar H_{T_{0}}+\bar H_{T_{1}})\geq 0$. The continuity of $c$ imposes that as long as $\sigma H(s_0)+c(\bar H_{T_{0}}+\bar H_{T_{1}})\geq 0$, one has $c(K_\sigma)=\sigma H(s_0)+c(\bar H_{T_{0}}+\bar H_{T_{1}})$, and $c(K_\sigma)=0$ in the opposite case. In particular for $\sigma=1$, we get \eqref{eq:4} (see Figure \ref{fig:bif-negative-base}). 
\begin{figure}[h!]
\centering
\def\svgwidth{1\textwidth}
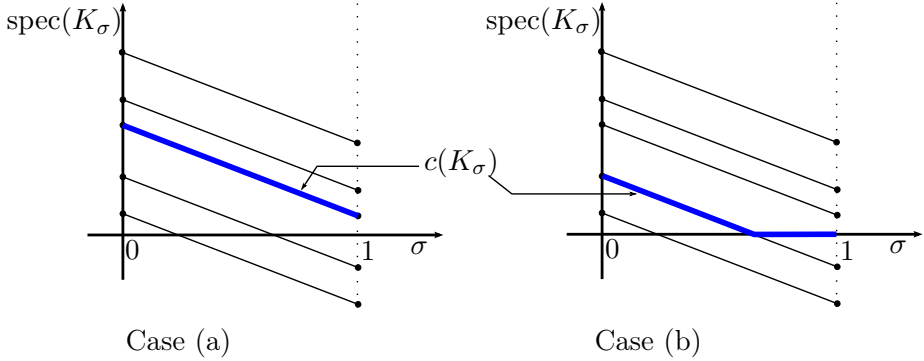
\caption{Bifurcation diagram for the spectrum of the deformation $K_\sigma$: (a) represents the case  $H(s_0)+c(\bar H_{T_{0}}+\bar H_{T_{1}})> 0$ and (b) the opposite case.}
\label{fig:bif-negative-base}
\end{figure}

\bigskip
\noindent\textbf{Case 2: $H>0$ on $b$.} 

This case is much more complicated than the previous one and we will divide its proof into several claims. We will first prove, in claims \ref{claim:upper1} and \ref{claim:upper2}, that $c(H) \leq \min(\m N_{b} , H(s_0)+ c(\bar H_{T_{0}} + \bar H_{T_{1}})     ).$

\begin{claim}\label{claim:upper1} Assume that $H$ is positive on $b$. Then,
$$ c(H) \leq \cN_{b}.$$
\end{claim}

The proof of this claim will require us to compute $c$ explicitly for a class of very simple functions. This is the content of the following lemma.

\begin{lemma}\label{lemma:c-for-simple-bumps} (See Figure~\ref{fig:c-for-simple-bump})	Let $H$ be a radial function defined by $H(x,y)=f(\pi(x^2+y^2))$ for some function $f:[0,+\infty)\to[0,+\infty)$ satisfying for some $A>0$:
\begin{itemize}
\item For all $a\geq A$, $f(a)=0$,
\item $f(0)>A$ and $f'(0)=0$.
\item $f''$ vanishes at a unique point $a_0$ in $(0,A)$,
\end{itemize}
Then $c(H)=f(a_1)+a_1$ where $a_1$ is the unique real number for which $f'(a_1)=-1$ and $f''(a_1)>0$. 
\end{lemma}

\begin{figure}[h!]
\centering
\def\svgwidth{0.5\textwidth}
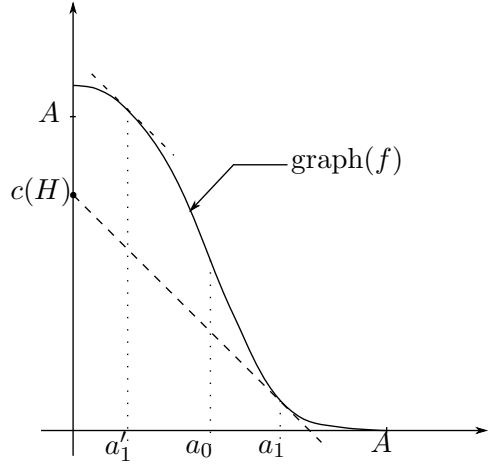
\caption{\label{fig:c-for-simple-bump}Graph of $f$ satisfying the assumptions of Lemma \ref{lemma:c-for-simple-bumps}: $a_{1}'$ and $a_1$ are the only points where $f'=-1$.}
\label{fig:simple-bump}
\end{figure} 

\begin{proof} First note that it follows from the assumptions that $f'<0$ on $(0,A)$. Moreover, $f'$ decreases between $0$ and $a_0$ and then increase between $a_0$ and $A$. Thus $f'$ attains its minimum at $a_0$ and $f'(a_0)<-1$. As explained in Section \ref{sec:exampl-radi-hamilt}, the fixed points of $\phi_H^1$ correspond to the values of $a$ for which, $f'(a)$ is an integer. We see that in our case, for each integer $f'(a_0)<k<0$, we have either 0 or 2 possibilities that we denote $a_k'< a_k$. We have also seen in Section \ref{sec:exampl-radi-hamilt} how to compute the action of such fixed points. In particular, all points have non-negative action, and the critical point 0 has action $>A$. A crucial remark for our purpose is that the point $a_1$ corresponds to the strict minimum of all non-zero actions of the fixed points of $\phi_H^1$. 

Let us first study the case where our function satisfies $f'>-2$. The spectrum of $H$ is made up of four values corresponding to the actions of $0$, $a_1'$, $a_1$ and the points outside the support. The spectral invariant $c$ cannot be reached outside the support by non-degeneracy. Moreover, the action of $a_1'$ is larger than that of 0 and the action of 0 is larger than $A$. Now, the area of the support of $H$ is less than $A$ and thus by the energy-capacity inequality the spectral invariant cannot be reached at any of these two points and therefore is reached at $a_1$ as claimed.

Let us now turn to the general case where $f'$ is not assumed larger than $-2$. Let $\tilde f$ be a function satisfying the assumptions of the lemma and with $\tilde f'>-2$. We leave to the reader to check that there exist a continuous path between $f$ and $\tilde f$ within the functions satisfying the assumptions of the lemma. We consider the bifurcation diagram of spectra obtained from this deformation. For $\tilde{f}$, the spectral invariant is reached at the point $a_1$ (which moves along the deformation but never disappears). Now it follows from the remark made above that the path in the bifurcation diagram associated to $a_1$ has no bifurcation, and therefore that the spectral invariant for $f$ is reached at $a_1$.
\end{proof}

We are now ready to prove the claim.

\begin{proof}[Proof of Claim \ref{claim:upper1}] Let $x_0$ be a fixed point of $\phi_H^1$ in $b$ for which $\cN_b=\cA_H(x_0)$. We need to prove that $c(H)\leq \cA_H(x_0)$. Denote by $\alpha_0$ the area enclosed by the orbit of $x_0$. If $\alpha_0=0$, which means that $x_0$ is the unique critical point of $H$, then $\cA_H(x_0)=\max(H)\geq c(H)$. Assume now that $\alpha_0>0$. 


By conjugating with an area preserving diffeomorphism and using 
 symplectic invariance of $c$, we can assume that  
the 1--periodic orbits of $H$ in the base $b$ are all included in an annulus $b'=\{(x,y)\in\R^2\,|\,\alpha<\pi(x^2+y^2)<\alpha'\}\subset b$  having the same outer boundary as $b$, and that on this annulus $b'$, $H$ has the form of Section \ref{sec:exampl-radi-hamilt}, i.e.  $H(x,y)=f(\pi(x^2+y^2))$, for all $(x,y)\in b'$, for some smooth decreasing function $f:(\alpha,\alpha')\to\R$. Note that for $\pi(x^2+y^2)\geq\alpha'$, one has $H(x,y)=0$. Also note that $f'(\alpha_0)=-1$. To see this, assume that we have $f'(\alpha_0)<-1$, and consider the smallest value $\alpha_1>\alpha_0$ for which $f'(\alpha_1)=-1$. Then, we see easily by considering the diagram of  Figure \ref{fig:action-rotation-number} that the action value associated to $\alpha_1$ is smaller than that of $\alpha_0$. This would then contradict the definition of $\alpha_0$. 

We now choose a radial Hamiltonian $H_1\geq H$ given by $H_1(x,y)=f_1(\pi(x^2+y^2))$, for a function $f_1:[0,+\infty)\to[0,+\infty)$ satisfying the assumptions of Lemma \ref{lemma:c-for-simple-bumps} and the following additional properties (see Figure~\ref{fig:c-upper-bound-1}): $f_1(\alpha_0)=f(\alpha_0)$, $f_1'(\alpha_0)=f'(\alpha_0)=-1$ and $f''(\alpha_0)>0$. By Lemma \ref{lemma:c-for-simple-bumps}, $c(H_1)=\cA_H(x_0)$. As a consequence, we obtain $c(H)\leq \cA_H(x_0)$ using monotonicity.
\end{proof}

\begin{figure}[h!]
\centering
\def\svgwidth{0.8\textwidth}
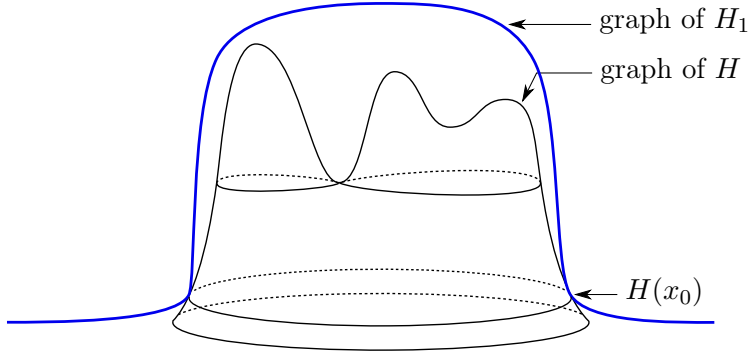
\caption{(Proof of Claim~\ref{claim:upper1}) construction of $H_{1} \geq H$ with $c(H_{1}) = \cN_{b}$}
\label{fig:c-upper-bound-1}
\end{figure}

\begin{claim}\label{claim:upper2}
 If $H$ has at least one saddle and if $H\geq 0$ on $b$, then
$$
c(H) \leq H(s_{0})+c(\bar H_{T_{0}} + \bar H_{T_{1}}). 
$$ 
\end{claim}

\begin{proof} Let $F:\R^2\to\R$ be a smooth compactly supported non-negative function that equals $
H(s_0)$ on the support of $H$, has only $0$ and $H(s_0)$ as critical values, and whose flow has no non-trivial 1--periodic orbit (see Figure~\ref{fig:c-upper-bound-2}). By construction $H\leq F+H_{T_0}+H_{T_1}$, where for $i=1,2$, $H_{T_i}$ is the continuous function that coincides with $H-H(s_0)$ on $T_i$ and vanishes elsewhere. Hence $c(H)\leq c(F+H_{T_0}+H_{T_1})$.    Let $\eps>0$.
According to Lemma \ref{lemma:properties-bar-H}, the functions $\bar H_{T_0}$, $\bar H_{T_1}$ can be chosen so that their $C^0$ distance to respectively $H_{T_0}$ and $H_{T_1}$ is arbitrary small.
The continuity of spectral invariants gives: $$c(H)\leq c(F+\bar H_{T_{0}}+\bar H_{T_{1}})+\eps.$$
By the Lipschitz property we get: 
\begin{align*}c(H)&\leq c(\bar H_{T_{0}}+\bar H_{T_{1}})+\max F+\eps\\
&= H(s_0)+c(\bar H_{T_{0}}+\bar H_{T_{1}})+\eps. 
\end{align*}

\begin{figure}[h!]
\centering
\def\svgwidth{1\textwidth}
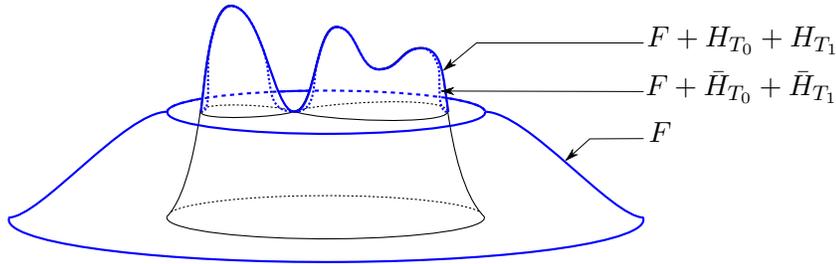
\caption{(Proof of Claim~\ref{claim:upper2}) construction of $F+H_{T_0}+H_{T_1} \geq H$ with $c(F+H_{T_0}+H_{T_1}) \leq H(s_0)+c(H_{T_{0}}+H_{T_{1}})$}
\label{fig:c-upper-bound-2}
\end{figure}

Now according to Remark \ref{remark:c(barH)}, the values of $c(\bar H_{T_{0}})$ and $c(\bar H_{T_{1}})$
are independent of the choices of $\bar H_{T_{0}}$ and $\bar H_{T_{1}}$. This means that $\eps$ can be made arbitrary small and concludes the proof. 
\end{proof}

By Claims \ref{claim:upper1} and \ref{claim:upper2}, we have established the upper bounds required for Proposition \ref{prop:recurs-form-c-disk}. We now turn to the proof of the lower bounds. The next claim achieves the case of Hamiltonians without any saddle point.

\begin{claim}\label{claim:H-pos-no-saddle} Assume that $H$ is Morse, non-negative and has no saddle point. Then, $c(H)=\cN_b$.
\end{claim}

\begin{proof} By non-degeneracy, we have $c(H)>0$. Thus, $c(H)$ is the action of a point in the interior of the support of $H$, hence, by definition, cannot be smaller than $\m N_b$. By Claim \ref{claim:upper1}, we get $c(H)=\m N_b$.
\end{proof}

\noindent\textit{End of the proof of Proposition \ref{prop:recurs-form-c-disk}.} It remains to  establish that  
\begin{equation}\label{eq:5}
c(H)\geq\min(\cN_b,H(s_0)+c(\bar H_{T_0}+\bar H_{T_1})).
\end{equation}

First assume that $c(H)$ is the action of a fixed point in $b$. Then, by definition of $\cN_b$, $c(H)\geq \cN_b$. By Claim \ref{claim:upper1}, we get $c(H)=\cN_b$ which implies \eqref{eq:5}.

Assume now that $c(H)$ is not attained on $b$. Then, by Claim \ref{claim:upper1}, $c(H)<\cN_b$. Similarly to the argument used in Case 1, for all $\eps>0$, we can find at $C^0$-distance less than $\eps$ from $H$ a Hamiltonian of the form $F+\bar H_{T_0}+\bar H_{T_1}$, where $F$ is a smooth non-negative function, with only two critical values: 0, attained on $Y$, and  $h=H(s_0)$ attained on a neighborhood of $\overline{T_0\cup T_1}$. We also choose $F$ close enough to $H$ on $b$ so that it has no  non-trivial 1--periodic orbit with action in $(0,\cN_b)$.  Since $|c(H)-c(F+\bar H_{T_0}+\bar H_{T_1})|\leq\eps$, we have for $\eps$ small enough $c(F+\bar H_{T_0}+\bar H_{T_1})<\cN_b$. This implies in particular that  $c(F+\bar H_{T_0}+\bar H_{T_1})$ is attained in $T_0\cup T_1$. Similarly as in Case 1, we will consider a deformation of the form $K_\sigma=F_\sigma+\bar H_{T_0}+\bar H_{T_1}$, with $F_0$ arbitrarily close to $F$ and $F_1=0$ to prove that 
$c(F_0+\bar H_{T_0}+\bar H_{T_1})=h+c(\bar H_{T_0} + \bar H_{T_1})$, 
which in turn implies the same equality for $c(H)$. Nevertheless, we will have to be slightly more careful in the way we construct it.

\begin{claim}\label{claim:special-deformation} Let $F:\R^2\to\R$ be a smooth non-negative function, with only two critical values: 0 attained on the complement of an open disk $D$, and $h>0$ attained on a smaller closed disk $D'\subset D$. Then, arbitrarily $C^1$-close to $F$, there exists a function $\tilde{F}$ that coincides with $F$ on $D'\cup(\R^2\setminus D)$, and a one parameter family of smooth functions $(F_\sigma)_{\sigma\in[0,1]}$ with $F_0=\tilde{F}$, $F_1=0$ and the two properties:
\begin{enumerate}
\item $F_\sigma$ has only two critical values, 0 and $\max F_\sigma=(1-\sigma)h$,
\item Every Lipschitz function $\delta$ defined on an interval $I\subset[0,1]$ such that $ \delta(\sigma)\in\spec(F_\sigma)$
for all $\sigma\in I$, satisfies $\delta'(\sigma)\geq-h$ almost everywhere.
\end{enumerate}
\end{claim}

\begin{remark}The second property, which may appear rather strange at first glance, simply states that all the  curves in the bifurcation diagram of the deformation $(F_\sigma)_{\sigma\in[0,1]}$ have slope $\geq -h$.\end{remark}

We assume this claim for the time being and postpone its proof to the end of this section.
As explained above, we let $K_\sigma=F_\sigma+\bar H_{T_0}+\bar H_{T_1}$, where $F_\sigma$ is a one parameter family as provided by Claim \ref{claim:special-deformation}. Of course, in our settings, the disks $D'$ and $D$ of Claim \ref{claim:special-deformation} are respectively a neighborhood of $\overline{T_0\cup T_1}$ and the interior of the support of $H$.  For all $\sigma\in[0,1]$, the spectrum of $K_\sigma$ is given by:
\begin{align*}\spec(K_\sigma) = \spec(F_\sigma) \cup((1-\sigma)h+\spec(\bar H_{T_{0}}+\bar H_{T_{1}})).
\end{align*}
The bifurcation diagram $\bigcup_{\sigma\in[0,1]}\{\sigma\}\times\spec(K_\sigma)$ is the union of the horizontal line corresponding to the action 0, parallel lines with slope $-h$ that correspond to the subset $(1-\sigma) h+\spec(\bar H_{T_{0}}+\bar H_{T_{1}})$  and pieces of curves corresponding to the actions of the non-trivial 1--periodic orbits of $F_\sigma$ (see Figure \ref{fig:bif-fancy-def}). 

\begin{figure}[h!]
\centering
\def\svgwidth{0.7\textwidth}
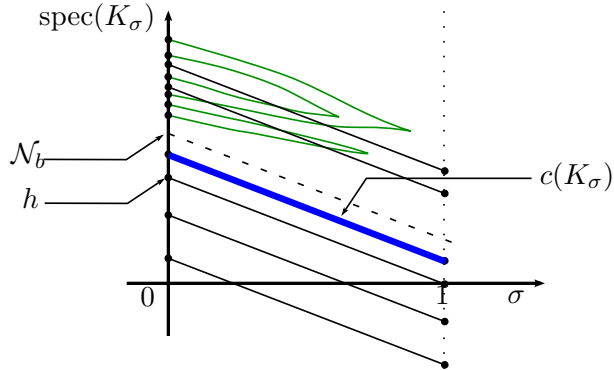
\caption{ The bifurcation diagram of the deformation $K_{\sigma}$.}
\label{fig:bif-fancy-def}
\end{figure}

 These pieces of curves never decrease faster than $-h$, as follows from Property 2 in Claim \ref{claim:special-deformation}. Moreover at $\sigma=0$ these curves are all above the value $\cN_b$. Thus, no curve in the bifurcation diagram that start from a value $>\cN_b$ crosses a line of slope $-h$ with initial value $<\cN_b$. Since $c(K_0)$ is smaller than $\cN_b$ and belongs to the spectrum, $c(K_\sigma)$ remains on the line $(\sigma,c(K_0)-\sigma h)_{\sigma\in[0,1]}$, until it reaches the value 0. After that point, if it exists, the positivity of $c$ implies that it remains constant equal to zero. Now, since $H$ is positive on $b$ and $s_0$ is a non-degenerate saddle, one of the two functions $\bar H_{T_0}$ and $\bar H_{T_1}$ must be positive near the boundary of its support. Thus, it is a consequence of the Max Formula and  Lemma \ref{lemma:positive-base-positive-c} below that 
 
$c(K_1)=c(\bar H_{T_0} + \bar H_{T_1}) = \max(c(\bar H_{T_0}), c(\bar H_{T_1}))$ is positive. 
This implies 
$c(K_0)=h+c(K_1)=h+c(\bar H_{T_0} + \bar H_{T_1})$ 
 and we see that the proof of the Proposition \ref{prop:recurs-form-c-disk} is achieved up to Claim \ref{claim:special-deformation} and Lemma \ref{lemma:positive-base-positive-c} below.
\end{proof} 

\begin{remark}
A consequence of Proposition \ref{prop:recurs-form-c-disk} and the positivity of $c$ is that in the case $c(H)<\cN_b$, which we were just considering, we have the inequality $h\leq\cN_b$. Since $\spec(F_1)=\{0\}$, this implies in particular that all the curves in the bifurcation diagram corresponding to the non-trivial periodic orbits of $F_\sigma$ that start at $\sigma=0$ must die at some point. Moreover, it will follow from the proof of Claim \ref{claim:special-deformation} that $F_\sigma$ can be constructed so that no birth occurs in its bifurcation diagram. This is illustrated on Figure \ref{fig:bif-fancy-def}.
\end{remark}

\begin{lemma}\label{lemma:positive-base-positive-c}
If $H$ is positive on $b$, then $c(H)>0$.
\end{lemma}

\begin{proof} Up to conjugation by an area preserving diffeomorphism, $H$ is larger than a smooth radial Hamiltonian of the form of Section  \ref{sec:exampl-radi-hamilt}: $F_1(x,y)=f_1(\pi(x^2+y^2))$, for all $(x,y)\in\R^2$, with $f_1:[0,+\infty) \rightarrow \R$ having a simple profile: for some real numbers $0<a_0<a_1<a_2$, it is strictly increasing on $[0,a_1]$, strictly decreasing on $[a_1,a_2]$, $f_1(a_0)=0$ and $f_1$ vanishes on $[a_2,+\infty)$ (note that $f_{1}(0)$ may be negative).   By monotonicity, we only have to verify that $c(F_1)>0$ to prove $c(H)>0$.

Let $F_0$ be a smooth non negative approximation of the function $\max(0,F_1)$. By non-degeneracy, $c(F_0)>0$. Now Let $F_\sigma$ be a smooth decreasing deformation from $F_0$ to $F_1$. We may also assume that all the functions $F_\sigma$ are radial, hence of the form $F_\sigma(x,y)=f_\sigma(\pi(x^2+y^2))$ and that all the functions $f_\sigma$ are increasing on $[0,a_0]$ and coincide with $f_1$ on $[a_0,+\infty)$. All the orbits of $F_\sigma$ located in the circle of area $a_0$ have negative action. Thus, the non-negative part of the spectrum remains unchanged along the deformation. As a consequence, using spectrality and continuity we obtain $c(F_1)=c(F_0)>0$, and thus $c(H)>0$. 
\end{proof}

There only remains to construct the deformation of Claim \ref{claim:special-deformation}.

\begin{proof}[Proof of Claim \ref{claim:special-deformation}]
Up to conjugation with an area preserving diffeomorphism, we may assume that the annulus $D\setminus D'$ is given in coordinates by $\{(x,y)\in\R^2\,|\,\alpha<\pi(x^2+y^2)<\beta\}$ and $F$ is a radial Hamiltonian, as in Section \ref{sec:exampl-radi-hamilt}: for all $(x,y)\in \R^2$, $F(x,y)=f(\pi(x^2+y^2))$, where $f=h$ on $[0,\alpha]$, $f$ decreases on $(\alpha,\beta)$ and $f=0$ on $[\beta,+\infty)$. We denote $g=-f'$. We will construct the deformation $F_\sigma$ as radial functions $F_\sigma(x,y)=\int_{\pi(x^2+y^2)}^{+\infty}g_\sigma(u)du$, where $g_\sigma$ will be a deformation such that $g_1 =\tilde{g}$ where $\tilde{g}$ is a function arbitrarily close to $g$, $g_0 = 0$ and  $g_{\sigma}$  vanishes on $[0,\alpha]$ and $[\beta,+\infty)$ for all $\sigma$. Recall from Section \ref{sec:exampl-radi-hamilt} that the spectrum of $F_\sigma$ is calculated by considering the points where $g_\sigma$ is an integer. 

We first perturb $g$ so that the set of $s\in(\alpha,\beta)$ such that $g(s)$ is an integer is finite. Then, we let $\tilde{g}$ be a smooth $C^0$-perturbation of $g$ obtained by flattening $g$ in a small neighbourhood of all the points $s$ where $g(s)$ is an integer (See Figure \ref{fig:fancy1}). The function $\tilde{F}$ is then defined as $$\tilde{F}(x,y)=\int_{\pi(x^2+y^2)}^{+\infty}\tilde{g}(u)du.$$

\begin{figure}[h!]
\centering
\def\svgwidth{0.9\textwidth}
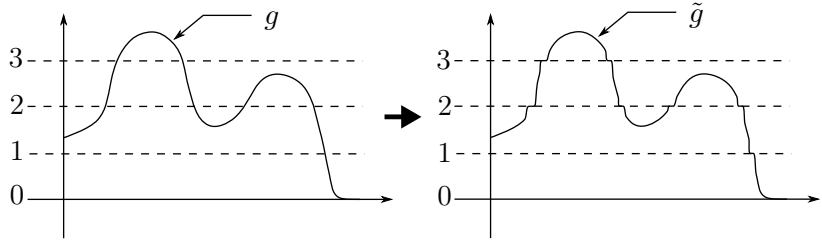
\caption{The deformation from $g$ to $\tilde{g}$.}
\label{fig:fancy1}
\end{figure}

To construct the deformation $g_\sigma$ from $\tilde g$ to $0$, we first introduce the following set of notations. Let $N$ be the integer part of $\max\tilde g$. For all integers $k=0,1,\ldots,N+1$, we set the truncated functions $\gamma_k=\min(g,k)$ and $\delta_k=\gamma_{k+1}-\gamma_{k}$. Clearly, $\gamma_{N+1}=\tilde{g}$ and $\gamma_0=0$, hence $\tilde{g}=\sum_{k=0}^{N}\delta_k$. 
The effect of the perturbation $\tilde{g}$ is that each function $\gamma_k$, $\delta_k$ is smooth whereas an analoguous definition for $g$ would only yield continuous functions.
We also set $h_k=\int_0^{+\infty}\delta_k(u)du$, so that $h=\sum_{k=0}^Nh_k$. Finally, let $\tau_k=\frac1h(h_k+\ldots+h_N)$. In particular, $0=\tau_{N+1}<\tau_N<\ldots<\tau_1<\tau_0=1$.

We can now define the deformation:

\begin{equation}\label{eq:12} g_{\sigma}(s)=\gamma_k(s)+\tfrac h{h_k}(\tau_{k}-\sigma)\delta_k(s),
\end{equation}
for all $k=0,\ldots,N$, $\sigma\in[\tau_{k+1},\tau_k)$, and $s\in[0,+\infty)$ (see Figure \ref{fig:fancy2}).
Let us check that this deformation suits our needs.

\begin{figure}[h!]
\centering
\def\svgwidth{0.9\textwidth}
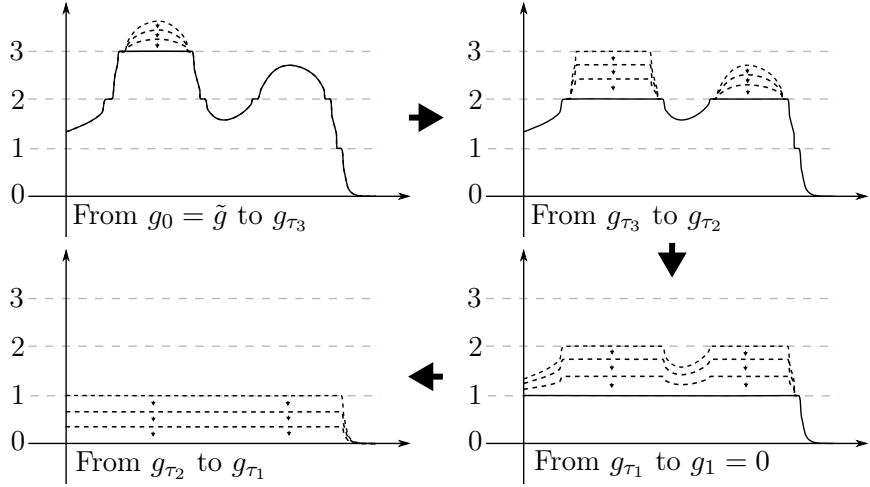
\caption{The deformation from $\tilde{g}$ to 0 via $g_{\tau_3}$,  $g_{\tau_2}$ and  $g_{\tau_1}$.}
\label{fig:fancy2}
\end{figure}

First, note that $\sigma\mapsto g_\sigma$ is continuous on $[0,1]$ in the $C^0$-topology. This follows from the fact that when $\sigma$ evolves from $\tau_{k+1}$ to $\tau_k$, the factor $\tfrac{h}{h_k}(\tau_{k}-\sigma)$  evolves from 1 to 0, and so $g_{\sigma}$ evolves from $\gamma_{k+1}$ to $\gamma_k$. As a consequence of this continuity, Property 1 in Claim \ref{claim:special-deformation} can be checked by considering separately each interval of deformation $(\tau_{k+1}, \tau_k)$. The maximum of $F_\sigma$ is the total integral $\int_{0}^{+\infty}g_\sigma(u)du$. By equation \eqref{eq:12}, its rate of decrease on the interval  $(\tau_{k+1}, \tau_k)$ is $$\tfrac h{h_k}\int_0^{+\infty}\delta_k(u)du=h.$$
This proves the first property.

As the first one, the second property in Claim \ref{claim:special-deformation} only needs to be established on each interval $(\tau_{k+1}, \tau_k)$. As we already recalled, it follows from Section \ref{sec:exampl-radi-hamilt} that the spectrum of $F_\sigma$ can be computed by only considering the points where $g_{\sigma}$ is an integer $\ell$. It turns out that along each interval $(\tau_{k+1}, \tau_k)$ and for each integer $\ell$, the set of these points remains unchanged. Moreover, for each such point $s$, the action is obtained as the area of the shaded region in Figure \ref{fig:action-rotation-number}. This area has two parts, a rectangle part whose area is $\ell s$ and an integral part whose area is $\int_s^{+\infty}g_{\sigma}(u)du$. Along the deformation interval $(\tau_{k+1}, \tau_k)$, the rectangle part of the area remains constant, whereas the integral part decreases  at the rate 
$$ \tfrac{h}{h_k}\int_s^{+\infty}\delta_k(u)du\leq h.$$
As a consequence, over the interval  $(\tau_{k+1}, \tau_k)$, the action spectrum of $F_\sigma$ is a finite union of non-increasing smooth curves  whose slopes are never smaller than $-h$. Property 2 of Claim \ref{claim:special-deformation} follows.
\end{proof}

\subsection{Theorem \ref{theo.axiomatic-c=n} for Morse functions on closed surfaces of genus $\geq 1$} \label{sec:c=N-Morse-surface}

In this section, we prove the equality $c=\cN$ for Morse functions on closed surfaces of positive genus. This is done in two steps. We first establish a formula  which reduces the problem of computing $\m N$ to computations for Hamiltonians supported in disks; see Proposition \ref{prop:form_N_general} in  Section \ref{sec:formula-n-positive-genus}.  We  then show that this formula is also satisfied by $c$; see Proposition \ref{prop:c=max-on-disks} in Section \ref{sec:proof-c=N-Morse-surface}.

\subsubsection{A formula for $\m N$}\label{sec:formula-n-positive-genus}
Let $\Sigma$ denote a closed surface of positive genus and consider a Morse function $H : \Sigma \rightarrow \R$. The goal of this section is to present a formula which reduces computing $\m N(H)$ to computing $\m N$ on  the restriction of $\phi^1_{H}$ to a collection of invariant disks; see Proposition \ref{prop:form_N_general} below.

Let $s$ be a saddle point of $H$ and denote by $C(s)$ the connected component of $H^{-1}(H(s))$ which contains $s$.  Note that $C(s)$ is a circle pinched at $s$.  Equivalently, we can view $C(s)$ as a union of two circles $C_0(s), C_1(s)$ whose intersection is $\{s\}.$  We will say that the saddle $s$ is \emph{essential} if at least one of these two circles is not contractible in $\Sigma.$  The following proposition describes a decomposition of the surface $\Sigma$ obtained by cutting it along the pinched circles of essential saddles.  We postpone the proof to the end of this section.

\begin{prop}\label{prop:cut_ess_sadd} (See Figure~\ref{fig:essential-saddles})
Let $\Sigma'$ be the open and disconnected surface obtained from $\Sigma$ by removing $C(s)$ for each essential saddle $s$ of $H$. Let $S$ denote a connected component of $\Sigma'$.  Then,
\begin{enumerate}
\item $S$ is either a disk or a cylinder.
\item If $S$ is a cylinder then  $\phi_{H}^1$ has no contractible fixed point in $S$. 
\item The map $i_*: \pi_1(S) \rightarrow \pi_1(\Sigma)$ induced by inclusion is injective.  
\end{enumerate}  
\end{prop}

\begin{figure}[ht!]
\centering
\includegraphics[width=5cm,height=6cm]{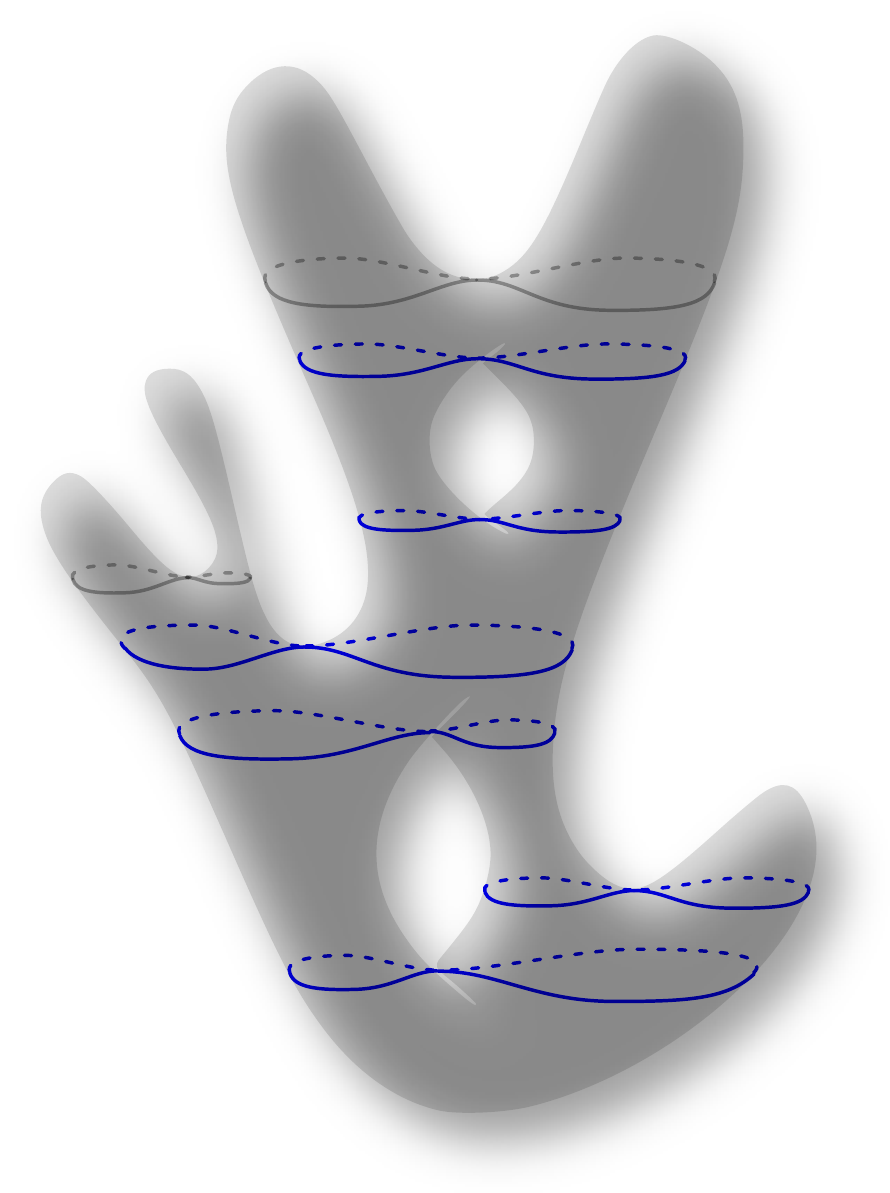}
\hspace{2cm}
\includegraphics[width=5cm,height=6cm]{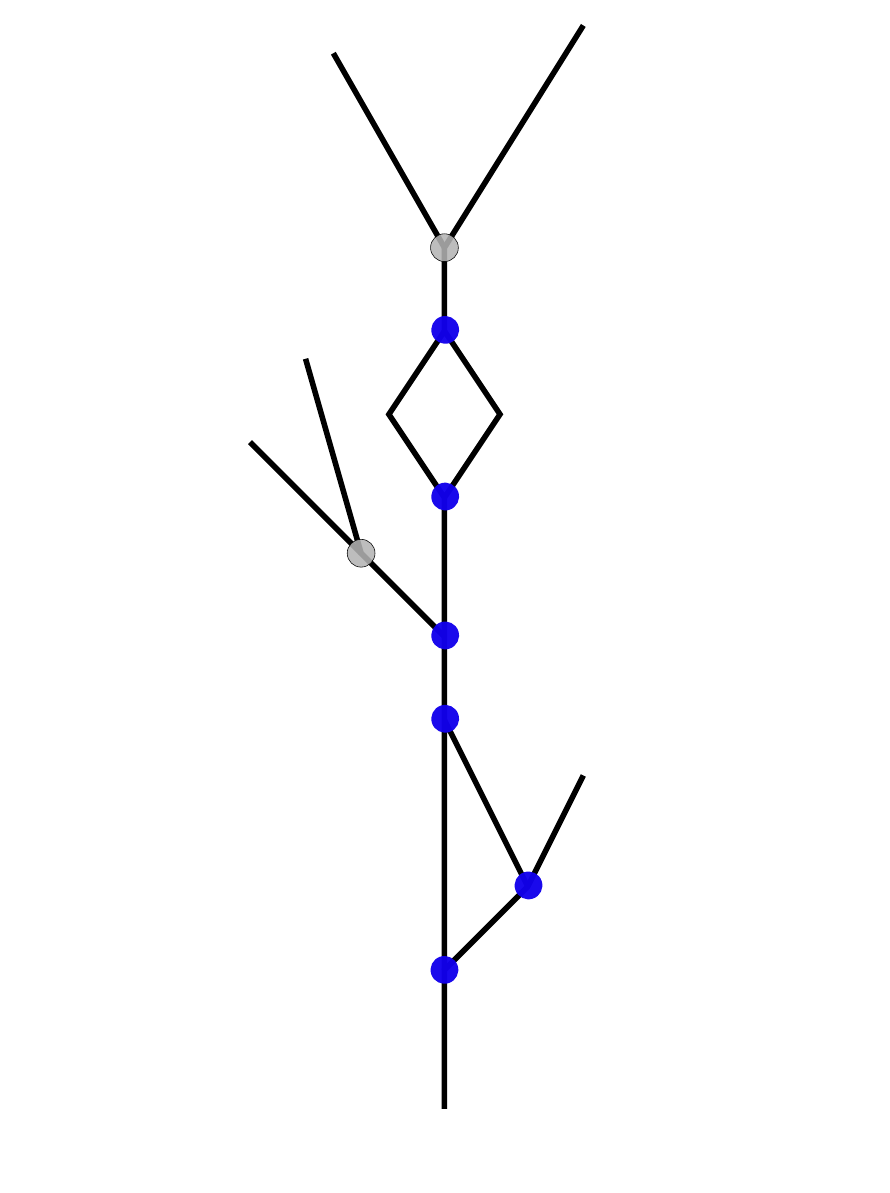}
\caption{\label{fig:essential-saddles}A typical Hamiltonian function (the $z$ coordinate) on a genus two surface, with six essential saddles which decompose the surface into four disks and seven essential annuli. On the right, the corresponding Reeb graph, whose vertices are the critical points and whose edges are the connected component of the complement of the union of the $C(s)$'s. The ``free ends'' of the Reeb graph corresponds to the components of $\Sigma'$ that are disks. Essential saddles correspond to vertices which belong to the ``core graph'', the subgraph obtained by removing the free ends.}
\label{fig:Reeb-graph}
\end{figure}

Define  $\cD$ to be the set of  all the disks obtained via the above decomposition of $\Sigma$.
Note that $H$ is constant on the boundary of each of these disks. For every disk $D \in \cD$ let $\bar H_{D} \in \m E_D$ be an appropriate smoothing of $H|_D - H(\partial D)$  defined exactly as in Notation \ref{notation:bar-H}. 
We can now present the main result of this section.

\begin{prop} \label{prop:form_N_general}
$$
\m N(H) = \max \{ H(\partial D) + \m N(\bar H_{D}) : D \in \cD \}.
$$
\end{prop}
\begin{proof} 
Let $\m S$ denote the set of critical points of $H$ that do not belong to the union of the open disks $D \in \cD$. According to Proposition~\ref{prop:cut_ess_sadd}, this is exactly the set of essential saddles.
Using Corollary~\ref{corol:unlinked-sets-autonomous} and Proposition~\ref{prop:cut_ess_sadd}, we get the following description. The mnus's for $H$ are the sets of the form
\begin{align} \label{eq:mnuss_formula}
X = \m S \cup \bigcup_{D \in \cD} X_{D}
\end{align}
where $X_{D}$ is a subset of $D$ which is negative, unlinked, and maximal for inclusion among the negative unlinked subsets of $D$.
Similarly, according to the definition of $\bar H_{D}$, the mnus's for $\phi^1_{\bar H_{D}}$ are the sets 
$$
X_{D} \cup \nu_{D}
$$
where $\nu_{D}$ is the connected component of $\Sigma \setminus D$ in the set $\bar H_{D}^{-1}(0)$, and $X_{D}$ is as above.
Now the properties of $\bar H_{D}$, as expressed in Lemma~\ref{lemma:properties-bar-H}, entail that
$$
\sup_{x \in X_{D} \cup \{s_{D}\}} \cA_{H}(x) = H(\partial D) + \sup_{x \in X_{D} \cup \nu_{D}} \cA_{\bar H_{D}} (x)
$$
where, for each $D \in \cD$, the point  $s_{D}$ is the unique saddle point of $H$ in $\partial D$. From equality~\ref{eq:mnuss_formula} we deduce that 
$$
\m N (H) = \max \left( \{ H(\partial D) + \m N(\bar H_{D}) : D \in \cD \} \cup \{ H(s) : s \in \m S \} \right) .
$$
We would like to get rid of the last term of the union. Let $s \in \m S$ be an essential saddle. Consider first the case where $s$ does not belong to the boundary of any disk $D \in \cD$ (in the example of Figure~\ref{fig:Reeb-graph}, $\m S$ contains 2 such elements). According to Proposition~\ref{prop:cut_ess_sadd}, $s$ is in the boundary of three essential annuli, and the second point of the proposition entails that there exists at least another essential saddle $s'$ on the other boundary of one of the three annuli such that $H(s) < H(s')$. In the opposite case when $s$ belongs to the boundary of some $D \in \cD$, note that $H(s) \leq H(\partial D) + \m N(\bar H_{D})$, indeed $H(s) = H(\partial D)$ and $\m N(\bar H_{D}) \geq 0$.
 From these considerations it follows that in the last formula for $\m N(H)$, the maximum is always attained in the first term of the union, and we get Proposition~\ref{prop:form_N_general}.

\end{proof}

\begin{proof}[Proof of Proposition \ref{prop:cut_ess_sadd}]

Let $S$ be a component of $\Sigma'$. Denote by $\chi(S)$, the Euler characteristic of $S$.  Recall that
$$\chi(S) = 2 - 2g(S) - N_b(S),$$
where $g$ is the genus of $S$ and  $N_b$ denotes the number of boundary components of $S$.  Since $S$ is a surface with boundary we see immediately that $\chi(S) \leq 1.$  We will suppose for the rest of the proof of that $S$ is not a disk which implies that $ \chi(S) \leq 0.$

If $s$ is a saddle point in $S$ then it is not essential, and each of the loops of $C(s)$ bounds a disk in $\Sigma$. Let $D$ be a disk bounding one of the two loops of $C(s)$, say $C_0(s)$. We claim that $D$ is included in $S.$ Indeed, otherwise, the interior of $D$ meets the boundary of $S$, and hence it intersects $C(s')$ for an essential saddle $s'$. But the boundary of $D$, i.e. $C_0(s)$, is contained in the interior of $S$ and hence it does not meet $C(s')$. By a connectedness argument, $C(s')$  is entirely included in $D$ and hence $s'$ is not an essential saddle; contradiction.  

Now let $S'$ be the surface obtained from $S$ by removing a neighborhood of the disks bounding the two loops of $C(s)$ for each saddle $s$ in $S$.   We have obtained $S'$ from $S$ by removing a number of disks from $S$ and hence $\chi(S')  \leq \chi(S) $.   Now, let $N_{max}(S'), N_{min}(S'), N_{sad}(S')$ denote number of maxima, minima, and saddles of $H$ inside $S'$.  By the Poincar\'e-Hopf theorem, $$\chi(S') = N_{max}(S') -  N_{sad}(S') + N_{min}(S').$$  The function $H$ has no saddles in $S'$ and so $N_{sad}(S') = 0$.  Therefore,  $\chi(S') \geq 0.$ Since $\chi(S)\leq 0$, we see immediately that $\chi(S') = \chi(S) = 0$.  We conclude that $S$ is a cylinder and that there are no saddles of $H$ in $S$.  Another application of the Poincar\'e-Hopf theorem implies that $H$ has, in fact, no critical point inside $S$.  Of course, this implies that the time-1 map $\phi^1_H$ has no non contractible fixed point in $S$ as the disk bounding a 1--periodic orbit would necessarily contain a critical point of $H$.

It remains to prove that $i_* : \pi_1(S) \rightarrow \pi_1(\Sigma)$ is  injective. Note that  $\pi_1(S)$ is generated by either one of the boundary components of $S$ and so it is sufficient to prove that these two loops are not contractible in $\Sigma$.   These boundary components, say $C, C',$  are loops associated, as described earlier, to two essential saddles $s, s',$ respectively. We will first show that $s, s'$ are distinct: Indeed, if $s = s'$ then $H$ takes the same value on the two boundary components of $S$ and this would force $H$ to have a critical point inside $S$.  Next, for a contradiction suppose that $C$ is contractible in $\Sigma$.  Let $D$ be a disk bounding $C$. Then, $D' = D \# S$ is a disk bounding $C'$.  The disk $D'$ meets the pinched circle $C(s)$ but the boundary of $D'$, i.e. $C'$, is disjoint from $C(s)$.  We see that $C(s)$ is entirely contained in $D'$, which contradicts the fact that the saddle $s$ is essential.  This completes the proof.
\end{proof}

%


\subsubsection{Proof of $c=\m N$}\label{sec:proof-c=N-Morse-surface}

The main step here is to prove that $c$ is determined by its value on functions supported on the  disks delimited by essential saddles, in the same way as $\cN$.
We use the notations of Proposition \ref{prop:form_N_general}.

\begin{prop}\label{prop:c=max-on-disks}
$$c(H)=\max\{H(\partial D)+c(\bar H_D)\,|\,D\in\m D\}.$$
\end{prop}

\begin{proof}
It will be convenient to assume that $H$ is positive. This can be assumed without loss of generality thanks to the shift property of formal spectral invariants. 

We first claim that the only contractible periodic orbits of $H$ (of any length) inside $ S = \Sigma\setminus\bigcup_{D \in \m D} D$ are the essential saddles of $H$.    Indeed, according to Proposition \ref{prop:cut_ess_sadd}, all the critical points of $H$ in the surface $S$ are essential saddles, and if we remove from $S$ the sets $C(s)$ for all essential saddles $s$, we are left with a collection of disjoint cylinders which contain no critical points of $H$. If $S_{0}$ is such a cylinder, it is foliated by periodic orbits of the Hamiltonian parallel to the boundary curves of $S_{0}$.  It follows from the third point of Proposition \ref{prop:cut_ess_sadd} that $S_{0}$ contains no contractible periodic orbit of $H$.

Arbitrarily close to $H$, one can find a Hamiltonian $H'$ of the form $H'=F+\sum_{D\in\m D}\bar H_D$, where $F:\Sigma\to\R$ is a smooth positive function which is constant and equal to $H(\partial D)$ on each of the disks $D\in\m D$ and which coincides with $H$ on all of $S$ except near the boundary of $S$.
We will prove that 
\begin{equation}\label{eq:7} c(H') = \max\{H(\partial D)+c(\bar H_D)\,|\,D\in\m D\}.
\end{equation}
By continuity of $c$, the result then follows.

We will now use the symplectic contraction principle to build a deformation of $H'$ as follows.  Note that we can not treat the disks $D \in \m D$ as Liouville domains as  the boundary of $D \in \m D$  could be a pinched circle.  For each $D\in\m D$, we choose an open disk $\nu_D$ with a smooth boundary which is compactly contained in $D$ and which contains the support of $\bar H_D$.
It is clear that $\bigcup \nu_D$ is a Liouville domain.  Let $\xi$ be a Liouville vector field for $\bigcup \nu_D$.  For $s\in(-\infty,0]$, denote by  $A_s:\bigcup  \nu_D\to \bigcup  \nu_D$ the negative flow of $\xi$ and set
$$H'_s(x)=
\begin{cases}e^s F(x) & \text{if }  x\in \Sigma\setminus A_s(\bigcup \nu_D),\\
 e^s H(\partial D) +  e^s \bar H_D(A_s^{-1}(x)) &  \text{if }  x\in A_s(\nu_D),\ D \in\m D.
\end{cases}
$$

Note that $H_0'=H'$. Moreover, since (by the discussion in the first paragraph of the proof)  $F$ has no non-trivial contractible periodic orbits in $ S = \Sigma\setminus\bigcup D$,  the spectrum of $H'_s$ satisfies: $\spec(H'_s)=e^s\spec(H')$.
The spectrality and continuity of spectral invariants thus yield $c(H'_s)=e^sc(H')$ for all $s\in(-\infty,0]$. Similarly, $c(\bar H_{D,s})=e^sc(\bar H_D)$ where $\bar H_{D,s}$ is defined by 
$$\bar H_{D,s}(x)=
\begin{cases}0 &\text{if }  x\in \Sigma\setminus A_s(\bigcup \nu_D),\\
 e^s \bar H_D(A_s^{-1}(x)) &  \text{if }  x\in A_s(\nu_D),\ D\in\m D.
\end{cases}
$$
Therefore, \eqref{eq:7} will be proved if we prove that the following equality holds for some, and hence all, $s\in(-\infty,0]$: 
\begin{equation}\label{eq:8} c(H'_s)= \max\{e^sH(\partial D)+c(\bar H_{D,s})\,|\,D\in\m D\}.
\end{equation}

We will prove this in two steps.  

\medskip

\noindent \textbf{Step I:} We prove that $c(H'_s)\geq  \max\{e^sH(\partial D)+c(\bar H_{D,s})\,|\,D\in\m D\}, $ for all $s \in (-\infty,0]$.  

Since $F$ is positive, there exists a function $G\leq F$, supported in $\bigcup D$, that coincides with $H(\partial D)$ on each of the sets $\nu_D$ and with no critical points other than those outside its support and those in $\bigcup \nu_D$.
Pick $s$ close enough to $-\infty$ so that $e^sG$ has no non-trivial periodic orbit of length $1$. By monotonicity, $$c(H'_s)=c\left(e^sF+\sum_{D\in\m D}\bar H_{D,s}\right)\geq c\left(e^sG+\sum_{D\in\m D}\bar H_{D,s}\right).$$ Since $G$ is supported in the union of the disks $D\in\m D$ we may apply the max formula of Definition \ref{def:formal_spec}. If $G_D$ stands for the component of $G$ supported in $D$, we get:
$$c\left(e^sG+\sum_{D\in\m D}\bar H_{D,s}\right)=\max\{c(e^sG_D+\bar H_{D,s})\,|\,D\in\m D\}.$$
Now we claim that $c(e^sG_D+\bar H_{D,s})=e^sH(\partial D)+c(\bar H_{D,s})$. Together with the previous inequality it yields $c(H'_s)\geq  \max\{e^sH(\partial D)+c(\bar H_{D,s})\,|\,D\in\m D\} $ for $s$ sufficiently close to $- \infty$ which of course implies that the inequality holds for all $s \in (-\infty, 0]$.

To prove our claim we need to distinguish between two cases. To simplify the  notations, we name the functions involved by $h=\bar H_{D,s}$, $g=e^sG_D$ and let $\kappa$ be the real number $e^sH(\partial D)$, so that our claimed equality is now
\begin{equation}\label{eq:10}c(g+h)=\kappa+c(h).
\end{equation}
To summarize the settings, $g+h\geq 0$, $g$ has no non-trivial 1--periodic orbits and has only two critical values 0 and $\kappa$, the critical locus $g^{-1}(\kappa)$ contains the open set $\nu_D$ which contains the support of $h$.

First case: $c(h)>0$. In that case, we
consider the deformation $(K_u)_{u\in[0,1]}$ defined by $K_u=ug+h$. The spectrum of $K_u$ is by construction the union of $\{0\}$ and a shifted part $u\kappa+\spec( h)$. By monotonicity, $c(K_u)$ increases with $u$. Moreover, since $c(K_0)=c(h)>0$, we conclude that $c(K_u)$ never vanishes along the deformation and hence belongs to the shifted part of the spectrum. By continuity, it follows that for all $u$, $c(K_u)=u\kappa+c(K_0)$. Taking $u=1$, we get exactly Equation \eqref{eq:10}.
 
Second case: $c(h)=0$. In that case we can find arbitrarily $C^2$-close to $h$ a function $\tilde{h}$
satisfying $c(\tilde{h})>0$. Indeed, take $f$ to be a $C^2$-small non-negative bump function whose support is included in a disk contained in $\nu_D$ that does not intersect a disk containing the support of $h$. Let $\tilde{h}=f+h$. By the max formula and the non-degeneracy property, $c(\tilde h)=\max(c(f),c(h))=c(f)>0$.
Now we may apply the first case to $\tilde{h}$ to obtain $c(g+\tilde h)=\kappa+c(\tilde h)$. Equation \eqref{eq:10} then follows by continuity of $c$ . 

\noindent \textbf{Step II:} We prove that $c(H'_s)= \max\{e^sH(\partial D)+c(\bar H_{D,s})\,|\,D\in\m D\}, $ for all $s \in (-\infty,0]$. 
 Let $G$ be as in Step I.
Once again, we pick $s$ close enough to $-\infty$ so that $e^sG$ has no non-trivial periodic orbit of length $1$.

We will now show that 
$$c(H'_s)=c\left(e^sF+\sum_{D\in\m D}\bar H_{D,s}\right) = c\left(e^sG+\sum_{D\in\m D}\bar H_{D,s}\right).$$

To simplify the notation let $f = e^s F, g = e^s G, h_D= \bar H_{D,s}.$  We want to show that $$c(f + \sum_D h_D) = c(g+ \sum_D h_D).$$  


Consider the deformation $K_u = u g + (1-u) f + \sum_D h_D$, where $u\in [0,1]$.  

Note that 
\begin{itemize}
\item on $\nu_{D}$ we have $g=f =e^s H(\partial D)$, hence $K_{u} = e^s H(\partial D) + h_{D}$,
\item on $D \setminus \nu_{D}$, $f$ is still constant, hence $K_{u} = (1-u) e^s H(\partial D) + u g$,
\item on $S$, $K_{u} = (1-u)f$ since $g= \sum h_{D}=0$.
\end{itemize}
Since $ug$ and $(1-u)f$ have no $1$-periodic orbits except their critical points, we get
 $$\spec(K_u) = (1-u)\, \spec(f) \cup\,\bigcup_D \; ( e^s H(\partial D) + \spec(h_D) ).$$

Note that $K_0 = f + \sum_D h_D$ and $K_1 = g + \sum_D h_D$.  

\begin{claim}$c(K_0)  \geq \max(\spec(f))$. 
\end{claim}
\begin{proof}[Proof of the claim]
Since by positivity $c(\bar{H}_{D,s}) \geq 0$, it follows from Step I that $c(f + \sum_D h_D ) \geq \max\{e^s H(\partial D): D \in \m D\}.$  As for the other values in $\spec(f)$, they are all smaller than $\max\{e^s H(\partial D): D \in \m D\}.$  This is because, by Proposition \ref{prop:cut_ess_sadd}, the Morse function $H|_S$, where $S = \Sigma \setminus \bigcup \nu_D$, has no local maxima in the interior of $S$ and hence it must attain its maximum on a boundary component of the surface $S$.
\end{proof}

We will now use the above claim to finish the proof of Step II.

First, assume that $c(K_0) > \max(\spec(f))$.  It follows from the above description of $\spec(K_u)$ that the bifurcation diagram $\bigcup_{u \in [0,1]} \{u\}\times \spec(K_u)$ consists of straight lines with slope 0 corresponding to  elements of the form $( e^s H(\partial D) + \spec(h_D) )$ and decreasing lines corresponding to elements of $(1-u) \; \spec(f)$.  It follows from the above claim that the decreasing lines in $(1-u) \;  \spec(f)$ never intersect the line with slope zero corresponding to $c(K_0)$.  Hence, by continuity of $c$, $c(K_u) = c(K_0)$ for all $u \in [0,1]$.  In particular, $c(K_1) = c(K_0)$.

Next, suppose that $c(K_0) = \max(\spec(f))$.   As in the last paragraph of Step I, by making a $C^2$--small perturbation we can ensure that $c(\bar H_{D}) >0$ for each $\bar H_{D}$ and therefore  $\max\{H(\partial D)+c(\bar H_{D})\,|\,D\in\m D\} > \max\{H(\partial D)|\,D\in\m D\}.$  It follows from Step I that in fact $c(H'_s) >  \max\{e^s H(\partial D)|\,D\in\m D\}.$ Hence, we may in fact assume that  $c(K_0) >  \max(\spec(f))$.  
\end{proof}

\begin{proof}[Proof of Theorem \ref{theo.axiomatic-c=n} for Morse functions on higher genus surfaces.] Let $H$ be a Morse function on $\Sigma$. Cut $\Sigma$ along all essential saddles as described in Section \ref{sec:formula-n-positive-genus} and let $\m D$ be the family of disks obtained. For every disk $D$, the function $\bar H_D$ has only non-degenerate critical points in the interior of its support. Since we already proved Theorem \ref{theo.axiomatic-c=n} for Morse functions on the plane, Lemma \ref{lemma:restricted-N-and-c} below implies that $c(\bar H_D)=\cN(\bar H_D)$.

As an immediate consequence of Propositions \ref{prop:c=max-on-disks} and \ref{prop:form_N_general} we get $c(H)=\cN(H)$.
\end{proof}

 The following lemma compares the invariant $\cN: C^{\infty}([0,1]\times \R^2) \to \R$ with its sibling $\cN: C^{\infty}([0,1]\times \Sigma) \to \R$.  We will denote the first one by $\cN_{\R^2}$ and the latter by $\cN_{\Sigma}$.
\begin{lemma}\label{lemma:restricted-N-and-c} Let $\iota:a\D^2\to \Sigma$ be an area preserving embedding of the standard disk of area $a$ into $\Sigma$. Let $D=\iota(a\D^2)$ be its image. 
\begin{itemize}
\item For every function $H$ with support in $a\D^2$, $\cN_{\Sigma}(H')=\cN_{\R^2}(H)$, where $ H' = H \circ \iota^{-1}$ on $ D$ and $ H'=0$ elsewhere.
\item Let $c: C^{\infty}([0,1]\times \Sigma) \rightarrow \R$  be a formal spectral invariant. Then, the map $\iota^*c :C^{\infty}([0,1]\times \R^2) \to \R$ defined for every function $H$ supported in $a\D^2$
 by $
\iota^*c(H)= c(H')$,
extends to a formal spectral invariant on $\R^2$.
 \end{itemize}

\end{lemma}

\begin{proof} The first part of the lemma follows immediately from Corollary~\ref{cor:mnus_closed_surface}.
Let us now turn to the second part of the lemma. Denote $\bar c=\iota^*c$. since the spectrum of $H$ is the same as the spectrum of $H'$, it is clear that $\bar c$ defines a formal spectral invariant on the set of Hamiltonians supported in $a\D^2$.  
To extend $\bar c$ to every function on $\R^2$ we use the symplectic contraction principle. Let $\zeta$ be the standard Liouville vector field on $\R^2$ and denote by $A_s$, where $s \in \R$, the time $s$ map of its flows.

Given a Hamiltonian $F:[0,1] \times \R^2 \rightarrow \R$, we pick $s\leq 0$ such that $A_s(\supp(F)) \subset a \D^2$ and define $F_s(t,x):= e^s F(t, A_s^{-1}(x))$.
Now, $F_s$ is supported in $a \D^2$ and so we can now define $\bar c(F)$ by  $$\bar c(F):= e^{-s} \bar c(F_s).$$
One can easily check that this defines a formal spectral invariant on $\R^2$.
\end{proof}

\subsubsection{Byproduct: quasi-states, heavy and super-heavy sets}\label{sec:quasi-states}
Since they are consequences of the tools developed in the preceding sections \ref{sec:formula-n-positive-genus} and \ref{sec:proof-c=N-Morse-surface}, we now give the proofs of Theorem \ref{theo:quasi-state} and Proposition \ref{prop:heaviness}.

\begin{proof}[Proof of Theorem \ref{theo:quasi-state}] First note that every disk in $\Sigma$ can be symplectically contracted to a displaceable disk. It follows from the energy-capacity inequality and the symplectic contraction principle that for every disk $D$ there is a uniform bound $C_D$ on the value of $c$ on functions supported in $D$. 

Now let $H$ be a Morse function on $\Sigma$ and let $\m D$ be the family of disks obtained by cutting along essential saddles as in Section \ref{sec:formula-n-positive-genus}. Note that for all integer $k$, the function $kH$ yields to the same decomposition. Thus we may apply Proposition \ref{prop:c=max-on-disks} to $kH$: 
$$c(kH)=\max\{kH(\partial D)+c(\bar{(kH)}_D)\,|\,D\in\m D\}.$$
Since $0\leq c(\bar{(kH)}_D)\leq C_D$, we deduce that $\zeta(H)=\max\{H(\partial D)\,|D\in\m D\}$. 
 Since the maximum of $H$ over its essential saddles is nothing but the maximum of $H(\partial D)$, $D\in\m D$, this concludes the proof of the theorem for Morse functions.
 
 Now for any continuous function $H$ on $\Sigma$ we consider the quantity 
$$
\eta(H) = \inf\left\{h_{0}: H^{-1}(h_{0}, +\infty) \mbox{ is contractible in } \Sigma\right\}.
$$
We leave it to the reader to check that $\eta$ depends continuously on $H$ when the space of continuous functions is equipped with the sup norm. Besides, $\zeta$ is $1$-Lipschitz for the sup norm. Proposition~\ref{prop:cut_ess_sadd} implies that if $H$ is Morse then $\zeta(H) = \eta(H)$. Since Morse functions are dense in the space of continuous functions, we conclude that $\zeta = \eta$.
\end{proof}

\begin{proof}[Proof of Proposition \ref{prop:heaviness}] 
We begin with the first part of the proposition.  Assume that a closed subset $X\subset \Sigma$ is not included in an open disk. Since $\zeta$ is continuous, it is sufficient to verify the definition of heaviness, or super heaviness, for Morse functions. Thus, let $H$ be a Morse function on $\Sigma$ and denote $C:= \inf(H|_X)$. Consider the decomposition described in Proposition \ref{prop:cut_ess_sadd} associated to $H$. If $X$ intersects the level set of an essential saddle, then $\zeta(H)\geq C$ by Theorem \ref{theo:quasi-state}. Otherwise $X$ is included in the surface $\Sigma'$ of Proposition \ref{prop:cut_ess_sadd}. Since it is not included in a disk, $X$ meets one of the cylinders, say $S$, which form the connected components of $\Sigma'$. Since the restriction of $H$ to $S$ has no critical point, it attains its maximum on one of the boundary components of $S$, hence this maximum is the value of $H$ at an essential saddle; on the other hand, it is larger than the minimum of $H$ on $X$, hence larger than $C$. Using Theorem  \ref{theo:quasi-state} again, we get $\zeta(H)\geq C$.

Conversely assume that $X$ is included in an open disk $D$. Then, by the argument used in the proof of Theorem \ref{theo:quasi-state}, there exists a constant  $C_D$ uniformly bounding the values of $\zeta$ on functions supported in $D$.  Thus, $\zeta$ vanishes on all functions supported in $D$. Taking a smooth function supported in $D$ with value $C$ on $X$, we see that $X$ cannot be heavy.

Next, we prove the second half of the proposition. Let $H$ be a Morse function and denote $C:= \sup(H|_X)$. Assume that the complement of $X$ admits no closed non-contractible curves.  By Theorem \ref{theo:quasi-state}, showing $\zeta(H) \leq C$,  reduces to showing that $X$ meets the level set of all essential saddles of $H$. But this is immediate since $X$ meets all non-contractible curve. 

 Conversely, assume that the complement of $X$ contains a curve $Y$ which is non-contractible. Then, by the first part of the proposition $Y$ is heavy.  Moreover, it is disjoint from $X$.  It is easy to see from the definition that every superheavy set must intersect every heavy set.  We conclude that $X$ is not superheavy. 
\end{proof}

\subsection{From Morse functions to any autonomous Hamiltonian}\label{sec:c=N-nonMorse}

In this section, we finish the proof of Theorem \ref{theo.axiomatic-c=n}. In Sections \ref{sec:proof-c=N-Morse-disk} and \ref{sec:proof-c=N-Morse-surface}, we proved it for Morse functions. We will deduce Theorem \ref{theo.axiomatic-c=n} from this particular case. The \emph{formal} spectral invariant $c$ depends continuously on $H$; thus the deduction would be immediate if $\cN$ shared the same property. While this continuity property is still unknown, we will see that every $H$ can be approximated by some particular Morse function $H'$ for which we can prove that $\m N(H')$ is close to $\m N (H)$.

\paragraph{Topology of fixed and periodic orbits.}
Let $H: \Sigma \to \R$ be a smooth function. We decompose the set of contractible fixed points as a disjoint union:
$$ \Fix_{c}(\phi^{1}_{H}) = \Per_{c}(H) \sqcup \Crit_{iso}(H) \sqcup \Crit_{acc}(H),$$
where $\Per_{c}(H)$ is the set of contractible fixed points of $\phi_H^1$ that are not critical points of $H$, $\Crit_{iso}(H)$ the subset of isolated points in $\Crit(H)$ and $\Crit_{acc}(H)$ its subset of non isolated points.

Since $H$ is smooth, the closure of $\Per_{c}(H)$ does not meet $\Crit_{acc}(H)$. Indeed, let $x$ be accumulated by critical points. Then the second differential $d^2H(x)$ is degenerate. Up to replacing $H$ by $H \circ A$ where $A$ is a Hamiltonian diffeomorphism for which $x$ is a saddle point with a big dilatation in the degenerate direction, we may assume that $\norm{d^2H(x)}$ is arbitrarily small. Choose a neighborhood $V$ of $x$ on which $\norm{d^2H}$ is still small. Then on the one hand, by continuity of the flow, every 1--periodic orbit starting close enough to $x$ is included in $V$; on the other hand, by a standard argument, $V$ does not contain any 1--periodic orbit (see for example~\cite{audin-damian}, Proposition 6.5.1).

Let $U$ be the complement of the closure of $\Per_{c}(H)$. As a consequence, only a finite number of connected components of $U$ intersect $\Crit(H)$, and each of these connected components is the interior of a compact manifold with boundary, the boundary being made of a finite number of contractible 1--periodic orbits. 
We denote by $U_{1}, \dots, U_{\ell}$ the connected components that meet the set $\Crit(H)$. Each $U_{i}$ is invariant by the flow. 

We will now use the above description to build a Morse perturbation $H'$ such that $\m N(H)$ is close to $\m N(H')$.

\paragraph{Perturbation Lemma}

For each index $i$, we set
$$
X(U_{i}) = \{x \in \Crit(H) \cap U_{i}\,|\, \rho(x) \leq 0\}.
$$
and choose a function $G_i$ compactly supported in $U_{i}$ as follows. If $X(U_i)$ is empty, we let $G_i=0$. Otherwise, let $x_{i} \in X(U_{i})$ be such that $H(x_{i}) = \max_{x \in X(U_{i})} H(x)$. Again if $d^2H(x_{i})$ is negative definite then let $G_{i}=0$. In the remaining case, define $G_{i}$ so that its maximum is attained at $x_{i}$, such that the Hessian $d^{2}G_{i}(x_{i})$ is negative definite. Note that $x_{i}$ is still a fixed point for the time one of the flow associated to $H+G_{i}$, and its rotation number is strictly negative. Moreover, we choose $G_i$ to be $C^2$-small enough that $H+G_i$ has no non-trivial periodic orbit in $U_i$.

Now, for each $U_{i}$, choose $F_{i}$ compactly supported in a neighborhood of the set of degenerate critical points of $H$ in $U_{i}$, 
which is $C^{2}$-small and such that the Hamiltonian
$$
H' = H + \sum_{i=1}^l (G_{i}+ F_{i})
$$
is a Morse function. Finally, let $X'(U_{i}) = \{x \in \Crit(H') \cap U_{i}\,|\, \rho(x) \leq 0\}$.

\begin{lemma}\label{lemma:perturbation}
If each $F_{i}$ is small enough in the $C^{2}$ topology, then
\begin{enumerate}
\item $\Per_{c}(H') = \Per_{c}(H)$,
\item for each $U_{i}$, the set $X(U_{i})$ is empty if and only if the set $X'(U_{i})$ is empty,
\item $\max_{X'(U_{i})} H'$ is close to $\max_{X(U_{i})} H$,
\item $\cN(H')$ is close to $\cN(H)$.
\end{enumerate}
\end{lemma}

\begin{proof}
 The three first properties are easily obtained. Indeed, by choosing the $F_i$'s small enough in the $C^2$ sense, we first ensure that $H'$ has the same non-trivial 1--periodic orbits as $H$. This gives Property 1. 
 
If $X(U_{i})$ is empty then it contains no degenerate critical point of $H$, we get $H'=H$ 
on $U_{i}$ and thus $X'(U_{i})$ is also empty. If it is not empty, then 
 the $C^2$-smallness of $F_{i}$ implies that the rotation number of the point $x_i$ remains negative. Property 2 follows. Finally, Property 3 is an immediate consequence of Property 2 and the $C^0$-smallness of the $F_i$'s.

To prove Property 4, we establish a bijective correspondence between mnus's of $H$ and mnus's of $H'$. 
 The structure of unlinked sets for autonomous systems is described
by Corollary~\ref{corol:unlinked-sets-autonomous}. As a consequence of this description, the mnus's are the sets of the following form: a certain (finite) collection $Y \subset \Per_c(H)$ and all the critical points in the complement of the union of the disks $D(y)$ bounded by the  1--periodic orbits of points $y$ in $Y$.
In particular, the mnus's of $H$ are all of the form:
$$X=Z\cup \bigcup_{i=1}^lX_i,$$
where $Z$  is a subset of the closure of $\Per_c(H)$ and each $X_i$ is either $X(U_i)$ or $\emptyset$. The mnus's of $H'$ have a similar description. To every mnus $X=Z\cup\bigcup_{i=1}^l X_i$ of $H$, we associate a set $\Psi(X)=Z\cup\bigcup_{i=1}^l X_i'$, where for every $i\in\{1,\ldots,l\}$, $X_i'=\emptyset $ if $X_i=\emptyset$, and $X_i'=X'(U_i)$ if $X_i=X(U_i)$. It follows from Property 2 that the map $\Psi$ is a bijection between mnus's of $H$ and mnus's of $H'$. Moreover, Property 3 implies that the maximum of the action of $H$ over $X$ is close to the maximum of the action of $H'$ over $\Psi(X)$. Taking minimum over all mnus's we get Property 4.
\end{proof}

\paragraph{End of the proof of Theorem~\ref{theo.axiomatic-c=n}}

Let $H$ be a smooth function on $\Sigma$. Then, according to Lemma \ref{lemma:perturbation}, we can find arbitrary close to $H$ a Morse function $H'$ such that $\cN(H')$ is close to $\cN(H)$. On the other hand, the continuity of spectral invariants also implies that $c(H')$ is close to $c(H)$. Since we proved that $c=N$ for all Morse functions, we obtain that $c(H)$ is arbitrary close to $\cN(H)$. Thus $c(H)=\cN(H)$.\hfill$\Box$


\section{Max Formulas for spectral invariants of Schwarz and Viterbo} \label{sec:max_form}
The main goal of this section is to prove that the spectral invariants constructed by Viterbo on $\R^{2n}$ and by Schwarz on closed aspherical manifolds satisfy certain max formulas.  It is an immediate consequence of these max formulas  that the spectral invariants of Viterbo and Schwarz are both \emph{formal} spectral invariants in the sense of Definition \ref{def:formal_spec}.   As mentioned in the introduction, these max formulas are of independent interest and have consequences that go beyond the scope of this paper. 
For this reason, in this section of the paper we no longer restrict ourselves to two dimensional symplectic manifolds.  

 \medskip
 
 \noindent \textbf{The max formula on $\R^{2n}$:} Following  Viterbo's notation, we will denote by $c_+$ and $c_-$ the two spectral invariants constructed by him in \cite{viterbo}.  We will recall their construction, which is based on generating functions, in Section \ref{sec:def_c_R2n}.
 
 We will say that $N$ subsets $A_1,\ldots,A_N$ in $\R^{2n}$ are \emph{symplectically separated} if the minimum over all indices $1\leq i<j\leq N$ of the euclidean distance between $\psi(A_i)$ and $\psi(A_j)$ can be made arbitrary large for some symplectic diffeomorphism $\psi$. For example, two disjoint convex sets are always symplectically separated. In Section \ref{sec:proof-max-formula-R2n} will prove the following statement.

\begin{theo}\label{theo:max-formula-R2n} If $H_1,\ldots, H_N$ are compactly supported Hamiltonian diffeomorphisms of $\R^{2n}$ whose supports are symplectically separated, then: $$c_+(H_1+\ldots +H_N)=\max(c_+(H_1),\ldots,c_+(H_N)),$$
$$ c_-(H_1+\ldots+H_N)=\min(c_-(H_1),\ldots,c_-(H_N)).$$
\end{theo}

The proof of this theorem is by induction. For $N=2$, the idea is that when both supports are far enough from each other (which can be achieved by a suitable sympletic diffeomorphism), then it becomes possible to build a generating function of $H_1+H_2$ that coincides with a generating function of $H_1$ on some open set surrounding the support of $H_1$ and with a generating function of $H_2$
 on some open set surrounding the support of $H_2$. Then an argument based on the Mayer-Vietoris long exact sequence, applied to the sublevels of the generating functions, allows us to compare the different spectral invariants. The details will be carried out in Section  \ref{sec:proof-max-formula-R2n}.

\medskip
 \noindent \textbf{The max formula on closed and aspherical symplectic manifolds:}
 Let $c$ denote the spectral invariant constructed by Schwarz on a closed and aspherical symplectic manifold $M$.  We will recall the construction of $c$ in section \ref{sec:def_c_aspherical}.
 
 Recall the definition of an incompressible Liouville domain from Section \ref{sec:symp-cont}.    In Section \ref{sec:proof-max-formula-aspherical} will prove the following max formula for Hamiltonians whose supports are contained in a disjoint union of incompressible Liouville domains.
  
  \begin{theo}\label{theo:max-formula-aspherical}
   Suppose that $F_1,\ldots, F_N$ are  Hamiltonians whose supports are contained, respectively, in pairwise disjoint incompressible Liouville domains $U_1 ,\ldots, U_N$. Then,
 $$c(F_1 +\ldots + F_N) =  \max\{c(F_1), \ldots, c(F_N)\}.$$
  \end{theo}
  Interestingly enough, this max formula does not hold on non-aspherical manifolds.  In Section \ref{sec:countr_exmpl} we will construct an example of a Hamiltonian on the sphere which does not satisfy this max formula.
  
  Here is an overview of our strategy for proving the above theorem.   The idea is to symplectically contract each of the $F_i$'s, as described in Section \ref{sec:symp-cont}, to obtain functions $F_{i,s}$.  Equation \eqref{eq:spec_dial2} implies that it is sufficient to prove the max formula for the $F_{i,s}$'s.  Next we study the Floer trajectories of (an appropriate perturbation of) $F_{1,s} +\ldots + F_{N,s}$. An application of  Lemma \ref{lem:energy_est} will provide us with a positive constant $\epsilon > 0$ such that any Floer trajectory which travels between distinct $U_i$ and $U_j$ has energy greater than $\epsilon.$ On the other hand,  by picking $s$ to be sufficiently negative we can ensure, using Equation \eqref{eq:spec_dial1}, that the spectrum of $F_{1,s} +\ldots + F_{N,s}$ is contained in $(- \frac{\epsilon}{4}, \frac{\epsilon}{4})$ and hence any Floer trajectory traveling between distinct $U_i$ and $U_j$ has action less than $\frac{\epsilon}{2}$. Using these ideas, in Lemmas \ref{lem:no_cont_traj} and \ref{lem:no_bdry_traj}, we conclude that there exist no such Floer trajectories.  This drastically simplifies the Floer homological picture and allows us to fully describe the relations among the various Floer cycles representing the fundamental class $[M]$; see Lemma \ref{lem:main_lemma}.   We carry out the details of this strategy in Section \ref{sec:proof-max-formula-aspherical}.

\subsection{The max formula on $\bb R^{2n}$}

In this section, we establish the max formula for the spectral invariant $c_+$ introduced by Viterbo in \cite{viterbo} using generating functions. Let us quickly remind the reader of its construction.

\subsubsection{Generating functions and the construction of $c_+$}\label{sec:def_c_R2n}

Given a Lagrangian submanifold $L$ in a cotangent bundle $T^*M$ of a closed manifold $M$, a generating function quadratic at infinity (or g.f.q.i) for $L$ is a function $S:M\times\R^N\to\R$ for some integer $N$, such that $L$ admits the following description
$$L=\{(x,p)\in T^*M\,|\, \exists\xi\in\R^N,\partial_\xi S(x,\xi)=0, \partial_x S(x,\xi)=p\},$$
and moreover $S$ coincide with a quadratic form $Q$ at infinity, i.e., there exists a compact set $K\subset M\times\R^N$  and a non-degenerate quadratic form $Q$ on $\R^N$ such that for every $(x,\xi)\notin K$, $S(x,\xi)=Q(\xi)$.
According to a theorem of Laudenbach and Sikorav (\cite{sikorav}, \cite{brunella}), every Lagrangian submanifold which is Hamiltonian isotopic to the zero section admits a g.f.q.i.  

Hamiltonian diffeomorphisms of the standard symplectic space $(\R^{2n},\omega_0)$ can also be represented by generating functions by the following construction. Let $\phi\in\Ham_c(\R^{2n})$ and denote by $\Gamma_\phi$ its graph which is a Lagrangian submanifold of $(\R^{2n}\times\R^{2n},-\omega_0\oplus\omega_0)$. Given a symplectic diffeomorphism $\Psi:\R^{2n}\times\R^{2n}\to T^*\R^{2n}$ the Lagrangian $\Psi(\Gamma_\phi)$ is Hamiltonian isotopic to the zero section and therefore admits a g.f.q.i. $S:\R^{2n}\times\R^N\to \R$. This function can be extended to $\S^{2n}=\R^{2n}\cup\{\infty\}$ by setting $S(\infty,\xi)=Q(\xi)$ for all $\xi \in \R^N$. We continue to denote this extension by $S$ and refer to it as a g.f.q.i. for $\phi$. 

Spectral invariants are  defined as follows. Let us denote by $e$ and $\mu$ the generators of the cohomology groups $H^0(\S^{2n})$ and $H^{2n}(\S^{2n})$ (with coefficients in a field $\mathbb F$). Given a function $F$, we denote by $F^\lambda=\{x\,|\, F(x)\leq\lambda\}$ its $\lambda$ sublevel. Moreover, the notations ``$F^{-\infty}, F^{\infty}$'' will mean ``$F^\lambda$ for $\lambda$ close to $-\infty, \infty$'', respectively. Let $d$ stand for the dimension of the negative space of the quadratic form $Q$.  Recall that
$H^k(Q^{+\infty},Q^{-\infty})=\{0\}$ for every integer $k \neq d$ and 
$H^d(Q^{+\infty},Q^{-\infty})=\F$.
For every real number $\lambda$ there is a group homomorphism $i_\lambda:H^*(\S^{2n})\to H^{*+d}(S^\lambda,s^{-\infty})$ which is the composition of the following natural maps:
\begin{align*}H^*(\S^{2n})&\simeq H^*(\S^{2n})\otimes H^{*}(Q^{+\infty},Q^{-\infty})\\
&\simeq H^{*}(\S^{2n}\times Q^{+\infty},\S^{2n}\times Q^{-\infty})=H^{*}(S^{+\infty},S^{-\infty})\\
&\to H^{*}(S^\lambda,S^{-\infty}).
\end{align*}
Note that for a class $\alpha \in H^*(\S^{2n})$ of degree $k$, $i_\lambda(\alpha)$ has degree $k+d$.
It follows from the Viterbo-Th\'eret uniqueness theorem (\cite{viterbo}, \cite{theret}) that the following definition does not depend on the choice of the g.f.q.i.

\begin{defi}(Viterbo \cite{viterbo})
\begin{align*} c_-(\phi) &=\inf\{\lambda\,|\,i_\lambda(e)\neq 0\},\\
c_+(\phi) &=\inf\{\lambda\,|\,i_\lambda(\mu)\neq 0\}.
\end{align*}
\end{defi}

The two invariants are related by the duality formula $c_+(\phi)=-c_-(\phi^{-1})$ for every $\phi\in \Ham_c(\R^{2n})$ and satisfy the inequalities $c_-\leq 0\leq c_+$.
It is known that the invariant $c_+$ satisfies all the axioms of Theorem \ref{theo.axiomatic-c=n} except for the Max Formula which will be established below. 

We define these spectral invariants for a compactly supported Hamiltonian $H$ by setting
$$c_+(H)=c_+(\phi_H^1),\quad c_-(H)=c_-(\phi_H^1).$$

\subsubsection{Proof of the max formula on $\R^{2n}$}\label{sec:proof-max-formula-R2n}
\begin{proof}[Proof of Theorem \ref{theo:max-formula-R2n}] First note that by an easy induction argument the general case follows from the particular case where $N=2$. Next, remark that by the duality formula, the max formula for $c_+$ is equivalent to the min formula for $c_-$. We will prove the min formula for $c_-.$ 

We will use the notation $\phi_1, \phi_2$ for the time-one maps of $H_1$ and $H_2$. 
Let $S_1:\S^{2n}\times\R^{N_1}\to\R$, $S_2:\S^{2n}\times\R^{N_2}\to\R$ be generating functions quadratic at infinity for $\phi_1$ and $\phi_2$. It follows from the proof of the existence of generating functions that $S_1$ and $S_2$ can be chosen so that they have the same number of extra-parameters, i.e. $N_1=N_2=:N$ and they coincide at infinity with the same quadratic form $Q:\R^N\to\R$. Indeed, if we refer for instance to the proof given in \cite{brunella}, the quadratic form obtained when one constructs a g.f.q.i. for a diffeomorphism $\phi$ can be chosen to depend only on the number of diffeomorphisms $C^1$-close to the identity used to decompose $\phi$. Moreover, using the fact that the supports are symplectically separated, we can conjugate $\phi_1$ and $\phi_2$ by an appropriate symplectic diffeomorphism $\psi$ to ensure that we are in the following situation (recall that $c_\pm$ is conjugation invariant): There exist open sets $U_1$ and $U_2$ in $\S^{2n}$  such that 
\begin{itemize}
\item $U_1\cup U_2=\S^{2n}$,
\item $U_1$ and $U_2$ are contractible and their intersection is connected,
\item $\forall (x,v)\in U_2\times \R^N,\ S_1(x,v)=Q(v)$,
\item $\forall (x,v)\in U_1\times\R^N,\ S_2(x,v)=Q(v)$.
\end{itemize}
In particular, $S_1$ and $S_2$ coincide with $Q$ on $(U_1\cap U_2)\times \R^N$.

Let $\phi=\phi_1\circ\phi_2=\phi_{H_1+H_2}^1$. It follows from the assumptions above that the Lagrangian $\Psi(\Gamma_\phi)$ coincides with $\Psi(\Gamma_{\phi_1})$ on $T^*U_1$ and with $\Psi(\Gamma_{\phi_2})$ on $T^*U_2$. Therefore, the function $S:\S^{2n}\times\R^N\to\R$  defined by 
$$ S(x,v)=\begin{cases} S_1(x,v) \text{ if } x\in U_1,\\
 S_2(x,v) \text{ if } x\in U_2,
\end{cases}$$
is a generating function for $\phi$.

Let $\lambda<0$ be a negative real number. 
For $i=1,2$, we consider the following (commutative) diagram of inclusions of pairs
$$
\begin{CD}
(S^{\lambda},S^{-\infty}) @<<< (S_i^{\lambda}\cap(U_i\times\R^N),S_i^{-\infty}\cap (U_i\times\R^N))\\
@VVV @VVV\\
(S^{+\infty},S^{-\infty}) @<<<   (S_i^{\lambda},S_i^{-\infty}).
\end{CD}
$$
Note that  $(S_i^{\lambda}\cap(U_i\times\R^N),S_i^{-\infty}\cap (U_i\times\R^N))= (S^{\lambda}\cap(U_i\times\R^N),S^{-\infty}\cap (U_i\times\R^N))$, which gives the top horizontal arrow.
We denote by $A_i$ the cohomology group $A_i= H^d(S_i^{\lambda}\cap(U_i\times\R^N),S_i^{-\infty}\cap (U_i\times\R^N))$. The above diagram induces the following commutative diagram in degree $d$ cohomology:
$$
\begin{CD}
H^d(S^{\lambda},S^{-\infty}) @>>> A_i\\
@AAA @AAA\\
H^0(\S^{2n}) @>>>  H^d (S_i^{\lambda},S_i^{-\infty}).
\end{CD}
$$

We now prove that the right vertical map $H^d(S_i^{\lambda},S_i^{-\infty})\to A_i$ is injective. We prove it for $i=1$, the case $i=2$ being similar. Consider the Mayer-Vietoris sequence for the covering  $\{S_1^{\lambda} \cap (U_1\times \R^N), S_1^{\lambda} \cap (U_2\times \R^N)\}$ of $S_1^{\lambda} $
 It provides in particular an exact sequence
$$C\to H^d(S_1^{\lambda},S_1^{-\infty})\to A_1\oplus B,$$
where 
\begin{align*} B &= H^d(S_1^{\lambda}\cap(U_2\times\R^N),S_1^{-\infty}\cap (U_2\times\R^N))\\
 &= H^d(U_2\times Q^{\lambda},U_2\times Q^{-\infty})\\
 &= \{0\},
\end{align*}
where the last equality holds since $\lambda<0$ and hence $U_2\times Q^{\lambda}$ retracts onto $U_2\times Q^{-\infty}$, and  
 \begin{align*}C &= H^{d-1}(S_1^{\lambda}\cap((U_1\cap U_2)\times\R^N),S_1^{-\infty}\cap ((U_1\cap U_2)\times\R^N)))\\
&= H^{d-1}((U_1\cap U_2)\times Q^{\lambda},(U_1\cap U_2)\times Q^{-\infty})\\
&= \{0\}.
\end{align*}
Thus, $H^d(S_1^{\lambda},S_1^{-\infty})\to A_1$ is injective.

We then consider the ``direct sum'' diagram: 
\begin{equation}\label{diagram}
\begin{CD}
H^d(S^{\lambda},S^{-\infty}) @>>> A_1\oplus A_2\\
@AAA @AAA\\
H^0(\S^{2n}) @>>>  H^d (S_1^{\lambda},S_1^{-\infty})\oplus H^d (S_2^{\lambda},S_2^{-\infty}).
\end{CD}
\end{equation}
We have seen that the right vertical arrow is injective.
Let us now show that the top horizontal arrow is also injective. This follows again from a Mayer-Vietoris sequence, the same as before but with $S$ instead of $S_1$:
$$C\to H^d(S^{\lambda},S^{-\infty})\to A_1\oplus A_2,$$
where 
\begin{equation*}C = H^{d-1}(S^{\lambda}\cap((U_1\cap U_2)\times\R^N),S^{-\infty}\cap ((U_1\cap U_2)\times\R^N))) = \{0\},
\end{equation*}
 as above.

We can now conclude. In the diagram (\ref{diagram}), the top horizontal arrow and the right vertical arrow are both injective. Therefore, for all $\lambda<0$ the image of a generator $e$ of $H^0(\S^{2n})$ by the bottom horizontal arrow is zero if and only if its image by the left vertical arrow is zero. Since $c_-\leq 0$, this implies the min formula for $c_-$. By duality, the max formula for $c_+$ follows.
\end{proof}

\subsection{The max formula on closed and aspherical symplectic manifolds}\label{sec:max_form_aspherical}
In this section, we establish the max formula for the spectral invariant $c$ introduced by Schwarz in \cite{schwarz} using Hamiltonian Floer theory. Let us quickly remind the reader of its construction.

\subsubsection{Hamiltonian Floer theory and spectral invariants}  \label{sec:def_c_aspherical}
In this section, we review the necessary preliminaries on Hamiltonian Floer theory and spectral invariants.   We refer the reader to Section \ref{sec:preliminary-N} for preliminaries, and our conventions, on the action functional and the Conley--Zehnder index.  Throughout the section, $(M, \omega)$ will denote a closed, connected and aspherical symplectic manifold.  The closed symplectic manifolds we are interested in this paper, i.e. closed surfaces other than $\S^2$, are all aspherical.  Floer homology was first introduced in the setting of aspherical manifolds by Floer \cite{Floer88}.  The standard reference for Floer theory in the settings of this section is \cite{salamon_lec}.  For further information on the subject  we invite the reader to consult \cite{mcduff-salamon}, \cite{audin-damian}.

  Although spectral invariants are defined for degenerate and even continuous Hamiltonians, Hamiltonian Floer homology can only be defined for non-degenerate Hamiltonians and therefore throughout the rest of this section we suppose that all Hamiltonians are non-degenerate. The Floer complex of (non-degenerate) $H$ is defined as the  $\bb Z_2$--vector space spanned by $\Crit(\m A_H)$ the set of critical points of the action functional.  Recall that $\Crit(\m A_H)$ is the set of contractible 1--periodic orbits of $\phi^t_H$.  This complex is graded by the Conley-Zehnder index.

Floer's differential is defined by counting perturbed pseudo-holomorphic cylinders:  pick a 1--parameter family of  $\omega$--compatible almost complex structures $J_t$ and consider maps $ u \co \bb R \times S^1 \rightarrow M$ satisfying Floer's equation
\begin{align}\label{eq:Floer}
\del_s u + J_t(u)(\del_t u - X_H^t(u))=0.
\end{align}
The set of Floer trajectories between two critical points of $\cA_H$, $x_-$ and $x_+$, is defined as 
\begin{align*}
  \m{\widehat{M}}(x_-,x_+;H,J) = \left\{\! u \co \bb R \times S^1 \rightarrow M \left|\!
      \begin{array}{l} u \text{ satisfies } \eqref{eq:Floer} \\ \forall t,\, u(\pm\infty,t)= x_\pm(t) \end{array}
\!\right. \!\!\right\}
\end{align*}
where the limits $u(\pm\infty,t)$ are uniform in $t$. Note that the above set admits an $\bb R$--action by reparametrization $s \mapsto s+\tau$.   The moduli space of Floer trajectories between $x_-$ and $x_+$, denoted by $\m M(x_-,x_+;H,J)$, is the quotient $\m{\widehat{M}}(x_-,x_+;H,J)/\bb R$. 

 The almost complex structure $J$ is said to be \emph{regular} if the linearization of the operator $u\mapsto \del_s u + J_t(u)(\del_t u - X_H^t(u))$ is onto for all $u$ in $\m{\widehat{M}}(x_-,x_+;H,J)$. Regularity of $J$ implies that the above moduli spaces  are all smooth finite dimensional manifolds and the dimension of $\m M(x_-,x_+;H,J)$ is $\CZ (x_-) - \CZ (x_+) -1$.   A suitably generic choice of $J$ is regular in the following sense:  The set of regular $J$'s, denoted by $\m J_{reg}(H)$, is of second category in the set of all compatible almost complex structures.  If $\CZ(x_-) - \CZ(x_+)=1$, the moduli space is compact and hence finite.  This allows us to define the Floer boundary map
$ \partial : CF_*(H) \rightarrow CF_{*-1}(H)$: For a generator $x_-$ we define  $\partial (x_-)$ by     
 \begin{align*}
  \partial(x_-) = \sum_{x_+} \# \m M (x_-,x_+;H,J) \cdot x_+
\end{align*}
where the sum is taken over all 1--periodic orbits $x_+$ such that $\CZ(x_-) -\CZ(x_+) = 1$ and $\#$ denotes the mod--2 cardinality of $\m M (x_-,x_+;H,J)$.  The above definition is extended to the entire chain complex by linearity.  

It is well-known that $\partial^2 = 0$ and thus $\partial$ defines a differential on  $CF_*(H)$.  The Floer homology of $(H, J)$, denoted by $HF_*(H, J)$, is the homology of the complex $(CF_*(H), \partial)$.  

In the course of the proof of Theorem \ref{theo:max-formula-aspherical}, we will appeal to the following observation about the structure of $\m J_{reg}(H)$.
\begin{remark} \label{rem:trans_boundary}
Suppose that $H$ is a non-degenerate Hamiltonian and let $W$ denote an open subset of $M$ containing all the 1--periodic orbits of the flow of $H$.  Fix an almost complex structure $J_0$ on $M$.   One can find a regular almost complex structure $J \in \m J_{reg}(H)$ such that $J = J_0$ on the complement of $W$. 

 This fact, which was explained to us by A. Oancea, follows easily from the content of the proof of transversality presented in \cite{FHS}; see Theorem 5.1 of \cite{FHS}. 
\end{remark}

\noindent \textbf{Invariance of Floer homology.}  Although the Floer complex depends on $(H, J)$, the Floer homology groups are independent of this auxiliary data.  Indeed, there exist morphisms
\begin{align*}
  \Psi_{H_0}^{H_1} \co CF(H_0) \rightarrow CF(H_1)
\end{align*}
inducing isomorphisms in homology which are called continuation morphisms. (To keep the notation light we have eliminated the almost complex structures from our notations.) We now describe the morphism $\Psi_{H_0}^{H_1}$. 
Pick $J_i  \in \m J_{reg}(H_i)$ and take a homotopy, denoted by $(H_s, J_s)$, from $(H_0, J_0)$ to $(H_1, J_1)$   such that  $$(H_s, J_s)=\begin{cases}  (H_0, J_0) &\text{ if } s \leq 0\\
(H_1, J_1) &\text{ if }s \geq 1\end{cases}.$$ 

Consider maps $u: \R \times \S^1 \rightarrow M$ solving an $s$--dependent version of Floer's equation \eqref{eq:Floer}:
\begin{align}\label{eq:Floer_cont}
\del_s u + J_{s,t}(u)(\del_t u - X_H^{(s,t)}(u))=0 \,\, \forall (s,t) \in \bb R \times S^1.
\end{align}
 
For 1--periodic orbits $x_0 \in \Crit(\m A_{H_0}), x_1 \in \Crit(\m A_{H_1})$ define the moduli space 
$$\m{M}(x_0,x_1, H_s, J_s) = \left\{\! u \co \bb R \times S^1 \rightarrow M \left|\!
      \begin{array}{l} u \text{ satisfies } \eqref{eq:Floer_cont} \\ u(- \infty,t)= x_0(t), \, u(+ \infty,t)= x_1(t) \end{array}
\!\right. \!\!\right\}$$

The homotopy $(H_s, J_s)$ is said to be regular if the linearization of the operator $u\mapsto \del_s u + J_{s,t}(u)(\del_t u - X_H^{(s,t)}(u))$ is onto, which implies that
the above moduli spaces are smooth finite dimensional manifolds of dimension $\CZ(x_0) - \CZ(x_1)$.  A suitably generic choice of $(H_s, J_s)$ is indeed regular. When the moduli space is zero--dimensional it is compact and hence finite.  Thus, we can define
\begin{align}\label{eq:cont_map}
  \Psi_{H_0}^{H_1}(x_0) = \sum_{x_1} \# \m M (x_0,x_1;H_s,J_s) \cdot x_1
\end{align}
where the sum is taken over all $x_1 \in \Crit(\m A_{H_1})$ such that $\CZ(x_0)= \CZ(x_1)$ and $\#$ denotes mod--2 cardinality. The morphism $\Psi_{H_0}^{H_1}$ is then extended by linearity to all of $CF_*(H_0)$. It can be shown that continuation morphisms descend to homology; we will continue to denote the maps induced on homology by the same notation. The induced map on homology does not depend on the choice of the homotopy $(H_s, J_s)$. Furthermore, at the homology level, continuation maps satisfy the following composition rule: 
\begin{align}\label{eq:compos_cont_maps}
  \Psi_{H_0}^{H_0}=\id \quad \mbox{and} \quad \Psi_{H_0}^{H_1}\circ \Psi_{H_1}^{H_2}=\Psi_{H_0}^{H_2}.
\end{align}
We see that $\Psi_{H_0}^{H_1}$ gives an isomorphism between $HF_*(H_0, J_0)$ and $HF_*(H_1, J_1).$   

Lastly, if $H$ is taken to be a $C^2$--small Morse function then the Floer homology of $H$ coincides with its Morse homology.  It follows from the above that for any regular pair $HF_*(H,J) = H_*(M).$

Invariance of Floer homology can also be established via the PSS morphism \cite{PSS}, $$\Phi_H : H_*(M) \rightarrow HF_*(H, J),$$ which gives a direct isomorphism between Morse homology and Floer homology.  Below, we will use the fact that such isomorphism exists to construct spectral invariants but we will not recall the construction of the PSS isomorphism.

 The following observation, which is analogous to Remark \ref{rem:trans_boundary}, will be used in the course of the proof of   Theorem \ref{theo:max-formula-aspherical}.
\begin{remark} \label{rem:trans_cont}
Suppose that $H_0, H_1$ are non-degenerate Hamiltonians and $J_i \in \mathrm{J}_{reg}(H_i)$ are regular almost complex structures.  Let $(H_s, J_s)$ be any homotopy, as described above, from $(H_0, J_0)$ to $(H_1, J_1)$.
Let $W$ denote an open subset of $M$ containing all the 1--periodic orbits of the flows of $H_1$, $H_2$.  One can find a regular homotopy $(H'_s, J'_s)$ from $(H_0, J_0)$ to $(H_1, J_1)$  such that $H' = H$ and $J' = J$ on the complement of $W$. 

 This fact, like Remark \ref{rem:trans_boundary}, follows easily from the content of the proof of Theorem 5.1 of \cite{FHS}. 
\end{remark}

\noindent \textbf{Spectral Invariants.} 
Let $u: \R \times S^1 \rightarrow M$ denote a Floer trajectory solving either one of Equations \eqref{eq:Floer}, \eqref{eq:Floer_cont}.  The energy of $u$ is defined as 
\begin{align}\label{def:energy}
 E(u) := \int_{\R\times[0,1]}\|\partial_su\|^2dsdt ,
\end{align}
where $\|\cdot\|$ is the norm associated to the metric $\omega(\cdot,J\cdot)$. Clearly, $E(u) \geq 0$.

It follows from a standard computation that if  $u$ is a Floer trajectory contributing to the boundary map, i.e. $u \in  \widehat{\m M}(x_-,x_+;H,J)$, then
\begin{align}\label{eq:energy_action}
\m A_H(x_-)-\m A_H(x_+)= E(u).
\end{align}
 Thus  action decreases along Floer trajectories.   Now let $a\in\R$ be a regular value of the action functional, i.e. $a\notin\spec(H)$. It follows from this observation that if we denote by $CF_*^a(H)$ the $\Z_2$--vector space generated by 1--periodic orbits of action $<a$, then $CF_*^a(H)$
is a subcomplex of $CF_*(H)$. We denote $i^a:HF_*^a(H,J) \to HF_*(H,J)$ the map induced on homology by the inclusion.
Let $[M] \in H_*(M)$ denote the fundamental class \footnote{Spectral invariants can be defined for Morse homology classes other than $[M]$ however, we have not introduced spectral invariants in full generality since we will only be dealing with the  spectral invariants associated to $[M]$.} of $M$ and define the spectral invariant of $H$ to be the number
\begin{align}\label{def:spec_inv}
c(H)= \inf \{ a \in \bb R : \Phi_H([M]) \in \mathrm{im}(i^a) \} \,. 
\end{align}

Roughly speaking, this is the minimal action required to see the fundamental class $[M]$ in $HF_*(H,J)$.   Thus far we have defined $c(H)$ for non-degenerate $H$.  One can show that spectral invariants of two non-degenerate Hamiltonians $H$, $G$ satisfy the Lipshitz estimate from the Lipschitz continuity property in Section \ref{sec:formal-spec}.  This estimate allows us to extend $c( \cdot)$ continuously to all smooth (in fact continuous) Hamiltonians. 

The spectral invariant constructed in this section satisfies the spectrality and continuity axioms from Definition \ref{def:formal_spec} and all the properties discussed in Section \ref{sec:formal-spec}; for proofs we refer the reader to  \cite{Oh05b, Oh06a, schwarz}.  Below we prove that $c$ is indeed a \emph{formal} spectral invariant, in the sense of Definition \ref{def:formal_spec}, by showing that it satisfies the max formula.

\subsubsection{Proof of the max formula on closed aspherical manifolds.} \label{sec:proof-max-formula-aspherical}
 Our proof of Theorem \ref{theo:max-formula-aspherical}  relies on the following preliminary fact.
  
\medskip
  \noindent \textbf{Energy estimates for Floer trajectories:}  
  The following lemma is a slight reformulation of Proposition 3.2 of \cite{hein}. We will not provide a proof as it follows quite easily from Hein's argument.  A similar result appears in \cite{usher}; see Lemma 2.3 therein.  Recall that $E(u)$ denotes the Energy of a Floer trajectory as defined by Equation \eqref{def:energy}.
  
\begin{lemma}\label{lem:energy_est}
  Let $V$ denote an open subset of $M$ with (at least) two distinct smooth boundary components $W_1, W_2$. Consider a Hamiltonian $H$ which is autonomous in $V$ and whose time-1 map $\phi^1_H$ has no fixed points in $V$.  Furthermore, assume that $W_1$ and $W_2$ are contained in two distinct level sets of $H$.  Suppose that $u: \R \times \S^1 \rightarrow M$ satisfies Floer's equation \eqref{eq:Floer}.  There exists a constant $\epsilon(V, H|_V, J|_V) >0$, depending on the domain $V$ and the restrictions of the Hamiltonian $H$ and the almost complex structure $J$ to the domain $V$ such that if $u$ intersects $W_1$ and $W_2$ then  $$E(u) \geq \epsilon.$$
\end{lemma}

\begin{proof}[Proof of Theorem \ref{theo:max-formula-aspherical}]
Observe that it is sufficient to prove the theorem under the assumption that each $U_i$ is connected;  we will make this assumption from this point onward.  We first choose an auxiliary connected incompressible Liouville domain $U_0$ that does not intersect any of the $U_i$'s. For every $i=0,\ldots,N$, let $\xi_i$ denote a Liouville vector field of $U_i$. 
We construct shells  $V_0,\ldots, V_{N}$ near the boundary of the domains $U_0,\ldots, U_{N}$ as follows: a tubular neighborhood of the boundaries $\partial U_i$ can be identified, via a diffeomorphism, with  $(-\delta, \delta) \times \partial U_i$ such that $(-\delta, 0) \times \partial U_i$ is contained inside $U_i$.  Set $V_i = (0, \delta) \times \partial U_i$.   Observe that, since we are not supposing $\partial U_i$ is connected each shell $V_i$ might in fact be a union of connected shells.

 Take $\delta$ from the previous paragraph to be small enough such that $(-\delta, 0) \times \partial U_i$ does not intersect the support of $F_i$. Pick an autonomous Hamiltonian $H$ such that 
 \begin{enumerate}
 \item $H=0$ on $U_i \setminus (-\delta, 0) \times \partial U_i$ for all $i=0,\ldots,N$, and $H < 0$ on the rest of $M$. Hence, $H$ vanishes on the supports of all $F_i$'s and on $U_0$.
 \item $H$ has no critical points in $[-\delta, \delta] \times \partial U_i$, $i=0,\ldots,N$. 
 \item For each $i=0,\ldots,N$, the sets $\partial U_i$ and $\{\delta \}\times \partial U_i$ are contained in distinct level sets of $H$. 
 \item  In the interior of its support, $H$ is Morse and has no local maxima.  
 \item  In the interior of its support, $H$ is sufficiently $C^2$--small such that the only 1--periodic orbits of $H$ are its critical points and furthermore, the Morse index of these critical points coincides with their Conley--Zehnder index.
 \end{enumerate}

Fix an almost complex structure $J$ on $M$. Suppose that $u$ is a Floer trajectory, solving Floer's equation \eqref{eq:Floer} for any Hamiltonian and almost complex structure which coincide with $H$ and $J$ on the shells $V_0,\ldots, V_N$.   By applying Lemma \ref{lem:energy_est}, we obtain $\epsilon > 0$ such that if the image of $u$ crosses\footnote{To be more precise, by saying that the image of $u$ crosses one of the shells $V_0,\ldots, V_N$ we mean that there exists $i$ such that the image of $u$ intersects $U_i \setminus V_i$ and $M \setminus U_i$.} one of the shells $V_0,\ldots, V_N$ then 
 \begin{align} \label{eq:Doris_estimate}
  E(u) \geq 4 \epsilon. 
\end{align} 
 
    Next, we symplectically contract each of the $F_i$'s to obtain $F_{1,s},\ldots, F_{N,s}$ such that for each $i \in \{1,\ldots,N\}$ we have 
    \begin{align*}
    \spec(F_{i,s}) \subset (- \tfrac{\epsilon}{2}, \tfrac{\epsilon}{2}) \text{   and    }
    \|F_{i,s}(t, \cdot)\|_{\infty} \leq \tfrac{\epsilon}{2}.
    \end{align*} 
    By Equation \eqref{eq:spec_dial2}, $c(F_{i,s}) = e^s c(F_i)$ and $c(F_{1,s}+\ldots+ F_{N,s}) = e^s c(F_1 +\ldots+ F_N)$. Hence, it is sufficient to prove the max formula for the $F_{i,s}$'s.  To simplify our notation, we will continue to denote the newly obtained Hamiltonians $F_{i,s}$ by $F_i$.
   
  Define $F_{N+1} = F_1 +\ldots + F_N$.  We will need the following lemma to prove the max formula.   We postpone its proof to the end of this section.
   
   \begin{lemma}\label{lem:1}
   $c(F_i + H)= c(F_i )$ for $i=1,\ldots, N+1.$
   \end{lemma}
   
   Next, pick an autonomous Morse Hamiltonian $G_0$ which is a $C^2$--small perturbation of $H$, which coincides with $H$ outside of $U_0,\ldots, U_N$ and which has precisely $N+1$ maximum points $p_0 \in U_0$, $p_1\in U_1$,..., $p_N \in U_N$.  For $i=1,\ldots,N+1$ define $G_i = G_0 + F_i$.  For any indices $i,j$ denote $$\spec(G_i; U_j) = \{\m A_{G_i}(x): \, x \in \Crit(\m A_{G_i})  \text{ and $x$ contained in } U_j \}.$$

    Recall that $ \spec(F_{i,s}) \subset (- \frac{\epsilon}{2}, \frac{\epsilon}{2})$.  Therefore, by taking $G_0$ to be sufficiently $C^2$--close to $H$, and thus sufficiently  $C^2$--close to $0$ on $U_0 \cup\ldots\cup U_N$, we can guarantee that 
    \begin{align}\label{eq:small_spec}
    \spec(G_i; U_j) \subset (-\epsilon, \epsilon),
    \end{align}
for all $i\in\{0,\ldots,N+1\}$ and $j\in\{0,\ldots,N+1\}$.
    Furthermore, since $ \|F_{i,s}(t, \cdot)\|_{\infty} \leq \frac{\epsilon}{2}, \, \,\, \forall t \in [0,1]$ and the Hamiltonians $G_i$ all coincide with $H$ outside of the $U_i$'s we can also guarantee that
    \begin{align}\label{eq:small_norm}
     \| G_i(t, \cdot) - G_j(t, \cdot) \|_{\infty} \leq \epsilon, \, \,\, \forall t \in [0,1].
    \end{align}
    Lastly, by replacing $F_1,\ldots, F_{N+1}$ with $C^2$--nearby Hamiltonians we may assume that $G_1, \ldots, G_{N+1}$ are non-degenerate as well. 
    
   By Remark \ref{rem:trans_boundary} we can pick almost complex structures $J_i \in \m J_{reg}(G_i)$ such that on the shells $V_0,\ldots, V_N$ each $J_i$ coincides with the almost complex structure $J$ introduced above to obtain the estimate  \eqref{eq:Doris_estimate}.  By doing so, and noting that the $G_i$'s coincide with $H$ on the shells $V_0,\ldots, V_N$, we can ensure that the estimate $$E(u) > 4 \epsilon$$ holds for any Floer trajectory $u$ of the Hamiltonians $G_i$, solving Equation \eqref{eq:Floer}, which crosses any of the shells $V_0,\ldots, V_N$.  
  
  In the course of this proof we will also need to use the estimate \eqref{eq:Doris_estimate} for Floer trajectories of the various continuation morphisms  $\Psi^{G_i}_{G_j} :CF_*(G_i) \rightarrow CF_*(G_j)$, for any $i,j \in \{0,\ldots,N+1\}.$ To define these morphisms, we must make a specific choice of a homotopy from $(G_i, J_i)$ to $(G_j, J_j)$.  
   By Remark \ref{rem:trans_cont}, we can pick a regular homotopy $(G_s^{ij}, J_s^{ij})$ from $(G_i, J_i)$ to $(G_j, J_j)$ such that on the shells $V_0, \ldots , V_N$  the almost complex structures $J_s^{ij}$  coincides with $J$, introduced above, and the Hamiltonians $G_s^{ij}$ coincide with the linear homotopy   $ (1 -\beta(s)) G_i + \beta(s) G_j = G_i + \beta(s)(G_j - G_i),$  where $\beta :\R \rightarrow [0,1]$ is a smooth non decreasing function such that $\beta(s) = 0$ for $s\leq 0$ and $\beta(s) = 1$ for $s\geq 1$. Note that for each $s$ we have $G_s^{ij} = H$  on the shells $V_0, \cdots, V_N$.
Once again, it follows that the estimate $$E(u) > 4 \epsilon$$ holds for every Floer trajectory $u$, solving Equation \eqref{eq:Floer_cont} for $G_s$, which crosses some of the shells $V_0,\ldots,V_N $.  We will now use this estimate to prove  the following lemma which will be used repeatedly.
  
\begin{lemma}\label{lem:no_cont_traj}
  Let $x_0 \in \Crit(\m A_{G_i}), x_1 \in \Crit(\m A_{G_j})$ be of Conley--Zehnder index $2n$.   Consider solutions $u$ of \eqref{eq:Floer_cont} contributing to the continuation morphism $\Psi^{G_i}_{G_j} :CF_*(G_i) \rightarrow CF_*(G_j)$.
  If there exists $u$ such that $u(- \infty, t)=x_0(t)$ and $u( \infty, t)=x_1(t)$, then there exists $k\in\{0,\ldots,N\}$ such that both of $x_0, x_1$ are contained in $U_k$.  
  
  Furthermore, the entire image of the Floer trajectory $u$ is contained in the interior of $\bar{U}_k \cup V_k$.  
 \end{lemma} 
 \begin{proof}[Proof of Lemma \ref{lem:no_cont_traj}]
Note that all the 1--periodic orbits of the $G_i$'s with Conley--Zehnder index $2n$ are contained in the $U_i$'s. For a contradiction suppose that $x_0, x_1$ are not contained in the same $U_i$.  The Floer trajectory $u$ would have to cross at least one of the shells $V_i$ and hence must have energy greater than $4 \epsilon$; see \eqref{eq:Doris_estimate}.  On the other hand, by picking the homotopy $G^{ij}_s$ to be $C^{\infty}$ close to $G_i + \beta(s) (G_j -G_i)$ and using a standard computation in Floer theory (see for example Lemma 2.12 of \cite{schwarz}) we get 
\begin{align}\label{eq:energy_action2}
 E(u) \leq \m A_{G_i}(x_0)-\m A_{G_j}(x_1) + \int_{0}^{1} \| G_i(t, \cdot) - G_j(t, \cdot) \|_{\infty} \, dt + \epsilon,
\end{align}
   where $\| \cdot \|$ denotes the $L^{\infty}$ norm on functions. 
 The  terms from the right hand side of the above inequality are all smaller than $\epsilon$ by Equations \eqref{eq:small_spec} and \eqref{eq:small_norm}.  Therefore, the right hand side gives an upper bound of approximately $3 \epsilon$ for $E(u)$ contradicting the lower bound of $4 \epsilon$ for $E(u)$.  
 
 We see from the above that the Floer trajectory $u$ can not cross any of the shells $V_i$.  Hence, the entire image of $u$ must be contained in $\bar{U}_k \cup V_k$ for some $k$.
\end{proof}    
 
 We will also need a variation of the above lemma for the Floer boundary maps $\partial : CF_*(G_i) \rightarrow CF_*(G_i).$  We will not give a proof of this lemma as it is similar to, and in fact simpler than, the proof of Lemma \ref{lem:no_cont_traj}.
 \begin{lemma}\label{lem:no_bdry_traj}
  Let $x_0, x_1 \in \Crit(\m A_{G_i})$ such that $x_0, x_1$ are contained in $U_0 \cup \ldots\cup U_N$.   Consider solutions $u$ of \eqref{eq:Floer} contributing to the boundary map $\partial :CF_*(G_i) \rightarrow CF_*(G_i)$.  If there exists $u$ such that $u(- \infty, t)=x_0(t)$ and $u( \infty, t)=x_1(t)$, then there exists $k\in\{0,\ldots,N\}$ such that both of $x_0, x_1$ are contained in $U_k$.
  
  Furthermore, the entire image of the Floer trajectory $u$ is contained in the interior of  $\bar{U}_k \cup V_k$.
 \end{lemma}

   For $ i=0,\ldots,N+1$, denote by $[M]_{G_i} \in HF_*(G_i, J_i)$ the element of $HF_*(G_i, J_i)$ representing the fundamental class of $M$. The following lemma describes the relations among the Floer cycles which represent these fundamental classes.
  \begin{lemma}\label{lem:main_lemma}
  The Floer homology classes $[M]_{G_i}$ have the following forms:
  \begin{enumerate}
  \item $[M]_{G_0}$ is represented uniquely by $p_0 + \ldots + p_N.$
  \item  For all $i=1,\ldots,N$, any representative of the fundamental class $[M]_{G_i}$ is  of the form 
$\left[C_i + \sum_{j\in\{0,\ldots,N\},j\neq i}\ p_j\right],$ where $C_i$ is a non-trivial sum of 1--periodic orbits of $G_i$ each of which is contained in the region $U_i$.
   \item Any representative of the fundamental class $[M]_{G_{N+1}}$ is of the form  $[p_0+C_1 +\ldots + C_N],$ where each $C_i$ is a non-trivial sum of 1--periodic orbits of $G_i$ each of which is contained in the region $U_i$. 
   \end{enumerate}
   Furthermore, 
   \begin{equation*}
[M]_{G_{N+1}} = [p_0 + C_1 +\ldots C_{N}] \iff  \forall i\in\{1,\ldots,N\},\ [M]_{G_i}=\left[C_i + \sum\nolimits_{i\neq j}p_j\right].                    
  \end{equation*} 
  \end{lemma}

  The max formula is an easy consequence of the above lemmas.  Since all the points $p_i$ have action almost zero, it follows immediately from Lemma \ref{lem:main_lemma} that $c(G_{N+1})$ is almost $\max\{c(G_1), \ldots, c(G_N)\}$. On the other hand, by Lemma \ref{lem:1}, $c(G_i) = c(F_i)$. Thus, the $F_i$'s satisfy the max formula. 

It remains to prove the Lemmas \ref{lem:1} and \ref{lem:main_lemma}.
  
  \begin{proof}[Proof of Lemma \ref{lem:main_lemma}]
  Since $G_0$ is $C^2$--small its Floer and Morse theory coincide.  Now in Morse homology the fundamental class is uniquely represented by the sum of all maxima and hence $[M]_{G_0} = [p_0 + \ldots+ p_N]$. To see this, one can think of the isomorphism between Morse homology and cellular homology, induced by the map that associates to a critical point its unstable manifold. In cellular homology, the fundamental class is uniquely represented by the sum of all cells of top dimension. Thus, the fundamental class in Morse homology has to be represented by the sum of all the critical points of maximal Morse index, that is of all maxima.

  Next we will prove the second assertion with regards to the form of $[M]_{G_i}$. For simplicity, we write the proof for $i=1$ and $N=2$. The argument in the general case is similar except that the notation is heavier. All of the 1--periodic orbits of $G_1$ with Conley--Zehnder index $2n$ are contained in the interior of $U_0\cup U_1 \cup U_2$. Thus, any representative of $[M]_{G_1}$ is of the form $C_1 + \lambda p_0+ \mu p_2$ where $C_1$ is a sum of 1--periodic orbits in $U_1$ and $\lambda,\mu \in \Z_2$.  We must prove that $C_1$ is non-trivial, $\lambda \neq 0$ and $\mu \neq 0$. 
 
   Consider the continuation morphism $\Psi^{G_0}_{G_1} :CF_*(G_1) \rightarrow CF_*(G_0)$ as defined by Equation \eqref{eq:cont_map}.  Since $[M]_{G_0}$ is uniquely represented by $p_0+ p_1 + p_2$ it must be the case that $p_0+ p_1 + p_2 = \Psi^{G_0}_{G_1}(C_1 + \lambda p_0+\mu p_2) =  \Psi^{G_0}_{G_1}(C_1) +\Psi^{G_0}_{G_1}( \lambda p_0)+\Psi^{G_0}_{G_1}( \mu p_2).$   Lemma \ref{lem:no_cont_traj} implies that $\Psi^{G_0}_{G_1}(C_1) = p_1$, $\Psi^{G_1}_{G_0}( \lambda p_0) = p_0$ and $\Psi^{G_0}_{G_1}( \mu p_2)=p_2$.  In particular, $C_1 \neq 0$, $\lambda \neq 0$ and $\mu\neq 0$. 

The proof of the third  assertion, about $[M]_{G_3}$ is very similar to the above and hence we will omit it. 

Lastly, we prove the final assertion. Again, to lighten the notation we only show the argument in the case $N=2$. 
 We will need to compare different continuation maps and therefore, it will be necessary for us that the homotopies used to define these maps are compatible in the following sense: For all $0\leq i,j,i',j'\leq 3$, $(G_s^{ij},J_s^{ij})=(G_s^{i'j'},J_s^{i'j'})$ on each open set $\bar U_k\cup V_k$ where $G_i=G_{i'}$ and $G_j=G_{j'}$.

 First, suppose that $[M]_{G_3} = [p_0 + C_1 + C_2]$.  We must show that $[M]_{G_1} = [C_1 + p_0 + p_2]$ and $[M]_{G_2} = [C_2 + p_0 + p_1 ]$.  We begin by proving the following claim about the continuation morphism $\Psi^{G_3}_{G_3} :CF_*(G_3) \rightarrow CF_*(G_3).$

\begin{claim}\label{claim:1}
 $\Psi^{G_3}_{G_3} (C_i) = C_i + B_i,$ for $i=1,2$ where $B_i$ is in the image of the Floer boundary map $\partial: CF_*(G_3) \rightarrow CF_*(G_3).$
\end{claim}
\begin{proof}[Proof of Claim \ref{claim:1}]
Recall that $\Psi^{G_3}_{G_3}$ induces the identity map on homology; see \eqref{eq:compos_cont_maps}.
This  in particular implies that  $\Psi^{G_3}_{G_3}(p_0+C_1 + C_2) = p_0+C_1 + C_2 + B,$ where $B$ is a boundary term.  First suppose that $B$ has a non-trivial $p_0$ contribution.  This would entail the existence of a Floer boundary trajectory $u$, solving Equation \ref{eq:Floer} for the Hamiltonian $G_3$, such that $u(\infty, t) = p_0$.  Now, $u(-\infty, t)$ would have to be a 1--periodic orbit of CZ--index $2n+1$.  Since all such 1--periodic orbits are contained in the open sets $U_i$, we conclude using Lemma \ref{lem:no_bdry_traj} that $u(-\infty, t)$ is contained in $U_0$.  But $G_3|_{U_0} = G_0|_{U_0}$ and $G_0$ is a $C^2$--small Hamiltonian and hence it has no 1--periodic orbit of $CZ$ index $2n+1$. We see that $B$ can not have a non-trivial $p_0$ contribution.
Since all the remaining 1--periodic orbits of $G_3$ with Conley--Zehnder index $2n$ are contained in $ U_1 \cup U_2$ it follows that $B = B_1 + B_2$ where $B_i$ is a sum of 1--periodic orbits contained in $U_i$.
Applying Lemma \ref{lem:no_cont_traj}, we conclude that  $\Psi^{G_3}_{G_3}(C_i) = C_i + B_i$. 

It remains to show that each $B_i$ is a boundary term.  We know that there exists $D \in CF_{2n+1}(G_3)$ such that $\partial D = B.$  Observe that all the 1--periodic orbits of $G_3$ with Conley-Zehnder index greater than $2n$ are contained in $U_1 \cup U_2$:  this is because outside of  $U_1 \cup U_2$ the Hamiltonian $G_3$ coincides with $H$ which is sufficiently $C^2$--small and Morse; see the 5th property in the list of properties of $H$.  It follows that we can write $D =D_1 + D_2$ with $D_i$ being a sum of 1--periodic orbits contained in $U_i$.  Finally, applying Lemma \ref{lem:no_bdry_traj} we conclude that $\partial(D_i) = B_i$.
\end{proof}

  We will next show that $[M]_{G_1} = [p_0 + C_1 + p_2]$.  This will be achieved by proving that the continuation morphism $\Psi^{G_1}_{G_3} :CF_*(G_3) \rightarrow CF_*(G_1)$ satisfies the following:
\begin{align*} 
\Psi^{G_1}_{G_3}(C_1) =  C_1 + B_1,\  \Psi^{G_1}_{G_3}( p_0) = p_0\text{ and } \Psi^{G_1}_{G_3}( C_2) = p_2,
\end{align*}
where $B_1$ is a boundary term.   Since $\Psi^{G_1}_{G_3}(p_0+C_1 + C_2) = \Psi^{G_1}_{G_3}( p_0)+\Psi^{G_1}_{G_3}(C_1) + \Psi^{G_1}_{G_3}(C_2)$ is a representative for $[M]_{G_1}$, by the second assertion it is of the form $p_0+C_1' + p_2$, where $C_1'$ is a sum of 1--periodic orbits contained in $U_1$.  By Lemma \ref{lem:no_cont_traj}, this can only occur if $\Psi^{G_1}_{G_3}(C_1)= C_1'$, $\Psi^{G_1}_{G_3}( p_0) = p_0$ and $\Psi^{G_1}_{G_3}(C_2) = p_2$.  We must now show that $\Psi^{G_1}_{G_3}(C_1)= C_1 + B_1$.  We will apply the latter part of Lemma \ref{lem:no_cont_traj}: since $C_1$ and $\Psi^{G_1}_{G_3}(C_1)= C_1'$ are both contained in $U_1$, it must be the case that all the Floer trajectories contributing to $\Psi^{G_1}_{G_3}(C_1)$ are contained in the set $\bar{U}_1 \cup V_1$.  Observe that $G_1|_{\bar{U}_1 \cup V_1} = G_3|_{\bar{U}_1 \cup V_1}$ (indeed they both coincide with $G_0|_{\bar{U}_1 \cup V_1} + F_1$)and hence $(G_s^{33},J_s^{33})=(G_s^{31},J_s^{31})$ by the compatibility requirement. It can easily be checked that this implies that $\Psi^{G_1}_{G_3}(C_1) = \Psi^{G_3}_{G_3}(C_1).$ But, Claim \ref{claim:1} tells us that $\Psi^{G_3}_{G_3}(C_1) = C_1 + B_1$ where $B_1$ is a boundary term.  

Similarly, one can prove  that $[M]_{G_2} = [p_0 + p_1 + C_2]$ by showing that the continuation morphism $\Psi^{G_2}_{G_3} :CF_*(G_3) \rightarrow CF_*(G_2)$ satisfies the following:
\begin{align*} 
 \Psi^{G_2}_{G_3}(p_0)=p_0,\ \Psi^{G_2}_{G_3}( C_1) = p_1 \text{ and } \Psi^{G_2}_{G_3}(C_2) =  C_2 + B_2,
\end{align*}
where $B_2$ is a boundary term. 

Finally, it only remains to prove that if $[M]_{G_1} = [p_0 + C_1 + p_2]$ and $[M]_{G_2} = [p_0 + p_1 + C_2]$ then $ [M]_{G_3} =  [p_0 + C_1 + C_2]$.  We will use the following claim which is analogous to Claim \ref{claim:1}.  Its proof is similar to the proof of Claim \ref{claim:1} and hence will be omitted.
\begin{claim}\label{claim:2}
$\Psi^{G_i}_{G_i} (C_i) = C_i + B_i,$ for $i=1,2$ where $B_i$ is in the image of the Floer boundary map $\partial: CF_*(G_i) \rightarrow CF_*(G_i).$
\end{claim}

 Clearly, $\Psi^{G_3}_{G_1}(p_0+C_1 + p_2) = \Psi^{G_3}_{G_1}(p_0)+ \Psi^{G_3}_{G_1}(C_1) + \Psi^{G_3}_{G_1}(p_2)$.  As was done in the previous paragraph, by appealing to the latter part of Lemma \ref{lem:no_cont_traj} and observing that $G_1|_{\bar{U}_1 \cup V_1} = G_3|_{\bar{U}_1 \cup V_1}$, which implies $(G_s^{13},J_s^{13})=(G_s^{11},J_s^{11})$, one proves that  $\Psi^{G_3}_{G_1}(C_1) =  \Psi^{G_1}_{G_1}(C_1) = C_1 + B_1,$ where $B_1$ is a boundary term and the last equality follows from Claim \ref{claim:2}.  Similarly, by appealing to the latter part of Lemma \ref{lem:no_cont_traj} and observing that $G_1|_{\bar{U}_2 \cup V_2} = G_0|_{\bar{U}_2 \cup V_2}$, we conclude that $\Psi^{G_3}_{G_1}(p_2) = \Psi^{G_3}_{G_0}(p_2)$. 
Lastly, using similar arguments, we check that $\Psi^{G_3}_{G_1}(p_0) =\Psi^{G_3}_{G_0}(p_0)$. 
 We conclude from the above discussion that $$ [M]_{G_3} =  [\Psi^{G_3}_{G_1}( p_0+C_1 + p_2)] = [\Psi^{G_3}_{G_0}(p_0)+ C_1 + B_1 + \Psi^{G_3}_{G_0}(p_2)]$$ $$ = [\Psi^{G_3}_{G_0}(p_0)+ C_1  + \Psi^{G_3}_{G_0}(p_2)] .$$
Similarly, we obtain $$[M]_{G_3}  = [ \Psi^{G_3}_{G_2}( p_0 + p_1 + C_2)] = [\Psi^{G_3}_{G_0}(p_0) + \Psi^{G_3}_{G_0}(p_1) +  C_2 + B_2 ]$$ $$= [\Psi^{G_3}_{G_0}(p_0) + \Psi^{G_3}_{G_0}(p_1) +  C_2  ].$$
Comparing the above we get $ [\Psi^{G_3}_{G_0}(p_0)+ C_1  + \Psi^{G_3}_{G_0}(p_2)] = [\Psi^{G_3}_{G_0}(p_0) + \Psi^{G_3}_{G_0}(p_1) +  C_2  ].$  Rearranging and simplifying the terms in the above equality we obtain  $$ [ \Psi^{G_3}_{G_0}(p_0) 
+ C_1 + C_2] =  [\Psi^{G_3}_{G_0}(p_0+ p_1 + p_2) .]$$
Lastly, we appeal to the second assertion of Lemma \ref{lem:main_lemma} to conclude that $\Psi^{G_3}_{G_0}(p_0)=p_0$ and hence obtain
$$ [ p_0 
+ C_1 + C_2] =  [\Psi^{G_3}_{G_0}(p_0+ p_1 + p_2).]$$
The right hand side is clearly a representative for $[M]_{G_3}$ and thus so is $ [ p_0 
+ C_1 + C_2]$.

%
%

  \end{proof}

\begin{proof}[Proof of Lemma \ref{lem:1}]
  Without loss of generality, we may assume that every 1--periodic orbit of $F_i$ which is contained in the interior of its support is non-degenerate.  Indeed, this can be achieved by making a $C^2$ small perturbation of $F_i$ in the interior of its support. This in particular implies that every 1--periodic orbit of $F_i$ with non-zero action is non-degenerate.
  
  Consider the the $1$--parameter family of Hamiltonians $sH + F_i$, where $s \in [0,1]$.  The Hamiltonians $sH$ and $F_i$ have disjoint supports and hence $$\spec(sH + F_i) = \spec(sH) \cup \spec(F_i).$$ Since $\spec(F_i)$ is a set of measure zero and $c(sH+F_i)$ is a continuous function of $s$, Lemma \ref{lem:1} follows immediately from the following claim.
  
  \begin{claim}\label{claim:3}
  $c(sH+ F_i) \in \spec(F_i)$ for all $s \in[0,1]$.
  \end{claim}
  \begin{proof}[Proof of Claim \ref{claim:3}]
  We will use Lemma 9 of \cite{HLS12} which states the following: Let $H$ denote a possibly degenerate Hamiltonian and let $A = \{z \in Per_c(H): c(H) = \m A_{H}(z) \}$.  Suppose that every 1--periodic orbit in $A$ is non-degenerate. Then, there exists $z \in A$ such that $\CZ(z) = 2n.$ 
    
   Observe that every 1--periodic orbit of $sH + F_i$ with non-zero action is non-degenerate: This is because 1--periodic orbits with non-zero action are contained either in the interior of the support of $sH$ or the interior of the support of $F_i$.  Both $sH$ and $F_i$ are non-degenerate in the interior of their supports.
  
  In order to obtain a contradiction, suppose that the claim does not hold and hence there exists $s_0 \in [0,1]$ such that $c(s_0 H+ F_i) \in \spec(s_0 H) \setminus \spec(F_i).$    Note that $c(s_0 H+ F_i) \neq 0$ because $ 0 \in \spec(F_i)$.  We see that $\{z \in \Crit(\m A_{s_0 H+ F_i}): \, c(s_0 H+ F_i) = \m A_{s_0 H+ F_i}(z)\}$ is a subset of the non-degenerate critical points of $H$.  Because $H$ is Morse and $C^2$-small in the interior of its support the Conley--Zehnder index of these points coincides with their Morse index and, by construction, non of these critical points have Morse index $2n$.  This contradicts Lemma 9 of \cite{HLS12}, since every 1--periodic orbit of $sH + F_i$ with non-zero action is non-degenerate.
  \end{proof}
 \end{proof}
\end{proof}

\subsection{Counter-example on the sphere}\label{sec:countr_exmpl}
  In this section, we construct an example showing that the max formula of Theorem \ref{theo:max-formula-aspherical} does not hold for the spectral invariant constructed by Oh, in \cite{Oh05b,Oh06a},  on the sphere. Since neither $\omega$, nor the first Chern class $c_1$, vanish on $\pi_2(\S^2)$ we must adjust our definitions of the action functional, the Conley--Zehnder index and the spectral invariant $c$.
  
 We denote by $\Omega_0$ the space of contractible loops in $\S^2$ and define
 $$\tilde{\Omega}_0 = \frac{ \{ [z,u]: z \in \Omega_0(\S^2) , u: D^2 \rightarrow \S^2 , u|_{\partial D^2} = z \}}{[z,u] = [z', u'] \text { if } z=z' \text{ and } \bar{u} \# u' = 0 \text{ in } \pi_2(\S^2)},$$
   where $\bar{u} \# u'$ is the sphere obtained by gluing $u$, with its orientation reversed, to $u'$ along their common boundary. The disk $u$ in $[z, u]$, is referred to as the capping disk of the orbit $z$.  We define the action functional $\mathcal{A}_H: \tilde{\Omega}_0 \rightarrow \mathbb{R}$, associated to a Hamiltonian $H$, by
   $$\mathcal{A}_H([z,u]) =  \int_{0}^{1} H(t,z(t))dt \text{ }- \int_{D^2} u^*\omega.$$
The set of critical points of $\mathcal{A}_H$ consists of equivalence classes,  $[z,u] \in \tilde{\Omega}_0$, such that $z$ is a 1--periodic orbit of the Hamiltonian flow $\phi^t_H$.   

When $H$ is non-degenerate the set of 1--periodic orbits of $H$ can be indexed by the well known Conley--Zehnder index  $\CZ$. Here, we will recall some facts about $\CZ$ without defining it.  Our convention for normalizing $\CZ$ is as follows: Suppose that $g$ is a $C^2$--small Morse function. We normalize the Conley--Zehnder index so that for every critical point $p$ of $g$, 
$$ \mu_\mathrm{CZ}([p, u_p]) = i_{\text{Morse}}(p),$$
where $i_{\text{Morse}}(p)$ is the Morse index of $p$ and $u_p$ is a trivial capping disk.  For every $A \in  \pi_{2}(\S^2)$, the Conley--Zehnder index satisfies the following identity
\begin{equation}\label{eq:CZ-index identity}
   \CZ([z,u\#A]) = \CZ([z,u]) - 2 c_1(A).
\end{equation}

Floer homology of a non-degenerate Hamiltonian $HF_*(H)$ can be defined as in Section \ref{sec:def_c_aspherical} and it coincides with the quantum homology $QH_*(\S^2).$
The spectral invariant $c(H)$ is once again defined as the action value at which the fundamental class $[\S^2]$ appears in $HF_*(H)$.  We will now list, without proof, those properties of $c: C^{\infty}([0,1] \times \S^2)  \rightarrow \mathbb{R}$ which will be used later on. 
\begin{prop} \label{prop:properties_spec}
    The spectral invariant $c: C^{\infty}([0,1] \times \S^2)  \rightarrow \mathbb{R}$ has the following properties:
    \begin{enumerate}   
    \item \label{monotonicity} (Monotonicity) If $H \leq G$ then $c(H) \leq c(G)$.
    \item \label{continuity}(Continuity)  $ \displaystyle \int_{0}^{1} \min_{x \in M } (H_t-G_t) \, dt \leq c(H) - c(G) \leq \int_{0}^{1} \max_{x \in M } (H_t-G_t) \, dt.$
    \item \label{spectrality} (Spectrality) $c(H) \in \spec(H)$, i.e. $ \exists [z,u]$ such that $z$ is a 1--periodic orbit of $\phi^t_H$ and $c(H) = \mathcal{A}_H([z,u])$. Moreover, if $H$ is non-degenerate then $[z,u]$ can be chosen so that $\CZ([z,u]) = 2$. 
    \item \label{energy-capacity} (Energy-Capacity inequality) Suppose that the support of $H$ is displaced by $\phi^1_K$, i.e. $\phi^1_K(supp(H)) \cap supp(H) = \emptyset.$  Then, $$|c(H)| \leq \int_{0}^{1} \max_{x \in M } K_t - \min_{x \in M} K_t \, dt.$$
    \end{enumerate}   
   \end{prop} 
   For the proofs of the first three of the above properties we refer the reader to \cite{Oh05b, Oh06a}.  The fourth property can be deduced from Proposition 3.1 of  \cite{usher10}.

\subsubsection{The Counter example.}
We will begin with a description of our set up.  We equip $\S^2$ with the standard area form normalized such that the total area of the sphere is $1$.  We let $S, N$ denote the South and the North pole of the sphere and we denote by $z: \S^2 \rightarrow [0,1] $ the standard height function normalized such that 
\begin{itemize}
\item $z(S) = 0$,
\item Area of the disk $\{x \in \S^2: z(x) \leq a\}$ is $a$.
\end{itemize}

We will say that a Hamiltonian $H:\S^2 \rightarrow \R$ is a function of height if there exists a function $h:[0,1] \rightarrow \R$ such that $H = h(z)$.  As in the case of radial Hamiltonians considered in Section \ref{sec:exampl-radi-hamilt},  the 1--periodic orbits of $H$ occur at values of $z$ where $h'(z)$ is an integer. Now, suppose that $h'(z_{\alpha})=1$ and denote by  $\alpha$ the corresponding  1--periodic orbit of $H$.  Let $u_{\alpha} : \D \rightarrow \S^2$ denote the embedded disk which caps $\alpha$ and contains the North pole $N$; it follows from our conventions that  
 $\alpha$ runs from East to West, thus $u_{\alpha}$ has negative area,
 and the  action of the capped orbit $[\alpha, u]$ is given by $\m A_H([\alpha, u_{\alpha}]) = h(z_{\alpha}) + (1 - z_{\alpha}).$

   We will use the following lemma in the construction of our counter example:
\begin{lemma}\label{lem:count_exmpl}
Let $H = h(z)$ be a smooth function of height as described above and  suppose that it has the following properties: 
\begin{enumerate}
\item $h' = 0$ on an interval of the form $[0, \delta]$,
\item $0 < h' <2$ on $(\delta, 1)$,
\item $h' = 1$ at precisely two points  $z_\beta, z_\alpha.$  Suppose that $z_\beta < z_\alpha,$
\item $h'(1)$ is small but non zero.

\end{enumerate}
Then, $c(H) = \min\{h(z_\beta ) + (1 - z_\beta ), h(1) \}.$
\end{lemma}

Postponing the proof of this lemma to the end of this section, we now proceed with the construction of our counter example.  We pick a Hamiltonian $H = h(z)$ as in Lemma \ref{lem:count_exmpl} which satisfies the following additional properties:
\begin{enumerate}
\item $h(0) = - \frac{1}{2}, \, h(\frac{1}{2}) = 0, \, h(1) > \frac{1}{2},$
\item $z_\beta < \frac{1}{2}$ and $h(z_\beta) \approx - \frac{1}{2},$
\item $z_\alpha > \frac{1}{2}.$
\end{enumerate}  

\begin{figure}[h!]
\centering
\def\svgwidth{0.7\textwidth}
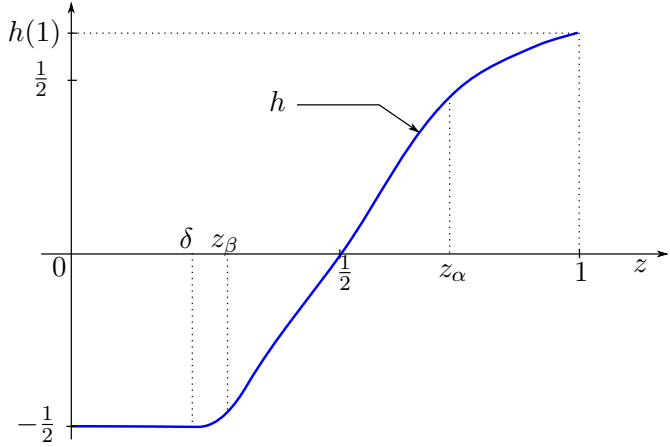
\caption{Graph of $h$.}
\label{fig:h_graph}
\end{figure}

See Figure \ref{fig:h_graph} for a graph of $h$.  The condition that $h(z_\beta) \approx - \frac{1}{2}$ implies that $h(z_\beta ) + (1 - z_\beta ) \approx  \frac{1}{2} - z_\beta <  \frac{1}{2} < h(1)$.  Hence, by Lemma \ref{lem:count_exmpl} we obtain 
\begin{align}\label{eq:c_H} 
c(H) = h(z_\beta ) + (1 - z_\beta ) \approx \frac{1}{2} - z_\beta.
\end{align}

Next, we perturb $H$ by a $C^0$--small amount in the following fashion: we modify $h$ on a small interval of the form $(\frac{1}{2} - 2 \delta',\frac{1}{2} + 2 \delta')$  so that $h$ becomes zero on the subinterval  $(\frac{1}{2} - \delta',\frac{1}{2} + \delta')$.  Call this new function $\tilde{h}$ and the corresponding Hamiltonian $\tilde{H} = \tilde{h}(z)$; see Figure \ref{fig:h_tilde_graph} for a graph of  $\tilde{h}$.  Observe that by picking $\delta'$ to be small enough we can ensure that  $\tilde{H}$ is $C^0$ close to $H$ and so  \begin{align}\label{eq:c_tildeH}
      c(\tilde{H}) \approx \frac{1}{2} - z_\beta.
    \end{align}
 
\begin{figure}[h!]
\centering
\def\svgwidth{0.7\textwidth}
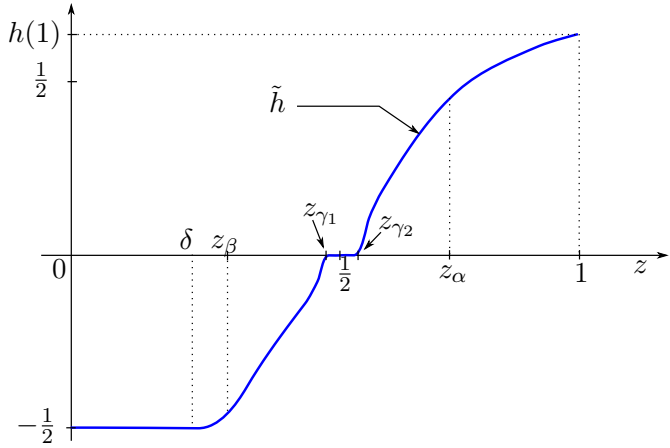
\caption{Graph of $\tilde{h}$.}
\label{fig:h_tilde_graph}
\end{figure}

Now, note that $\tilde{h} = h_1 + h_2$ where $h_1$ is supported in $[0, \frac{1}{2} - \delta']$ and $h_2$ is supported in $[\frac{1}{2} +\delta', 1]$.  Therefore, $\tilde{H} = H_1 + H_2$ where $H_i = h_i(z)$. The Hamiltonian $\tilde{H} = H_1 + H_2$ is precisely of the type appearing in the statement of the max formula in Theorem \ref{theo:max-formula-aspherical}: each of $H_1$ and $H_2$ is supported in a disk.  We will now prove that 
$\max\{c(H_1), c(H_2) \} \approx \frac{1}{2}$ which will imply by, Equation \ref{eq:c_tildeH}, that 

$$c(H_1 + H_2) \neq \max\{c(H_1), c(H_2) \}.$$

 Observe that the process of flattening $h$ to obtain $\tilde{h}= h_1 + h_2$ creates precisely two new points, $z_{\gamma_1} \in  (\frac{1}{2} - 2 \delta',\frac{1}{2} - \delta')$  and $z_{\gamma_2} \in  (\frac{1}{2} + \delta', \frac{1}{2} + 2 \delta')$, such that $h_1'(z_{\gamma_1}) = 1$ and $h_2'(z_{\gamma_2}) = 1$.  This implies that the Hamiltonian $H_2$ satisfies all the conditions of Lemma \ref{lem:count_exmpl} and hence $c(H_2) = \min\{h_2(z_{\gamma_2}) + (1 - z_{\gamma_2}), h_2(1) \}$.  Note that picking $\delta'$ to be sufficiently small forces $z_{\gamma_2}$ to be very close to $\frac{1}{2},$ which in turn forces $h_2(z_{\gamma_2}) + (1 - z_{\gamma_2}) \approx \frac{1}{2}.$ Since, $h_2(1) = h(1) > \frac{1}{2}$, we conclude that $$c(H_2) \approx \frac{1}{2}.$$  The Hamiltonian $H_1$ is negative and so $c(H_1) \leq 0$.  We conclude that $\max\{c(H_1), c(H_2) \} = c(H_2) \approx \frac{1}{2}.$  Next, we  prove Lemma \ref{lem:count_exmpl}.

 \begin{proof}[Proof of Lemma \ref{lem:count_exmpl}]
The spectrality property implies that $c(H+ r) =c(H) + r$, for any constant $r$,  and therefore we may assume without loss of generality that $h=0$ on the interval $[0, \delta]$.  Note that $H$ is supported in the disk $D = \{p \in \S^2: z(p) \geq \delta \}.$
 
 The Hamiltonian $H$ has two families of 1--periodic orbits corresponding to the heights $z_\beta, z_\alpha$.  Let $\beta, \alpha$ denote two 1--periodic orbits corresponding to the heights $z_\beta, z_\alpha$,  respectively.  We will denote by $u_\beta, u_\alpha$ the embedded disks which cap these orbits and contain the north pole $N$.  Let $A$ denote the generator of $\pi_2(\S^2)$ whose area is $1$.
 
 \begin{claim}\label{claim:last}
 For any Hamiltonian $H$ satisfying the conditions of Lemma \ref{lem:count_exmpl} $c(H)$ is attained by the action of one of the four capped 1--periodic orbits $$\{[S, -A], [\beta, u_\beta], [\alpha, u_{\alpha}\# A], N \},$$ 
 where $[S, -A]$ is the South pole with the capping $-A$ and  $N$ is  the North pole with its trivial capping.
 \end{claim} 
We will first explain how Lemma \ref{lem:count_exmpl} follows from the above claim.    Note that the actions of the above four orbits are as follows: $$\m A_H(N) = h(1), \ \ \m A_H([\beta, u_\beta]) = h(z_\beta) + 1 - z_\beta,$$
 $$\m A_H([\alpha, u_\alpha \#A]) = h(z_\alpha) - z_\alpha, \ \ \m A_H([S, -A]) = 1.$$

We will prove Lemma \ref{lem:count_exmpl} in two steps.

\noindent \textbf{Step I:}  We treat the case when  $H$ is supported in the interior of the Northern hemisphere, i.e. $ \delta > \frac{1}{2}.$

The Hamiltonian $H$ is supported in the Northern hemisphere which can be displaced with energy less than $\frac{1}{2}$.   Hence, the Energy-Capacity inequality implies that $c(H) < \frac{1}{2}$:  this rules out the possibility that $c(H) = \m A_H([S, -A]) $.

Next, note that $H \geq 0$ and so, by the continuity property in Proposition \ref{prop:properties_spec}, $c(H) \geq 0$. Since $\delta>\frac{1}{2}$ and $h' <2$ we get $h(z) <z$ for every $z>0$, and in particular $h(z_{\alpha}) - z_{\alpha}$ is negative. We deduce that $c(H) \neq h(z_{\alpha}) - z_{\alpha}$.


We conclude from the previous two paragraphs that $c(H)$ must be attained by the action of $[\beta, u_\beta]$ or $N$ and therefore $c(H) \in \{h(z_\beta ) + (1 - z_\beta ), h(1) \}$.  To prove  Lemma \ref{lem:count_exmpl} it is sufficient to show that $c(H) \leq \min \{h(z_\beta ) + (1 - z_\beta ), h(1) \}$.

The continuity property in Proposition \ref{prop:properties_spec} implies that $c(H) \leq \max(H)$ and since $\max(H) = H(N) = h(1)$ we see that $c(H) \leq h(1)$. 
It remains to prove that $c(H) \leq h(z_\beta) + 1 - z_\beta $.  To that end, we perturb $H$ by a $C^0$--small amount in the following fashion: we flatten $h$ near the point $z_\beta$ so that it becomes constant  on a small interval of the form $(z_\beta - \delta', z_\beta + \delta')$.  Call this new function $\tilde{h}$ and the corresponding Hamiltonian $\tilde{H} = \tilde{h}(z)$.  Observe that by picking $\delta'$ to be small enough we can ensure that  $\tilde{H}$ is $C^0$ close to $H$ and so $c(\tilde{H}) \approx c(H).$

Now, write $\tilde{h} = h_1 + h_2$ where $$h_1 = 
\begin{cases}
 \tilde{h} &\mbox{ on }  [0, z_\beta - \delta'] ,\\
 \tilde{h}(z_\beta - \delta') &\mbox{ on } [ z_\beta - \delta', 1].
\end{cases} $$

Let $H_1 = h_1(z), H_2 = h_2(z)$ denote the corresponding Hamiltonians on the sphere and note that $\tilde{H} = H_1 +H _2$.  It follows from the continuity property in Proposition \ref{prop:properties_spec} that $c(H_1 + H_2) \leq \max(H_1) + c(H_2)$.  Now $\max(H_1) = h_1(z_\beta - \delta') \leq h(z_\beta).$  Also, the Energy-Capacity inequality implies that $c(H_2) \leq (1- z_\beta):$ this is because $H_2$ is supported in the disk $\{p \in \S^2: z(p) \geq 1-z_\beta \}$ which can be displaced with energy $(1-z_\beta)$.   It follows that $$c(\tilde{H}) = c(H_1 + H_2) \leq h(z_\beta) + 1 - z_\beta.$$  Since $\tilde{H}$ can be picked to be arbitrarily close to $H$ we conclude that $c(H) \leq h(z_\beta) + 1 - z_\beta.$   This finishes Step I.
\medskip

\noindent \textbf{Step II:} We treat the case when the support of $H$ is not contained in the Northern hemisphere.  

A series of elementary computations, which we have omitted, reveal that the  restrictions imposed on $h'$ in the statement of Lemma \ref{lem:count_exmpl}, imply the following about the actions of the four orbits of Claim \ref{claim:last}:
\begin{enumerate}
\item $\m A_H([\alpha, u_\alpha \#A]) = h(z_\alpha) - z_\alpha$ is  strictly the smallest value among the actions of these four orbits,
\item $\m A_H([\beta, u_\beta]) = h(z_\beta) + 1 - z_\beta <  \m A_H([S, -A]) = 1.$
\end{enumerate}

These properties hold regardless of whether the support of $H$ is contained in the Northern hemisphere or not.
To complete Step II, we will rely on the symplectic contraction principle of Section \ref{sec:symp-cont}  The support of $H$ is contained in the Liouville domain $D = \{p \in \S^2: z(p) \geq \delta\}$ and hence we can symplectically contract $H$: For each fixed $s \leq 0$ denote by $A_s: D \rightarrow D$ the time $s$ map of the Liouville flow and let $$H_s(x):=  \begin{cases}
e^s H( A_s^{-1}(x)) &\mbox{ if } x \in A_s(D),\\
0 &\mbox{ if } x\notin A_s(D).
\end{cases} $$

We can write $H_s = h_s(z)$ and in fact the Hamiltonian $H_s$ will continue to satisfy the conditions of Lemma \ref{lem:count_exmpl}. Denote by $\beta_s, \alpha_s$ the orbits of $H_s$ corresponding to $\beta, \alpha$. Since $H_s$ satisfies the conditions of Lemma \ref{lem:count_exmpl}, $c(H_s)$ is attained by one of the four orbits  $[S, -A], [\beta_s, u_{\beta_s}], [\alpha_s , u_{\alpha_s} \# A],N$  and the actions of these four orbits will continue to satisfy the following relations:

\begin{enumerate}
\item $\m A_{H_s}([\alpha_s, u_{\alpha_s} \#A]) =  h_s(z_{\alpha_s}) - z_{\alpha_s}$ is the smallest value among the actions of these four orbits,
\item $ \m A_{H_s}([\beta_s, u_{\beta_s}])= h_s(z_{\beta_s}) + 1 - z_{\beta_s} < \m A_{H_s}([S, -A]) = 1.$
\end{enumerate}

Furthermore, as was done in Section \ref{sec:symp-cont}, one can easily see that 

\begin{enumerate}
\item $\m A_{H_s}(N) = h_s(1) = e^s h(1)$,
\item $ \m A_{H_s}([\beta_s, u_{\beta_s}]) = h_s(z_{\beta_s}) + 1 - z_{\beta_s} = e^s(  h(z_{\beta}) + 1 - z_{\beta}),$
\end{enumerate}

Below, we will use the above information to deduce that the actions of these orbits have a very simple bifurcation diagram. We will finish the proof,  using these bifurcation diagrams, by considering two cases.

\noindent \textbf{Case 1:} $\min \{h(z_{\beta}) + 1 - z_{\beta}, h(1) \} = h(1).$
   
    Since the Liouville flow $A_s$ contracts the disk $D$ towards the Northern hemisphere, there exists $s_0$ such that $H_s$ is supported in the Northern hemisphere for every $s \leq s_{0}$, and thus by Step I $$ c(H_s) = \min \{h_s(z_{\beta_s}) + 1 - z_{\beta_s}, h_s(1) \} = h_s(1), \;\; \forall s \leq s_0.$$
    
    The above listed relations among the actions of the four orbits in consideration imply that the curve $s \mapsto h_s(1) = e^sh(1)$ never intersects any of the other three curves in the bifurcation diagram; see Figure \ref{fig:bif-sphere} (In fact, in this case the four curves in the bifurcation diagram are mutually disjoint for all values of $s$.) It follows from the continuity of the spectral invariant $c$ that $c(H) = h(1).$
    
\begin{figure}[h!]
\centering
\def\svgwidth{1\textwidth}
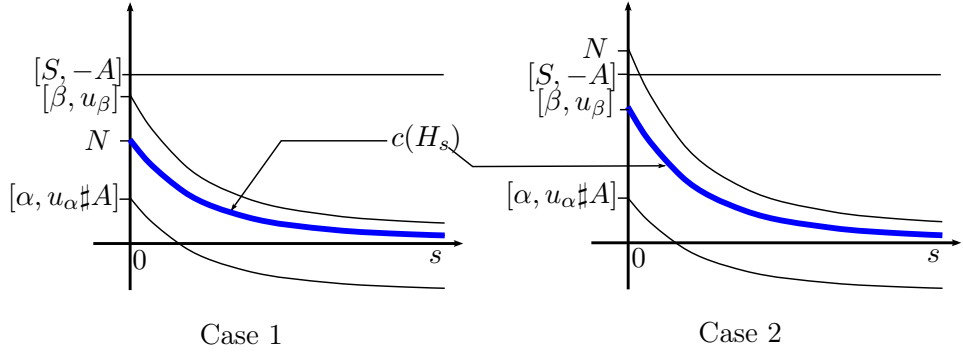
\caption{The bifurcation diagrams of $H_s$. Each curve represents the value of the action $\m A_{H_s}$ corresponding to the indicated point.}
\label{fig:bif-sphere}
\end{figure}

\noindent \textbf{Case 2:} $\min \{h(z_{\beta}) + 1 - z_{\beta}, h(1) \} = h(z_{\beta}) + 1 - z_{\beta} .$

Once again, by Step I, we know that $$ c(H_s) = \min \{h_s(z_{\beta_s}) + 1 - z_{\beta_s}, h_s(1) \} =  h_s(z_{\beta_s}) + 1 - z_{\beta_s}, \;\; \forall s \leq s_0.$$
As in the previous case, the relations among the actions of the four orbits in consideration imply that the curve $s \mapsto h_s(z_{\beta_s}) + 1 - z_{\beta_s} $ never intersects any of the other three curves in the bifurcation diagram; see Figure \ref{fig:bif-sphere}. It follows from the continuity of the spectral invariant $c$ that $c(H) = h(z_{\beta}) + 1 - z_{\beta}.$
\end{proof}
 
  We have now finished the proof of Lemma \ref{lem:count_exmpl} and it remains to prove Claim \ref{claim:last}. The Hamiltonian $H$ has four families of capped 1--periodic orbits.  Below, we examine these families and prove that the only orbits that can carry the action of $c(H)$ are the four listed in Claim \ref{claim:last}. The proof relies on computation of the Conley--Zehnder index and Property 3 in Proposition \ref{prop:properties_spec}.

In Section \ref{sec:rot_number} we define rotation numbers for fixed points on an aspherical surface. When the surface is the sphere, a rotation number $\rho(x,u_{x})$ for a $1$--periodic point $x$ with capping $u_{x}$ may still be defined, and depends on (the homotopy class of) the capping $u_{x}$. Furthermore, if $x$ is a non degenerate fixed point then the relation which was given at the end of Section \ref{sec:rot_number} between the rotation number and the Conley--Zehnder index still holds. We will use this relation to compute the Conley--Zehnder indices. Alternatively, one could use the computations of Section 3.3 of \cite{oancea}.
However, Oancea's conventions are different from ours.  A summary of Oancea's computations is listed in Section 3 of \cite{sey_killer}.

 \noindent \textbf{1. Capped orbits of the form $[N, kA]$:}  Since $h'(1)$ is small and non-zero, the Hamiltonian $H$ is $C^2$--small and Morse near $N$.  It is well-known that this implies $\CZ(N) = 2$ from which we conclude $\CZ([N, kA]) = 2 -2k.$ We see that, of the orbits in this family, the only one with the right Conley--Zehnder index is $N$.

 \noindent \textbf{2.  Capped orbits of the form $[\beta, u_\beta \# k A] $:}  
 We will give the outline of an argument proving that for non-zero values of $k$ the capped  orbit $[\beta, u_\beta \# kA] $ does not have the right Conley--Zehnder index.   The orbit $\beta$ is degenerate.  An appropriate $C^2$--small (time-dependent) perturbation of $H$ near this family of orbits yields two distinct 1--periodic orbits $\beta_1, \beta_2$, with corresponding cappings $u_1, u_2$.  The rotation number $\rho(\beta,u_{\beta})$ is $-1$.  Since $h'$ is increasing at $z_{\beta}$, after perturbation one of the orbit, say $\beta_{1}$, has rotation number 
 $\rho(\beta_{1},u_1) \in (-1,0)$ and the other one is a saddle satisfying  $\rho(\beta_{2},u_{2}) = -1$. From the above paragraph we see that the corresponding Conley--Zehnder indices are respectively $2$ and $3$.
 Using Equation \eqref{eq:CZ-index identity}, we see that of the capped orbits $[\beta_i, u_i \# kA]$ the only one with Conley--Zehnder index $2$ is $[\beta_1, u_1]$.  From this we deduce that the only orbit among the orbits $[\beta, u_\beta \# kA]$ which can attain the value of $c(H)$ is $[\beta, u_\beta]. $

 \noindent \textbf{3. Capped orbits of the form $[\alpha, u_\alpha\# k A] $:}
   The orbit $\alpha$ is degenerate.  An appropriate $C^2$--small(time-dependent) perturbation of $H$ near this family of orbits yields two distinct 1--periodic orbits $\alpha_1, \alpha_2$, with corresponding cappings $u_1, u_2$.  
   Similarly to the previous case we get that one of these two orbits has Conley--Zehnder index  $3$ and the other has Conley--Zehnder index $4$. Suppose that $[\alpha_1, u_1]$ is the one with Conley--Zehnder index $4$.  Then, using Equation \eqref{eq:CZ-index identity}, we see that the only capped orbit with Conley--Zehnder index $2$ is $[\alpha_1, u_1 \# A]$. We conclude that the only orbit among the orbits $[\alpha, u_\alpha \# kA]$ which could possibly attain the value of $c(H)$ is $[\alpha, u_\alpha \#A]. $

 \noindent \textbf{4. Capped orbits of the form $[p, kA]$ where $p$ is a point outside the support of $H$:}  A $C^2$--small perturbation of $H$ would yield a Hamiltonian with a single minimum point at the South pole $S$, whose Conley--Zehnder index with the trivial capping is $0$.  We see that, using Equation \eqref{eq:CZ-index identity}, of these orbits the only one which can carry the action of $c(H)$ is $[S, -A]$.
 
 This finishes the proof of our last claim.
 
\appendix

\section{Existence of maximal unlinked sets}

In this appendix, we prove Theorem~\ref{theo.unlinked-criteria1} and Proposition~\ref{prop.unlinked-criteria2} about unlinked sets. Throughout the appendix we consider  a compactly supported isotopy  $(\phi^t)_{t \in [0,1]}$  on an orientable surface $\Sigma$, and denote its time-one $\phi^1$ by $\phi$.
\subsection{Finite unlinked sets}

This section contains the proof of Proposition~\ref{prop.unlinked-criteria2} which characterizes finite unlinked sets as finite sets whose associated geometric braid represents the trivial braid. The crucial point in the proof is the following Lemma.
\begin{lemma}\label{lemma.fibration}
Let $X$ be a finite subset of $\Sigma$ of cardinality $n$. Then  the map $f \mapsto f(X)$ from $\mathrm{Diff}_{0}(\Sigma)$ to the space $\Xi(\Sigma,n)$ of $n$-tuples of distinct points in $\Sigma$, is a fiber bundle.
\end{lemma}
\begin{proof}
Using local charts the proof of the lemma reduces to the following easy fact. Let $\D^2$ denote the open unit disk, and $0$ be some point in $\D^2$. There exists a continuous map $x \mapsto \gamma_{x}$ from  $\D^2$ to the space $\mathrm{Diff}(\D^2)$ of diffeomorphisms of $\D^2$ with compact support, such that 
$\gamma_{0}$ is the identity, and for every point $x$ of $\D^2$, $\gamma_{x} (0) = x$.
There are many ways to construct $\gamma$, one possibility is to use Hamiltonian functions: then $\gamma_{x}$ is a Hamiltonian diffeomorphism of the disk (in particular, the lemma also holds when $\mathrm{Diff}_{0}(\Sigma)$ is replaced by $\mathrm{Ham}(\Sigma)$).
\end{proof}

\begin{proof}[Proof of Proposition~\ref{prop.unlinked-criteria2}]
The direct implication is straightforward.
For the reverse one,  consider a finite set $X$ of contractible fixed points for $(\phi^t)_{t \in [0,1]}$.
The geometric braid $b_{X,(\phi^t)}$ is a loop based at $X$ in the space $\Xi(\Sigma,n)$. Assume that $b_{X,(\phi^t)}$ represents the trivial braid. This means that it is a loop homotopic to the constant loop. A fiber bundle is a Serre fibration, that is, it has the ``homotopy lifting property'' for disks. Thus, according to the lemma, the homotopy from $b_{X,(\phi^{t})}$ to the constant loop may be lifted to a homotopy, with end-points fixed, between the isotopy $(\phi^t)_{t \in [0,1]}$ and an isotopy $I$ which lifts the constant loop in $\Xi(\Sigma,n)$, \emph{i.e.} which fixes every point of $X$. In other words, the set $X$ is unlinked.
\end{proof}

An unlinked set $X$ is \emph{maximal} if there is no unlinked set $X'$ strictly containing $X$.
\begin{corol}
 Let $X$ be an unlinked set, and $I$ an isotopy that fixes every point of $X$ and whose time one is $\phi$. 
If $\Sigma$ is the sphere, assume furthermore that $X$ does not contain exactly two elements.
 Then  $X$ is maximal if and only if  for every fixed point $x$ of $\phi$ which is not in $X$,
the trajectory of $x$ under $I$  is not contractible in $\Sigma \setminus X$.
\end{corol}
\begin{proof}
The direct implication may be proved by an argument similar to the proof of Proposition \ref{prop.unlinked-criteria2}.
The proof of the converse goes as follows. The case when $X$ is empty directly follows from the proposition.
Let $X$ be a non empty unlinked set, $I$ be an isotopy fixing every point of $X$, $x$ a point outside $X$, and $J$ an isotopy fixing $X \cup \{x\}$ and such that $I$ and $J$ are homotopic as paths in $\mathrm{Diff}_{0}(\Sigma)$. We want to prove that the trajectory  of $x$ under $I$ is contractible in $\Sigma \setminus X$. Let $\alpha$ be the class of this trajectory in $\pi_{1}( \Sigma \setminus X, x)$. Then $\alpha$ commutes with every element $\beta$ in $\pi_{1}( \Sigma \setminus X,x)$: indeed, the map $(s,t) \mapsto f_{t}(b(s))$, where $(f_{t})$ is the concatenation of $I$ with $J^{-1}$ and $b$ is a loop in the class $\beta$, may be seen as a homotopy between $\alpha\beta\alpha^{-1}\beta^{-1}$ and the trivial loop. 
We conclude that $\alpha$ is the trivial loop when the center of $\pi_{1}( \Sigma \setminus X,x)$ is trivial. Since $X$ is non empty and we have excluded the case when $\Sigma$ is the sphere and $X$ contains exactly two elements, this covers every case except when $\Sigma$ is the plane and $X$ is a single element. This last case may be solved by using the following fact: the space of compactly supported diffeomorphisms of the plane fixing a given point is contractible.
\end{proof}

\subsection{Infinite unlinked sets}
This section contains the proof of Theorem~\ref{theo.unlinked-criteria1}: a set $X$ of contractible fixed points of $(\phi^t)_{t \in [0,1]}$ is unlinked if and only if every finite subset of $X$ is unlinked. Note that the direct implication is immediate. For the converse, the key to the proof will be an argument, due to Michael Handel, showing that any surface diffeomorphism is isotopic to the identity in some neighborhood of its fixed point set.

Let $A(\phi)$ denote the set of accumulation points of the set of fixed points of $\phi$. 
If $A(\phi)=\Sigma$ the theorem is obvious, thus we may assume $A(\phi) \neq \Sigma$. We consider an open neighborhood $V$ of $A(\phi)$ which is not $\Sigma$. The surface $V$, being non-compact, may be endowed with a flat Riemannian metric, i.e. a metric which is locally isometric to the euclidean plane.
Let $\varepsilon>0$ be such that for every points $x,y$ on $\Sigma$ such that $d(x,y) < \varepsilon$, there exists a unique geodesic segment of length $d(x,y)$ joining $x$ to $y$. We denote by $\gamma_{x,y} : [0,1] \to \Sigma$ the parametrization of this segment with constant speed.
\begin{lemma}[Handel, Lemma 4.1 in \cite{handel1992}]
For every open set $V$ containing $A(\phi)$, there exists some open set $V'$, $A(\phi) \subset V' \subset V$, such that the "straight" homotopy
$$
f_t (x)= \gamma_{x,\phi(x)} (t) 
$$
is one-to-one on $V' \cup \mathrm{Fix}(\phi)$.
\end{lemma}
\begin{proof}
In local charts isometric to the plane, $\gamma_{x,\phi(x)} (t) $ reads $(1-t)x + t \phi(x)$.
The differential of $f_t$ at the point $x$ is 
$(1-t)\mathrm{Id} + tD\phi(x)$. At every point of $A(\phi)$ the differential  of $\phi$  has a fixed vector, and thus (since $\phi$ is orientation preserving) no negative eigenvalue. If $V'$ is close enough to $A(\phi)$ then $D\phi$ has no negative eigenvalue on $V$' either. The inverse function theorem implies that $\phi$ is one-to one on $V'$. For more details we refer to~\cite{handel1992}.
\end{proof}
Note that every germ of an orientation preserving diffeomorphism at some fixed point $x$ is locally isotopic to the identity (up to composing with a rotation $D\phi_{x}$ has no negative eigenvalue, and then one can use again the straight line isotopy). By extension of isotopies (see for example~\cite{hirsch}, Chapter 8), 
we get the following corollary.
\begin{corol}\label{coro.locally-isotopic-identity}
 There exists an isotopy $\Delta = (f_{t})_{t \in [0,1]}$ from $f_{0} = \phi$ to a diffeomorphism $f_1$, and an open  neighborhood $U$ of the set of contractible fixed points  of $(\phi^{t})_{t \in [0,1]}$, such that
 \begin{itemize}
 \item the isotopy $\Delta$ fixes every contractible fixed point of $(\phi^{t})_{t \in [0,1]}$,
 \item $f_1$ is the identity on $U$.
 \end{itemize}
\end{corol}
Let  $I$ denote  the isotopy which is the concatenation of   $(\phi^{t})_{t \in [0,1]}$ and $\Delta$, whose time one is $f_{1}$.
Note that a family $X$ of contractible fixed points of $(\phi^{t})_{t \in [0,1]}$ is unlinked if and only if it is unlinked for $I$. 
We may assume that the closure of $U$ is a compact subsurface of $\Sigma$ (with boundary). We denote by $\{U_{i}: i \in \pi_{0}(U)\}$ the connected components of $U$; the set $\pi_{0}(U)$ is finite, and we may assume that each $U_{i}$ contains some contractible fixed point of $(\phi^{t})_{t \in [0,1]}$; then every point in $U_{i}$ is a contractible fixed point for the isotopy $I$. Also note that since $f_{1}$ is the identity on $U_{i}$, all the points in $U_{i}$ have the same rotation number $\rho_{i}$  for  $I$, which is an integer. 
Let $x,y$ be two distinct points in the same $U_{i}$, and join them by an arc $\gamma$. In the universal cover $\tilde \Sigma \simeq \R^2$, let $\tilde x, \tilde y$ be two lifts of $x,y$ which are the endpoints of some lift $\tilde \gamma$ of $\gamma$. Then the linking number $\ell(\tilde x, \tilde y)$ for the isotopy that lifts $I$ is also equal to $\rho_{i}$. 
As a consequence, if $\rho_{i}$ is not zero, then every pair $\{x,y\}$ of contractible fixed points of $(\phi^{t})_{t \in [0,1]}$ included in some $U_{i}$ is linked. 

\bigskip

Let $X$ be a family of contractible fixed points of $(\phi^{t})_{t \in [0,1]}$. We choose a subset of $X$ by selecting at most two points of $X$ in each $U_{i}$. More precisely, we denote by $X_{U}$ any subset of $X$ with the following property: for every $i \in \pi_{0}(U)$,
\begin{itemize}
\item if the set $X \cap U_{i}$ has zero or one element, then $X_{U} \cap U_{i} = X \cap U_{i}$;
\item if the set $X \cap U_{i}$ contains more than one element, then $X_{U} \cap U_{i}$ has exactly two elements.
\end{itemize}

Note that the set $X_{U}$ is finite. Thus Theorem~\ref{theo.unlinked-criteria1} is an immediate consequence of the following lemma.
\begin{lemma}~
The set $X$ is unlinked for $(\phi^{t})_{t \in [0,1]}$ if and only if the set $X_{U}$ is unlinked for $(\phi^{t})_{t \in [0,1]}$.
\end{lemma}


Before proving the lemma, we will deduce the existence of maximal unlinked sets. The following corollary is not used in this text.
\begin{corol}\label{coro.maximal-unlinked-sets}
Every unlinked set is included in a maximal unlinked set.
\end{corol}
\begin{proof}[Proof of the corollary]
The corollary follows from Theorem~\ref{theo.unlinked-criteria1} by using Zorn's Lemma, as in the proof of Corollary~\ref{coro.existence-mnus}. Here is a more constructive argument.  Let $X$ be an unlinked set, and $X_{U}$ be as before. Note that $X_U$ is unlinked since it is a subset of $X$. Let $Y$ be a set which contains $X_{U}$, which is unlinked, which contains at most two points in each $U_{i} \in \pi_{0}(U)$, and which is maximal for inclusion among such sets. The cardinality of such a set is clearly less than twice the cardinality of $\pi_{0}(U)$ which is finite: thus the existence of $Y$ is immediate. Let $X'$ be obtained from $Y$ by adding to it, for every $U_{i} \in \pi_{0}(U)$ which contains two points of $Y$, all the contractible fixed points of $(\phi^{t})_{t \in [0,1]}$ which are included in $U_{i}$. Note that $Y$ satisfies the required properties for the set $X'_{U}$. Thus according to the above lemma, since $X'_{U}= Y$ is unlinked, the set $X'$ is unlinked. If $x$ is a contractible fixed point which is not in $X'$, then it follows from the lemma and the maximality of $Y$ that $X'\cup\{x\}$ is not unlinked. Thus $X'$ is a maximally unlinked set containing $X$, as wanted.
\end{proof}

\begin{proof}[Proof of the lemma]
The direct implication is immediate. Let $X,X_{U}$ be as before and assume that $X_{U}$ is unlinked for 
$(\phi^{t})_{t \in [0,1]}$. Then it is also unlinked for $I$.
Let $J_{0}=(g_t)_{t \in [0,1]}$ be an isotopy from the identity to $f_1$ fixing every point of $X_{U}$. We want to modify $J_{0}$ to an isotopy from the identity to $f_{1}$ that fixes every point of $X$. This will prove that $X$ is unlinked for $I$, and thus also for $(\phi^{t})_{t \in [0,1]}$.

Call \emph{arotational} the values of $i$ for which $U_{i}$ contains two elements of $X_{U}$. Since $X_{U}$ is unlinked, by the considerations following Corollary~\ref{coro.locally-isotopic-identity},  the rotation number $\rho_{i}$ vanishes for arotational $i$'s. 
Using this property, it is not difficult to modify the isotopy $J_{0}$ to an isotopy $J_{1} = (h_{t})_{t \in [0,1]}$ which still fixes every point of $X_{U}$, and such that, for the arotational indices $i$, the differential $Dh_t(x)$  is the identity for every $t$ at both points $x$ of $X_{U} \cap U_i$.

We want to further modify $J_{1}$ so that it fixes every point in the arotational $U_i$'s. 
For this we will use the following basic results on embeddings, which are proved below.  Let $S$ be a surface with boundary, $S'$ a connected subsurface in the interior of $S$, and $x_{0}$ a point in the interior of $S'$.
We denote by  $\mathrm{Diff}_{c}(S;x_{0})$  the space of diffeomorphisms of $S$ which are compactly supported in the interior of $S$, that fixes $x_{0}$, and whose differential at $x_{0}$ is the identity. Likewise, let 
 $E(S',S;x_{0})$ be the space of embeddings of $S'$ into the interior of $S$ that fixes $x_{0}$ and whose differential at $x_{0}$ is the identity.
\begin{prop}\label{prop.fibration}~

 \begin{enumerate}
 \item Every connected component of $E(S',S;x_{0})$ is simply connected.
 \item  The restriction map from $\mathrm{Diff}_{c}(S;x_{0})$ to $E(S',S;x_{0})$ is a fiber bundle.
 \end{enumerate}
\end{prop}

Let $i_{1}$ be some arotational index, and choose some $x_{1}$ in $X_{U_{i_{1}}}$. Since $h_{1} = f_{1}$ fixes every point in $U_{i_{1}}$, the family $\ell = (h_{t \mid U_{i_{1}}})_{t \in [0,1]}$ is a loop in  $E(U_{i_{1}},\Sigma;x_{1})$ based at the inclusion map $e : U_{i_{1}} \subset \Sigma$.
The first assertion of the proposition tells us that this loop $\ell$ is contractible in  $E(U_{i_{1}},\Sigma;x_{1})$: let $(\ell_{s})_{s \in [0,1]}$ be a deformation of loops with fixed base-point $e$ from $\ell_{1} = \ell$ to the trivial loop.
According to second assertion of the proposition, this deformation may be lifted to a deformation, with fixed end-points $\mathrm{Id}$ and $h_{1}$, from the isotopy $J_{1}=(h_t)$ to a new isotopy $J_{2}$ which is a lift of the trivial loop in $E(U_{i_{1}},\Sigma;x_{1})$, which means that $J_{2}$ fixes every point of $U_{i_{1}}$. We now consider a second arotational index $i_{2}$, and apply the proposition with $S = \Sigma \setminus U_{i_{1}}$, and $S'$ equal to the closure of $U_{i_{2}}$. This yields a new isotopy $J_{3}$ that fixes every point of $U_{i_{1}} \cup U_{i_{2}}$. We go on until we get an isotopy $J$ from $\mathrm{Id}$ to $f_{1}$ which fixes every point of every arotational  $U_{i}$. This isotopy fixes every point of $X$, as wanted.
\end{proof}

\begin{proof}[Proof of Proposition~\ref{prop.fibration}]
We begin with the second assertion. When no base point is given, the fact that the restriction map is a fiber bundle is due to Palais.
The base point case that we need follows immediately from the following result, due to Cerf: \emph{For every  embedding $f$ from some compact manifold $V$ into some manifold $M$, there is a neighborhood $U$ of $f$ in the space of embeddings, and a continuous map $\xi$ from $U$ to the space of diffeomorphisms of $M$ with compact support, such that for every $g$ in $U$, $g = \xi(g) \circ f$}. For references and details we refer to the very short paper of Lima (\cite{lima1963}). 

Now to prove the first assertion,  let $\gamma$ be a loop in the space $E(S',S;x_{0})$. We deform $\gamma$ into a trivial loop by successively using  the following three ingredients. Details are left to the reader (again, \cite{gramain1973} is a good reference).
The first ingredient allows us to deform $\gamma$ into a loop $\gamma_{1}$ that fixes one vector tangent to each boundary component of $S'$. The second ingredient allows us to further deform $\gamma_{1}$ into a loop $\gamma_{2}$ that fixes each point of the boundary of $S'$. The last ingredient shows that $\gamma_{2}$ is contractible in $E(S',S;x_{0})$.

\bigskip

\noindent \emph{Ingredient 1.} Let $x_{1}, \dots , x_{k}$ be distinct points of $S'$, distinct from $x_{0}$, and choose for each $i$ a non zero vector $v_{i}$ tangent to $S'$ at $x_{i}$ (if the points are on the boundary of $S'$ then the vectors are tangent to the boundary). Consider the natural map $\Psi_{(x_{i}, v_{i})}$  from $E(S',S;x_{0})$ to the space $\Xi_{k}$ of $k$-tuples of non zero vectors over distinct points in $S' \setminus \{x_{0} \}$, obtained by taking the images of the $(x_{i},v_{i})$'s. This map is a fiber bundle. Furthermore, the image of every loop $\gamma = (f_{t})$  in $E(S',S;x_{0})$ is contractible in $\Xi_{k}$. 
 
The key observation for this last property is the following. By definition we have $f_{t}(x_{0}) = x_{0}$ and $Df_{t}(x_{0}) = \mathrm{Id}$ for every $t$. Thus if $x'_{1}$ is some point  close enough to $x_{0}$ then the loop $t \mapsto f_{t}(x'_{1})$ will be included in a small disk $D_{1}$ not containing $x_{0}$, and furthermore for any non zero vector $v'_{1}$ at $x'_{1}$ the loop $\Psi_{(x'_{1}, v'_{1})}\gamma :  t \mapsto Df_{t}(x'_{1})\cdot v'_{1}$ will be close to the constant vector, and thus contractible in the complement of the zero section in the tangent bundle of $D_{1}$. To make use of this observation, we move the points $x_{1}, \dots , x_{k}$ into points $x'_{1}, \dots, x'_{k}$ with $x'_{1}$ close to $x_{0}$ as above, $x'_{2}$ much closer to $x_{0}$ than $x'_{1}$ so that it is included in a disk $D_{2}$ disjoint from $D_{1}$ and $x_{0}$, and so on. These moves induce a deformation of the loop $\Psi_{(x_{i}, v_{i})} \gamma$ into the loop $\Psi_{(x'_{i}, v'_{i})} \gamma$. According to the observation, this new loop is contractible.

\bigskip

\noindent \emph{Ingredient 2.} Choose one point $x_{i}$ on each boundary component of $S'$, and a vector $v_{i}$ tangent to $\partial S'$ at $x_{i}$. Let $E(S',S ; x_{0}, (x_{1}, v_{1}), \dots ,(x_{k}, v_{k}))$ denote the subspace of $E(S',S; x_{0})$ that fixes all the $x_{i}$'s and $v_{i}$'s. 

With obvious notations, the restriction map 
$$
E(S',S ; x_{0}, (x_{1}, v_{1}), \dots , (x_{k}, v_{k})) \longmapsto E(\partial S',S ; x_{0}, (x_{1}, v_{1}), \dots , (x_{k}, v_{k}))
$$
 is a  fiber bundle. 
 Furthermore, each connected component of the base of the fibration is contractible. This last property follows from Th\'eor\`eme~4 of~\cite{gramain1973}, which consider the case of the embedding of a single circle in $S$, by induction on the number of boundary components.

\bigskip

\noindent \emph{Ingredient 3.} Every connected component of the space $\mathrm{Diff}(S';\partial S',x_{0})$ of diffeomorphisms of $S'$ that are the identity on the boundary and tangent to the identity at $x_{0}$ has trivial homotopy groups. Indeed, this space is the fiber of the restriction map from  $\mathrm{Diff}(S';x_{0})$ to $\mathrm{Diff}(\partial S')$. This is a fiber bundle. 
 On the one hand, the space of orientation preserving diffeomorphisms of the circle has the homotopy type of $SO(1)$, and thus the connected component of the identity in the base of the fibration has trivial homotopy groups of order $\geq 2$. On the other hand, by Th\'eor\`eme~2 of~\cite{gramain1973}, the total space of the fibration is contractible. The triviality of the homotopy group of the fiber is now a consequence of the exact sequence of the fibration.
\end{proof}

%
%

\bibliographystyle{abbrv}
\bibliography{biblio}

\begin{thebibliography}{10}

\bibitem{audin-damian}
M.~Audin and M.~Damian.
\newblock {\em Morse theory and {F}loer homology}.
\newblock Universitext. Springer, London; EDP Sciences, Les Ulis, 2014.
\newblock Translated from the 2010 French original by Reinie Ern{\'e}.

\bibitem{bramham2}
B.~Bramham.
\newblock Periodic approximations of irrational pseudo-rotations using
  pseudoholomorphic curves.
\newblock {\em Ann. of Math. (to appear)}, ArXiv:1204.4694.

\bibitem{bramham1}
B.~Bramham.
\newblock Pseudo-rotations with sufficiently liouvillean rotation number are
  ${C}^0$-rigid.
\newblock {\em Invent. Math. (to appear)}, ArXiv:1205.6243.

\bibitem{brunella}
M.~Brunella.
\newblock On a theorem of {S}ikorav.
\newblock {\em Enseign. Math. (2)}, 37(1-2):83--87, 1991.

\bibitem{dore-hanlon}
D.~N. Dore and A.~D. Hanlon.
\newblock Area preserving maps on ${S}^2$: A lower bound on the ${C}^0$-norm
  using symplectic spectral invariants.
\newblock {\em Electron. Res. Announc. Math. Sci.}, 20:97--102, 2013.

\bibitem{entov}
M.~Entov.
\newblock Quasi-morphisms and quasi-states in symplectic topology.
\newblock {\em Proceedings of the International Congress of Mathematicians
  (Seoul, 2014)}, to appear, arXiv:1404.6408.

\bibitem{entov-polterovich03}
M.~Entov and L.~Polterovich.
\newblock Calabi quasimorphism and quantum homology.
\newblock {\em Int. Math. Res. Not.}, (30):1635--1676, 2003.

\bibitem{entov-polterovich06}
M.~Entov and L.~Polterovich.
\newblock Quasi-states and symplectic intersections.
\newblock {\em Comment. Math. Helv.}, 81(1):75--99, 2006.

\bibitem{entov-polterovich09}
M.~Entov and L.~Polterovich.
\newblock Rigid subsets of symplectic manifolds.
\newblock {\em Compos. Math.}, 145(3):773--826, 2009.

\bibitem{Floer88}
A.~Floer.
\newblock The unregularized gradient flow of the symplectic action.
\newblock {\em Comm. Pure Appl. Math.}, 41:775--813, 1988.

\bibitem{FHS}
A.~Floer, H.~Hofer, and D.~Salamon.
\newblock Transversality in elliptic {M}orse theory for the symplectic action.
\newblock {\em Duke Math. J.}, 80(1):251--292, 1995.

\bibitem{FrSc}
U.~Frauenfelder and F.~Schlenk.
\newblock Hamiltonian dynamics on convex symplectic manifolds.
\newblock {\em Israel J. Math.}, 159:1--56, 2007.

\bibitem{ginzburg}
V.~L. Ginzburg.
\newblock The {C}onley conjecture.
\newblock {\em Ann. of Math. (2)}, 172(2):1127--1180, 2010.

\bibitem{GiGu10}
V.~L. Ginzburg and B.~Z. G{\"u}rel.
\newblock Local {F}loer homology and the action gap.
\newblock {\em J. Symplectic Geom.}, 8(3):323--357, 2010.

\bibitem{gramain1973}
A.~Gramain.
\newblock Le type d'homotopie du groupe des diff{\'e}omorphismes d'une surface
  compacte.
\newblock {\em Annales scientifiques de l'{\'E}cole Normale Sup{\'e}rieure},
  6(1):53--66, 1973.

\bibitem{handel1992}
M.~Handel.
\newblock Commuting homeomorphisms of ${S}^2$.
\newblock {\em Topology}, 31(2):293--303, 1992.

\bibitem{hein}
D.~Hein.
\newblock The {C}onley conjecture for irrational symplectic manifolds.
\newblock {\em J. Symplectic Geom.}, 10(2):183--202, 2012.

\bibitem{hirsch}
M.~W. Hirsch.
\newblock {\em Differential topology}, volume~33 of {\em Graduate Texts in
  Mathematics}.
\newblock Springer-Verlag, New York, 1994.
\newblock Corrected reprint of the 1976 original.

\bibitem{HLS12}
V.~{Humili{\`e}re}, R.~{Leclercq}, and S.~{Seyfaddini}.
\newblock {New energy-capacity-type inequalities and uniqueness of continuous
  Hamiltonians}.
\newblock {\em Comment. Math. Helv.}, to appear (arXiv:1209.2134).

\bibitem{Jaulent:2012aa}
O.~Jaulent.
\newblock Existence d'un feuilletage positivement transverse {\`a} un
  hom{\'e}omorphisme de surface.
\newblock {\em Annales de l'institut Fourier}, 64(4):1441--1476, 2014.

\bibitem{katok_hasselblatt}
A.~Katok and B.~Hasselblatt.
\newblock {\em Introduction to the modern theory of dynamical systems},
  volume~54 of {\em Encyclopedia of Mathematics and its Applications}.
\newblock Cambridge University Press, Cambridge, 1995.
\newblock With a supplementary chapter by Katok and Leonardo Mendoza.

\bibitem{kawasaki}
M.~Kawasaki.
\newblock Superheavy lagrangian immersion in 2--torus.
\newblock {\em arXiv:1412.4495}, 2014.

\bibitem{lanzat}
S.~Lanzat.
\newblock Quasi-morphisms and symplectic quasi-states for convex symplectic
  manifolds.
\newblock {\em Int. Math. Res. Not. IMRN}, (23):5321--5365, 2013.

\bibitem{Lan}
S.~Lanzat.
\newblock Quasi-morphisms and symplectic quasi-states for convex symplectic
  manifolds.
\newblock {\em Int. Math. Res. Not. IMRN}, (23):5321--5365, 2013.

\bibitem{PLC1}
P.~Le~Calvez.
\newblock Propri\'et\'es dynamiques des diff\'eomorphismes de l'anneau et du
  tore.
\newblock {\em Ast\'erisque}, (204):131, 1991.

\bibitem{lecalvez2005}
P.~Le~Calvez.
\newblock Une version feuillet{\'e}e {\'e}quivariante du th{\'e}oreme de
  translation de brouwer.
\newblock {\em Publications Math{\'e}matiques de l'Institut des Hautes
  {\'E}tudes Scientifiques}, 102(1):1--98, 2005.

\bibitem{PLC3}
P.~Le~Calvez.
\newblock A finite dimensional approach to {B}ramham's approximation theorem.
\newblock {\em arXiv:1307.5278}, 2013.

\bibitem{lima1963}
E.~L. Lima.
\newblock On the local triviality of the restriction map for embeddings.
\newblock {\em Commentarii Mathematici Helvetici}, 38(1):163--164, 1963.

\bibitem{mcduff-salamon}
D.~McDuff and D.~Salamon.
\newblock {\em {$J$}--holomorphic curves and symplectic topology}, volume~52 of
  {\em American Mathematical Society Colloquium Publications}.
\newblock American Mathematical Society, Providence, RI, 2004.

\bibitem{oancea}
A.~Oancea.
\newblock A survey of {F}loer homology for manifolds with contact type boundary
  or symplectic homology.
\newblock In {\em Symplectic geometry and {F}loer homology. {A} survey of the
  {F}loer homology for manifolds with contact type boundary or symplectic
  homology}, volume~7 of {\em Ensaios Mat.}, pages 51--91. Soc. Brasil. Mat.,
  Rio de Janeiro, 2004.

\bibitem{Oh05b}
Y.-G. Oh.
\newblock Construction of spectral invariants of hamiltonian paths on closed
  symplectic manifolds.
\newblock {\em The breadth of symplectic and Poisson geometry. Progr. Math.
  \textbf{232}, Birkhauser, Boston}, pages 525--570, 2005.

\bibitem{Oh06a}
Y.-G. Oh.
\newblock Lectures on floer theory and spectral invariants of hamiltonian
  flows.
\newblock {\em Morse-theoretic methods in nonlinear analysis and in symplectic
  topology. NATO Sci. Ser. II Math. Phys. Chem., 217, Springer, Dordrecht,
  321-416}, 2006.

\bibitem{PSS}
S.~Piunikhin, D.~Salamon, and M.~Schwarz.
\newblock Symplectic {F}loer-{D}onaldson theory and quantum cohomology.
\newblock In {\em Contact and symplectic geometry ({C}ambridge, 1994)},
  volume~8 of {\em Publ. Newton Inst.}, pages 171--200. Cambridge Univ. Press,
  Cambridge, 1996.

\bibitem{Pol_book}
L.~Polterovich.
\newblock {\em The geometry of the group of symplectic diffeomorphisms}.
\newblock Lectures in Mathematics ETH Z\"urich. Birkh\"auser Verlag, Basel,
  2001.

\bibitem{polterovich}
L.~Polterovich.
\newblock Symplectic geometry of quantum noise.
\newblock {\em Comm. Math. Phys.}, 327(2):481--519, 2014.

\bibitem{PS00}
L.~Polterovich and K.~F. Siburg.
\newblock On the asymptotic geometry of area-preserving maps.
\newblock {\em Math. Res. Letters}, (7):233--243, 2000.

\bibitem{Py05}
P.~Py.
\newblock Quasi-morphisme de {C}alabi sur les surfaces de genre sup\'erieur.
\newblock {\em C. R. Math. Acad. Sci. Paris}, 341(1):29--34, 2005.

\bibitem{Py06}
P.~Py.
\newblock Quasi-morphismes de {C}alabi et graphe de {R}eeb sur le tore.
\newblock {\em C. R. Math. Acad. Sci. Paris}, 343(5):323--328, 2006.

\bibitem{Rosen}
M.~Rosenberg.
\newblock Py-{C}alabi quasi-morphisms and quasi-states on orientable surfaces
  of higher genus.
\newblock {\em Israel J. Math.}, 180:163--188, 2010.

\bibitem{salamon_lec}
D.~Salamon.
\newblock Lectures on {F}loer homology.
\newblock In {\em Symplectic geometry and topology ({P}ark {C}ity, {UT},
  1997)}, volume~7 of {\em IAS/Park City Math. Ser.}, pages 143--229. Amer.
  Math. Soc., Providence, RI, 1999.

\bibitem{SZ92}
D.~Salamon and E.~Zehnder.
\newblock Morse theory for periodic solutions of {H}amiltonian systems and the
  {M}aslov index.
\newblock {\em Comm. Pure Appl. Math.}, 45(10):1303--1360, 1992.

\bibitem{schwarz}
M.~Schwarz.
\newblock On the action spectrum for closed symplectically aspherical
  manifolds.
\newblock {\em Pacific J. Math.}, 193(2):419--461, 2000.

\bibitem{sey}
S.~Seyfaddini.
\newblock The displaced disks problem via symplectic topology.
\newblock {\em C. R. Math. Acad. Sci. Paris}, 351(21-22):841--843, 2013.

\bibitem{sey_killer}
S.~Seyfaddini.
\newblock Spectral killers and poisson bracket invariants.
\newblock {\em J. Mod. Dyn.}, 8(3), 2014 (to appear).

\bibitem{sikorav}
J.-C. Sikorav.
\newblock Sur les immersions lagrangiennes dans un fibr{\'e} cotangent
  admettant une phase g{\'e}n{\'e}ratrice globale.
\newblock {\em C. R. Acad. Sci. Paris S\'er. I Math.}, 302(3):119--122, 1986.

\bibitem{theret}
D.~Th{\'e}ret.
\newblock A complete proof of {V}iterbo's uniqueness theorem on generating
  functions.
\newblock {\em Topology and its Applications}, 96(3):246--266, 1999.

\bibitem{usher}
M.~Usher.
\newblock Floer homology in disk bundles and symplectically twisted geodesic
  flows.
\newblock {\em J. Mod. Dyn.}, 3(1):61--101, 2009.

\bibitem{usher10}
M.~Usher.
\newblock The sharp energy-capacity inequality.
\newblock {\em Commun. Contemp. Math.}, 12(3):457--473, 2010.

\bibitem{viterbo}
C.~Viterbo.
\newblock Symplectic topology as the geometry of generating functions.
\newblock {\em Math. Annalen}, 292:685--710, 1992.

\bibitem{Zap}
F.~Zapolsky.
\newblock Reeb graph and quasi-states on the two-dimensional torus.
\newblock {\em Israel J. Math.}, 188:111--121, 2012.

\end{thebibliography}

\end{document}